%% file: main.tex
\documentclass[11pt,reqno]{amsart}
\usepackage{amsfonts,amsmath,amssymb,amsbsy,amstext,amsthm, palatino, euler,
accents,color,enumerate,float,verbatim, mathbbol, bbm,
graphicx,
epstopdf,
mathtools,subcaption,url,microtype}
\usepackage[margin=1in,headheight=13.6pt]{geometry}
\usepackage{comment}
\usepackage[shortlabels]{enumitem}
\usepackage[ruled,vlined]{algorithm2e}

\usepackage[pdfencoding=auto, psdextra]{hyperref}
\usepackage[alphabetic,lite,backrefs]{amsrefs}

\usepackage{tikz,tikz-cd}
\usetikzlibrary{positioning, matrix, shapes}
\usetikzlibrary{decorations.pathmorphing}

\graphicspath{{ALCO/}}

\makeatletter
\@namedef{subjclassname@2020}{%
  \textup{2020} Mathematics Subject Classification}
\makeatother

\input macro.tex

\begin{document}

% Two authors
\title[Root polytopes, tropical types, and toric edge ideals]{Root polytopes, tropical types, and toric edge ideals}

% Tropical Hyperplane Arrangements with Arbitrary Support
\author{Ayah Almousa}
\address{University of South Carolina}
\email{aalmousa@sc.edu}
\urladdr{\url{https://sites.google.com/view/ayah-almousa}}

\author{Anton Dochtermann}
\address{Texas State University}
\email{dochtermann@txstate.edu}
\urladdr{\url{https://dochtermann.wp.txstate.edu/}}
% \email{}
% \urladdr{\url{}}
%\thanks{The firstauthor  ... thanks} 
%\author{secondauthor}
%\address{address2line1\\
%         address2line2\\
%         address2line3}
%\email{secondauthor@email address}
%\urladdr{http://secondauthorwebaddress}
%\thanks{The second author ... thanks}
% End two authors

\author{Ben Smith}
\address{Lancaster University}
\email{b.smith9@lancaster.ac.uk}
\urladdr{\url{https://sites.google.com/view/bsmithmathematics/}}

% \keywords{determinantal facet ideal, binomial edge ideal, initial ideals, linear strand, free resolutions}

% combinatorial aspects 05E40 
% ideals 13F20
% comm rings monomial ideals 13F55
% binomial ideals 13F65
% resolutions 13D02; 
% applications of comm rings 13P99; 
% 14P99; 
% comb aspects of polyhedra 52B05
% comp aspects of convexity 52B55
% comb of trop 14T15
% apps of trop 14T90

\date{\today}
\begin{abstract}
We consider arrangements of tropical hyperplanes where the apices of the hyperplanes are taken to infinity in certain directions. 
Such an arrangement defines a decomposition of Euclidean space where a cell is determined by its `type' data, analogous to the covectors of an oriented matroid. 
By work of Develin-Sturmfels and Fink-Rinc\'{o}n, these `tropical complexes' are dual to regular subdivisions of root polytopes, which in turn are in bijection with mixed subdivisions of certain generalized permutohedra.  
Extending previous work with Joswig-Sanyal, we show how a natural monomial labeling of these complexes describes polynomial relations (syzygies) among `type ideals' which arise naturally from the combinatorial data of the arrangement. 
In particular, we show that the cotype ideal is Alexander dual to a corresponding initial ideal of the lattice ideal of the underlying root polytope.
This leads to novel ways of studying algebraic properties of various monomial and toric ideals, as well as relating them to combinatorial and geometric properties.
In particular, our methods of studying the dimension of the tropical complex leads to new formulas for homological invariants of toric edge rings of bipartite graphs, which have been extensively studied in the commutative algebra community. 
\end{abstract}
\maketitle
%\tableofcontents

\section{Introduction}\label{sec:intro}

The study of \emph{tropical convexity} has its origins in optimization and related fields under the guise of max-plus linear algebra; see~\cites{butkovic, cohen+gaubert+quadrat, litvinov+maslov+shpiz}.
A geometric approach was initiated by Develin and Sturmfels in \cite{develin2004tropical}, where the notion of a \emph{tropical polytope} was defined as the tropical convex hull of a finite set of points in the tropical torus $\troptorus{d}$.
A tropical polytope can also be defined via an \emph{arrangement of tropical hyperplanes} in $\troptorus{d}$.  Such an arrangement leads to a decomposition of $\troptorus{d}$ called the \emph{tropical complex}, and the subcomplex of bounded cells recovers the tropical polytope. In \cite{develin2004tropical}, it was shown that the combinatorial types of tropical complexes arising from $n$ tropical hyperplanes in $\troptorus{d}$ are in bijection with the set of regular subdivisions of a product of simplices $\Delta_{n-1} \times \Delta_{d-1}$. Via the \emph{Cayley trick} these are in turn encoded by regular \emph{mixed subdivisions} of the dilated simplex $n \Delta_{d-1}$.

%Subdivisions and triangulations of products of simplices themselves have been an area of intense research. They are examples of totally unimodular polytopes, which implies that they are equidecomposable (that is, all triangulations have the same $f$-vector).  In addition to their applications to tropical geometry, triangulations of products of simplices have been used to find disconnected flip-graphs, study the combinatorics of flag arrangements \cite{ArdBil07}, and the geometry of products of minors of a matrix \cite{babson1998}. From an algebraic point of view, the toric ideal associated to the vertices of a product of simplices $\Delta_{n-1} \times \Delta_{d-1}$ is also a fundamental object of study.  In particular it can be recovered as a \emph{determinantal ideal} generated by the $2$-minors of an $n \times d$ matrix of determinants.  This ideal is widely studied and for instance has connections to problems in enumeration and optimization \cite{sturmfels1996grobner}.

An arrangement of tropical hyperplanes gives rise to  combinatorial information captured by the \emph{type} data of cells in the tropical complex, a tropical analogue of the covector of an oriented matroid. 
%We refer to Section \ref{sec:background} for more details. 
In \cite{dochtermann2012tropical}, the second author, Sanyal, and Joswig studied monomial ideals defined by fine and coarse type data and showed that various complexes arising from the tropical complex support minimal (co)cellular resolutions, leading to a number of combinatorial and algebraic applications. 
These results extend work of Block and Yu \cite{block2006tropical}, who first exploited the connection between tropical polytopes and cellular resolutions. The ideas go further back to a paper of Novik, Postnikov, and Sturmfels \cite{NPS} who studied ideals arising from (classical) oriented matroids.
In \cite{block2006tropical} the authors also establish a connection between type ideals and Gr\"obner bases of initial ideals of toric ideals, based on results developed by Sturmfels in \cite{sturmfels1996grobner}.  In particular, it is shown that the fine cotype ideal of a tropical hyperplane arrangement is Alexander dual to a certain initial ideal of the ideal of maximal minors of a $2 \times n$ matrix of indeterminates, where the term order is determined by the underlying arrangement.   
% Other applications of Alexander duality to the study of type ideals was explored in \cite{dochtermann2012tropical}. In the case of a \emph{generic} arrangement of $n$ tropical hyperplanes in $\troptorus{d}$ the coarse type ideal is given by the power of the maximal ideal $\langle x_1, \dots, x_d \rangle^n$. It then follows that any regular mixed subdivision of $n \Delta_{d-1}$ supports a minimal resolution of this ideal.  
% As a further algebraic application, the transition from the fine type ideal to the coarse type ideal can also be seen a \emph{polarization} of $\langle x_1, \dots, x_d \rangle^n$, as studied by the first author,  Fl{\o}ystad, and Lohne in \cite{almousa2019polarizations}.  

\subsection{Generalizations to root polytopes}
 A natural generalization of the product of simplices is given by the root polytopes introduced by Postnikov in \cite{postnikov2009permutohedra}.  For a finite graph $G$ on vertex set $[n]$, the \emph{root polytope} $Q_G \subset {\mathbb R}^n$ is the convex hull of ${\bf e}_i -{\bf e}_j$ for all edges $(i,j) \in E(G)$ with $i < j$.  We will mostly be interested in root polytopes determined by \emph{bipartite graphs} $B$. The root polytope $Q_B$ associated to the complete bipartite graph $K_{d,n}$ is a product of simplices $\Delta_{d-1} \times \Delta_{n-1}$, and for general $B \subseteq K_{d,n}$ the root polytope $Q_B$ is the convex hull of a subset of the vertices of $\Delta_{d-1} \times \Delta_{n-1}$.  
 It follows that $Q_B$ inherits unimodularity from the product of simplices.
 
A bipartite graph $B \subset K_{d,n}$ also gives rise to a polytope $P_B$ obtained as a Minkowski sum of simplices, an example of a \emph{generalized permutohedron} \cite{postnikov2009permutohedra}. The Cayley embedding of the corresponding simplices naturally recovers the root polytope $Q_B$, and in particular the fine mixed subdivisions of $P_B$ are in bijection with the triangulations of $Q_B$. For these (and other) reasons, the study of triangulations of $Q_B$ is an active area of research, and for instance an axiomatic approach is introduced in \cite{galashin2018trianguloids}.
In \cite{fink2015stiefel}, Fink and Rinc\'on show how the combinatorics of triangulations of root polytopes can be encoded in a generalized version of tropical hyperplane arrangements, where one is allowed infinite coordinates in the defining functionals.  The finite entries in such an arrangement define a bipartite graph $B$, and in \cite{fink2015stiefel} it is shown how the combinatorics of the resulting tropical complex are described by the corresponding regular subdivision of the root polytope $Q_B$. This theory was also worked out from a more combinatorial perspective by Joswig and Loho in \cite{JoswigLoho:2016}. Once again the faces in the tropical complex are encoded by tropical covectors which give rise to fine/coarse types and associated monomial ideals. 

\subsection{Our contributions}

In this paper we apply methods from \cite{dochtermann2012tropical} to the more general setting of root polytopes and find surprising connections to other well-studied objects.  We consider the fine and coarse (co)type ideals of a generalized tropical hyperplane arrangement and relate them to well-known algebraic objects, including determinantal ideals and toric edge ideals.  Using this `tropical perspective' we establish new results and provide novel and more compact proofs of several known theorems from the literature.

In what follows, we fix $A$ to be a $d \times n$ matrix with entries in ${\mathbb R} \cup \{\infty\}$ (which we will hereby refer to as a \textit{tropical matrix}), and let $\cH = \cH(A)$  denote the corresponding arrangement of $n$ tropical hyperplanes in $\troptorus{d}$.  The finite entries of $A$ define a bipartite graph $B_A \subseteq K_{d,n}$. The values of these entries define a lifting of the vertices of the polytope $Q_{B_A}$, which in turn induces a regular subdivision ${\mathcal S}(A)$ of $Q_{B_A}$. We say that $A$ is \emph{sufficiently generic} if ${\mathcal S}(A)$ is a triangulation. The arrangement $\cH(A)$ also defines a subdivision $\cC(A)$ of  $\troptorus{d}$ called the \emph{tropical complex}, and we let $\cB(A)$ denote the subcomplex consisting of the bounded cells of $\cC(A)$.  We refer to Example \ref{ex: running} and Figure \ref{fig: tropHAex} for illustrations of these constructions.

Our first result is a combinatorial formula for the dimension of $\cB(A)$, which will be pivotal for our algebraic applications.  For a bipartite graph $B$, denote by $\lambda(B)$ the \emph{recession connectivity} of $B$, defined in terms of subgraphs of $B$ that induce strongly connected directed graph structures (see Definition \ref{def: recession}).  Extending earlier work on tropical convexity, we establish the following. 

\begin{thm}[Proposition \ref{prop: bounded+complex+dimension}, Corollary \ref{cor: bounded+complex+equality}] \label{thm: dim}
Let $B \subset K_{d,n}$ be a connected bipartite graph. Then for any $d \times n$ tropical matrix $A$ satisfying $B = B_A$ we have  
\[\dim(\cB(A)) \leq \lambda(B_A) - 1.\]
Furthermore, if $A$ is sufficiently generic then we have equality in the above expression.
\end{thm}

\subsubsection{Resolutions of type ideals}
We next turn to algebraic applications. For a tropical hyperplane arrangement $\cH = \cH(A)$, the cells of the tropical complex $\cC(A)$ (and hence of $\cB(A)$) have a natural monomial labeling given by their fine/coarse type and cotype. These in turn define the fine/coarse type and cotype ideals; see Section \ref{sec:background}. Our next collection of results describes how homological invariants of these ideals are encoded in the facial structure of the associated tropical complexes. In particular, by employing methods from \cite{dochtermann2012tropical}, we establish the following.

\begin{thm}[Propositions \ref{prop: typesLabelCells}, \ref{prop: cocellularRes}, \ref{prop: cotypeCellRes}] \label{thm: resolutions}
Let $\cH = \cH(A)$ be a sufficiently generic arrangement of tropical hyperplanes in $\troptorus{d}$. Then with the monomial labels given by fine type (respectively, coarse type), the complex $\cC(A)$ supports a minimal cocellular resolution of the fine (resp. coarse) type ideal. Similarly, the labeled complex $\cB(A)$ supports a minimal cellular resolution of the fine (resp. coarse) cotype ideal.
\end{thm}

Recall that the finite entries of the matrix $A$ define a bipartite graph $B = B_A \subset K_{d,n}$.  We let $P_B$ denote the corresponding generalized permutohedron (sum of simplices), and note that the coarse type ideal is generated by monomials corresponding to the lattice points of $P_B$. An application of the Cayley trick therefore implies that $P_B$ supports a minimal cellular resolution of the ideal generated by its lattice points; see Corollaries \ref{cor: cellResGenPerm} and \ref{cor: coarsemonomials}. 
We refer to Figure \ref{fig: mixedSubdiv} for an illustration.

Our results involving cellular resolutions of type ideals allow us to interpret volumes of generalized permutohedra in terms of algebraic data. For instance, we compute the volume of $P_B$
as well as other combinatorial data including \emph{$G$-draconian sequences}, 
in terms of a computation in initial ideals (see Proposition \ref{prop: volumeGenPerm}). If $\cH = \cH(A)$ is a \textit{graphic} tropical hyperplane arrangement, i.e., each column in the tropical matrix $A$ has exactly two finite entries (see Definition \ref{def: graphic}), the associated polytope $P_B$ is a graphical \emph{zonotope}, and its volume is given by the dimension of the largest syzygy of the underlying type ideal (see Corollary \ref{cor: zonotope}).  

The entries of the tropical matrix $A$ naturally define a term order on the polynomial ring whose variables are given by the vertices of the associated root polytope $Q_B$. This in turn defines an initial ideal $\inn_A(L)$ of the toric ideal $L$ defined by $Q_B$. Extending results of \cite{block2006tropical} and \cite{dochtermann2012tropical} we relate this ideal to the cotype ideal associated to the arrangement ${\cH}(A)$.

\begin{thm}[Proposition \ref{prop: cotypeIdealAD}] \label{thm: Alexdual} Let $\cH = \cH(A)$ be any arrangement of tropical hyperplanes, and let $L$ be the lattice ideal of the associated root polytope $Q_{B_A}$. Then the fine cotype ideal associated to $\cH$ is the Alexander dual of $M(\inn_A(L))$, where $M(\inn_A(L))$ is the largest monomial ideal contained in $\inn_A(L)$.
\end{thm}

In the case that every entry of $A$ is finite (so that each hyperplane in ${\mathcal H}$ has full support) we have that $B = K_{d,n}$ is the complete bipartite graph.  In this case we see that $Q_B$ is a product of simplices $\Delta_{d-1} \times \Delta_{n-1}$, $P_B = n \Delta_{d-1}$ is a dilated simplex, and the toric ideal $L$ is generated by the $2$-minors of an $d \times n$ matrix of indeterminates. 

\subsubsection{Toric edge rings}
Our primary application of type ideals will be their connection to \emph{toric edge rings}.  Fix a field $\kk$, and let $G$ be a finite graph $G$ on vertex set $V$ and edge set $E$. The \emph{edge ring} $\kk[G]$ is the $\kk$-algebra generated by the monomials $x_ix_j$ where $\{i,j\} \in E$. Define a monomial map $\pi:\kk[E] \rightarrow \kk[V]$ by $e_j \mapsto v_{j_1}v_{j_2}$, where $e_j = \{v_{j_1}, v_{j_2}\}$. The \emph{toric edge ideal} of $G$ is $J_G = \ker \pi$, so that 
\[\kk[G] \cong \kk[E]/J_G.\] 
\noindent
Homological properties of toric edge rings have been intensively studied in recent years; see Section \ref{sec: lattice} for references.  Such ideals are homogeneous and are known to be generated by binomials corresponding to even closed walks in $G$ (see \cite{villarreal1995rees}). 
% Toric edge ideals also arise as the toric ideals defined by the \emph{edge polytope} ${\mathcal P}_G$ associated to $G$, which is the convex hull of all ${\bf e}_i + {\bf e}_j$ for all edges $\{i,j\}$ in $G$.  These polytopes have been studied in their own right; for example, a combinatorial characterization of their normality is provided by Ohsugi and Hibi in \cite{ohsugihibinormal}. 
% A key observation is that if the graph $B$ is \emph{bipartite} then the toric edge ideal $J_B$ is isomorphic to the lattice ideal of the root polytope $Q_B$ described above, see Lemma \ref{lem: bipartite}.  Although this connection is implicit in several works on the subject, it does not seem to have been fully exploited. 
We apply our results regarding type ideals to provide new insights into homological properties of toric edge ideals of bipartite graphs. For instance, because type ideals always have linear resolutions, we give a short proof that toric edge rings of bipartite graphs are Cohen-Macaulay, see Proposition \ref{prop: CM}. Alexander duality also provides us with the following geometric interpretation of the regularity of toric edge rings.

\begin{thm}[Theorem \ref{thm: tropicalCplxRegularity}, Corollary \ref{cor: krulldim}] \label{thm: regdim}
Suppose $B \subset K_{d,n}$ is a connected bipartite graph. Let $A$ be a sufficiently generic tropical matrix satisfying $B = B_A$, and let $\cB(A)$ denote the bounded subcomplex of the arrangement  ${\mathcal H}(A)$. Then the regularity and Krull dimension of the toric edge ring ${\kk}[B]$ are given by
\begin{equation}
\begin{split}
\reg({\kk}[B]) &= \dim(\cB(A)); \\
\dim({\kk}[B]) &= d+n-1.
\end{split}
\end{equation}
\end{thm}

This geometric perspective on the regularity of toric edge rings of bipartite graphs leads to a number of further results.  For instance, in Theorem \ref{thm: subgraph} we observe a monotonicity property for regularity, in the sense that if $B^\prime \subset B$ is a connected subgraph then
\[\reg ({\kk}[B^\prime]) \leq \reg ({\kk}[B]).\]
\noindent
This was proved for the case of \emph{induced} subgraphs (in the not necessarily bipartite context) by H\`a, Kara, and O'Keefe in \cite{ha2019}.

We apply our combinatorial characterization of $\dim(\cB(A))$ from Theorem \ref{thm: dim} to obtain new bounds on the regularity of toric edge rings. For instance, in Theorem \ref{thm: regbound} we prove a bound on $\reg({\kk}[B])$ for all connected bipartite graphs, in terms of the bipartition of $B$. This extends a result of Biermann, O'Keefe, and Van Tuyl from \cite{BieOkeVan17}, who established the bound for \emph{chordal bipartite} graphs.
Our Theorem \ref{thm: dim} also immediately implies a result of Herzog and Hibi (see \cite{herzog2020regularity}), who bound $\reg({\kk}[B])$
by one less than the \emph{matching number} of the graph $B$. It is clear that $\lambda(B) \leq \mat(B)$ and hence we obtain a strengthening of this result, see Corollary \ref{cor: matching}.

For our next result recall that ${\kk}[B]$ has a \emph{$q$-linear resolution} if the Betti numbers satisfy $\beta_{i,j} = 0$ for all $i$ and all $j \neq q + i -1$. Note that in this case we have $\reg({\kk}[B]) = q-1$ and $J_B$ is generated by homogeneous polynomials of degree $q$.  Our techniques lead to the following.

\begin{thm}[Theorem \ref{thm: linear}]\label{thm: qreg}
Suppose $B \subset K_{d,n}$ is a finite connected bipartite graph. Then we have
\begin{itemize}
     \item $\kk[B]$ has a $2$-linear resolution if and only if $B$ is obtained by appending trees to the vertices of $K_{2,d}$.   
    
    \item If $q \geq 3$ and  $\kk[B]$ has a $q$-linear resolution then $\kk[B]$ is a hypersurface.
    \end{itemize}
\end{thm}
\noindent The first statement was established by Ohsugi and Hibi in \cite{ohsugihibi1999}*{Theorem 4.6} and the second statement recovers a result of Tsuchiya from \cite{tsuchiya2021}. Our short proofs of these results demonstrates how tropical and discrete geometric methods can be brought to bear. 

In recent years we have seen a number of mathematicians working on the combinatorics of subdivisions of root polytopes, and a mostly disjoint group studying the homological properties of toric edge ideals.  As far as we know there has been no dialogue between these communities and it is our hope that the work presented here will lead to more fruitful connections and applications.  We explore some potential projects in Section \ref{sec: GP}.

\subsection{Running example}
We end this section with an extended example to illustrate some of the main results described above. These particular constructions will be used throughout the paper.

\begin{example}\label{ex: running}
Consider the $3 \times 4$ tropical matrix 
\[A = \begin{bmatrix} 0 & 3 & \infty & 0\\ 0 & 0 & 2 & 3 \\ 0 & \infty & 0 & \infty \end{bmatrix}.\]
\noindent
The matrix $A$ defines the bipartite graph $B_A$, the arrangement ${\mathcal H}(A)$ of four tropical hyperplanes in $\troptorus{3}$, as well as the bounded subcomplex $\cB(A)$ depicted in Figure \ref{fig: tropHAex}. 
\end{example}

\begin{figure}[h]
\centering
$\vcenter{\hbox{\includegraphics[scale = 0.25]{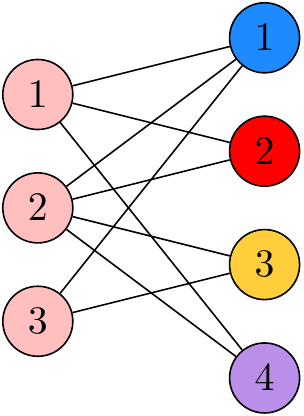}}} \qquad
\vcenter{\hbox{\includegraphics[scale = 0.75]{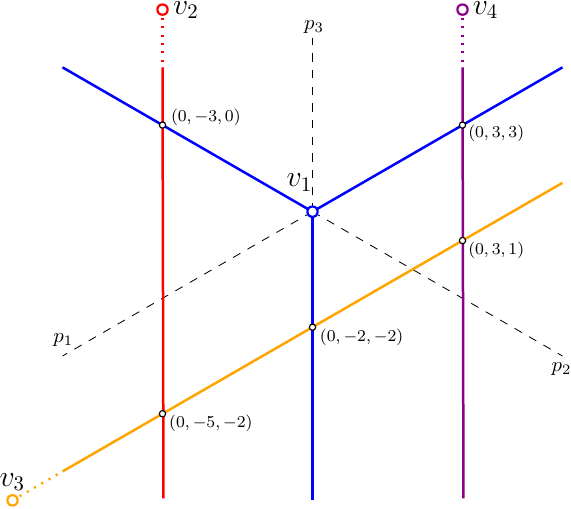}}}
\qquad
\vcenter{\hbox{\includegraphics[scale=0.75]{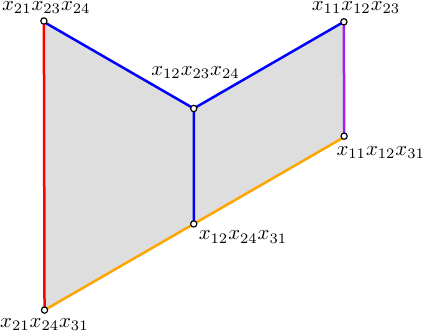}}}$
\caption{The bipartite graph $B_A$, hyperplane arrangement ${\mathcal H} = \cH(A)$, and bounded complex $\cB(A)$ associated to the tropical matrix $A$ from Example \ref{ex: running}. The points in $\troptorus{3}$ are labelled by their representative with $p_1 = 0$.}
\label{fig: tropHAex}
\end{figure}

One can check that the graph $B_A$ has recession connectivity $\lambda(B_A) = 3$, and hence Theorem \ref{thm: dim} tells us that $\dim(\cB(A)) = 2$, as indicated in the figure.
The natural labeling of cells in the decomposition of $\troptorus{3}$  determined by $\cH(A)$ defines the fine/coarse type ideals, as well as the fine/coarse cotype ideals.  Theorem \ref{thm: resolutions} says that the type ideals have cocellular resolutions supported on the tropical complex $\cC(A)$, whereas the cotype ideals have resolutions supported on the bounded complex $\cB(A)$. See Figure \ref{fig: tropHAex} for an illustration of the resolution of the fine cotype ideal.

\renewcommand*{\thethm}{\Alph{thm}}
The bipartite graph $B$ also defines a generalized permutohedron $P_B$, and if $B= B_A$ for some tropical matrix $A$ we obtain a mixed subdivision of $P_B$ that is dual to the arrangement ${\mathcal H}(A)$.  A corollary to Theorem \ref{thm: resolutions} (see Corollaries \ref{cor: cellResGenPerm} and \ref{cor: coarsemonomials}) says that this mixed subdivision supports a minimal cellular resolution of the ideal generated by the lattice points of $P_B$ (which coincides with the coarse type ideal associated to $A$). See Figure \ref{fig: mixedSubdiv} for an illustration.

\begin{figure}[h]
\begin{center}
  \includegraphics[width=\textwidth]{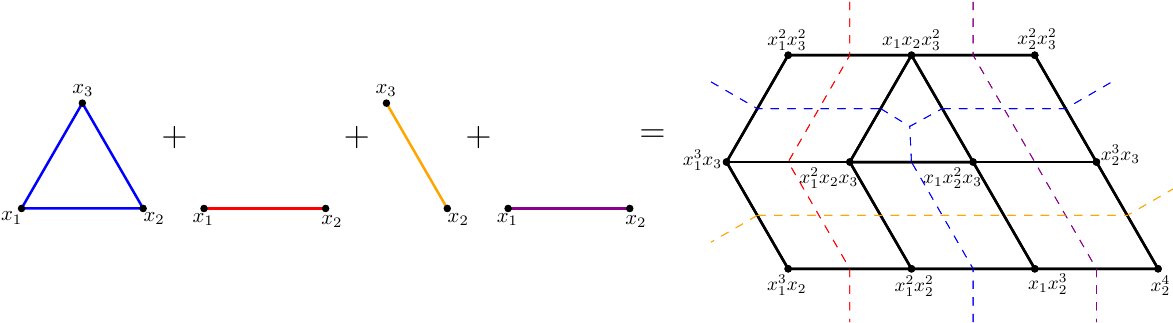}
\end{center}
\caption{The generalized permutohedron $P_B$ along with the mixed subdivision dual to the hyperplane arrangement in Figure \ref{fig: tropHAex}. The monomial ideal generated by the vertices is minimally supported by this complex.}\label{fig: mixedSubdiv}
\end{figure}

Next recall that any graph $B$ on edge set $E$ defines a toric edge ring $\kk[B]$ a toric edge ideal $J_B \subset \kk[E]$.  For the bipartite graph $B = B_A$ depicted in Figure \ref{fig: tropHAex} we have 
\[\widetilde{S} = \kk[E] = \kk[x_{11}, x_{12}, x_{14}, x_{21}, x_{22}, x_{23}, x_{24}, x_{31}, x_{33}],\]
\noindent
and the corresponding toric edge ideal is given by
\[J_B = \langle x_{11}x_{22} - \underline{x_{12} x_{21}}, \quad \underline{x_{11} x_{24}} - x_{14} x_{21}, \quad \underline{x_{12} x_{24}} - x_{14} x_{22}, \quad x_{21} x_{33} - \underline{x_{23} x_{31}}\rangle.\]
%\[J_B = \langle x_{11} x_{22} - x_{12} x_{21}, x_{21} x_{33} - x_{23} x_{31}, x_{11} x_{24} - x_{14} x_{21}, x_{12}x_{24} - x_{14}x_{22} \rangle,\]
\noindent
Note that the binomial generators are describe by the minimal closed even walks in $B$.  

Our observations imply that the ideal $J_B$ corresponds to the lattice ideal of the root polytope $Q_B$.  Furthermore, the entries of the tropical matrix $A$ provides a monomial term order on $\widetilde{S}$, which in turn defines an initial ideal $(\inn_A(J_B))$.  In our example $(\inn_A(J_B))$ is generated by the underlined monomials above. The Alexander dual of this ideal (in the polynomial ring $\widetilde S$) is given by
\[(\inn_A(J_B))^\vee =  \langle x_{12}x_{23}x_{24}, \; \; x_{11}x_{12}x_{23}, \; \; x_{11} x_{31} x_{12}, \; \; x_{31} x_{12} x_{24}, \; x_{21} x_{31} x_{24}, \; \; x_{21} x_{23} x_{24} \rangle.\]

Our Theorem \ref{thm: Alexdual} then says that $(\inn_A(J_B))^\vee$ coincides with the fine cotype ideal associated to the arrangement ${\mathcal H} = {\mathcal H}(A)$, which has a minimal resolution supported on the bounded complex $\cB$ (again see Figure \ref{fig: tropHAex}). Also by construction the arrangement ${\mathcal H}(A)$ depicted in Figure \ref{fig: tropHAex} satisfies $B = B_A$, and from Theorem \ref{thm: regdim} we conclude that
\[\reg(\kk[B]) = \dim(\cB(A)) = 2.\]

Note that in this case $J_B$ is generated by binomials of degree 2 but $\reg(\kk[B]) = 2$.  Hence this ideal does not have a linear resolution, which can also be deduced from our Theorem \ref{thm: qreg}.

\subsection{Organization and notation}

The rest of the paper is organized as follows. In Section~\ref{sec: arrangements}, we establish fundamental notation and terminology for the remainder of the paper surrounding tropical hyperplane arrangements with not necessarily full support, root polytopes, and generalized permutohedra. In Section~\ref{sec: boundedComplex}, we study the topology of the bounded complex obtained from such a tropical hyperplane arrangement. We also define the \textit{recession graph} (Definition~\ref{def: recession}) of a bipartite graph, whose connectivity properties can be used to detect bounded cells (Proposition~\ref{prop:bounded+graph}) and bound the dimension of the bounded complex (Proposition~\ref{prop: bounded+complex+dimension}). 

In Section~\ref{sec: cellRes}, we introduce the fine and coarse \textit{type} and \textit{cotype} ideals associated to an arrangement (Definition~\ref{def: typeIdeals}) and show that the bounded complex supports a minimal cellular resolution of cotype ideals (Proposition~\ref{prop: cotypeCellRes}), and that for sufficiently generic arrangements, the tropical complex supports a minimal, linear, cocellular resolution of type ideals (Proposition~\ref{prop: cocellularRes}). These results imply that the fine mixed subdivision of the generalized permutohedron associated to the arrangement supports a cellular resolution of the type ideal (Corollary~\ref{cor: cellResGenPerm}), which we apply to bound the (homological) types of type ideals in the case of graphic arrangements (Corollary~\ref{cor: zonotope}) and give a novel way of computing volumes of generalized permutohedra (Proposition~\ref{prop: volumeGenPerm}).

In Section \ref{sec: lattice}, we establish that the toric edge ideal of a bipartite graph $B$ is exactly the toric lattice ideal of the associated root polytope $Q_B$, and observe conditions on the root polytope such that its lattice ideal corresponds to well-studied ideals in commutative algebra such as ladder determinantal ideals and Ferrers ideals. This crucial observation drives the results of Section~\ref{sec: edge+ideal+results}, where we relate combinatorial and topological properties of the bounded complex associated to an arrangement with homological properties of the associated toric edge ideal. In particular, our methods recover and strengthen several results in the literature bounding regularity of toric edge ideals of bipartite graphs; see Theorem \ref{thm: chordalbip} and Corollary \ref{cor: matching}. In addition, we give a novel proof characterizing toric edge ideals of bipartite graphs with linear resolution (Theorem~\ref{thm: linear}). Finally, we conclude the paper with Section~\ref{sec: GP}, where we collect questions and potential directions for further exploration.

\subsection{Notation} Throughout the paper we make use of several mathematical objects that may not be familiar to the reader. For convenience we collect notations here.

%  $\cH = \cH(A)$  denote the corresponding arrangement of $n$ tropical hyperplanes in $\troptorus{d}$.  The arrangement $\cH$ induces a decomposition of the $\troptorus{d}$ which we call the \emph{tropical complex} of $\cH$, and denote $\cC(A)$.

\begin{itemize}
\item $\TT = (\RR\cup \{-\infty\},\oplus, \odot)$ is the \emph{(max)-tropical semiring} (Definition~\ref{def: tropicalSemiRing});

\item $\troptorus{d}$ is the tropical torus;

\item $A$ is a $d \times n$ matrix with entries in ${\mathbb R} \cup \{ \infty \}$ (a \textit{tropical matrix});

\item
$\cH(A)$ is the tropical hyperplane arrangement associated to $A$ (Definition~\ref{def: tropicalHyperplane});

\item
$\cC(A)$ is the tropical complex associated to $\cH(A)$, $\cB(A)$ is the bounded subcomplex (Definition~\ref{def: cells});

\item
$I_{T(\cH)}$ and $I_{\overline{T}(\cH)}$ are the fine type and cotype ideal, and $I_{\bft(\cH)}$ and  $I_{\overline{\bft}(\cH)}$ are the coarse type and cotype ideal, associated to $\cH$ (Definition~\ref{def: typeIdeals});

\item $B \subseteq K_{d,n}$ is a bipartite graph, $B_A$ is the bipartite graph associated to $A$;

\item
$Q_B$ is the root polytope associated to $B$ (Definition~\ref{def: rootPolytope});

\item
$P_B$ is the generalized permutohedron (sum of simplices) associated to $B$ (Definition~\ref{def: generalizedPermutohedron});

\item
${\mathbb k}[B]$ is the toric edge ring and $J_B$ is the toric edge ideal associated to $B$ with underyling field ${\mathbb k}$ (Definition~\ref{def: toric+edge+ideal});

\item
$\bddgraph{S}{B}$ is the recession graph of $S$ inside $B$ (Definition~\ref{def: recession+graph});
    
\item $\lambda(B)$ is the \textit{recession connectivity} of $B$ (Definition~\ref{def: recession}); 

\item $\LD(B)$ is the \textit{left-degree vector} of $B$ and $\RD(B)$ is its \textit{right-degree vector} (see Definition~\ref{def: leftrightDV}).
\end{itemize}

\section{Tropical hyperplanes, root polytopes, and generalized permutohedra}\label{sec: arrangements}
In this section, we review relevant notions from tropical convexity and establish notation that will be used in the rest of the paper.  We discuss tropical hyperplane arrangements and corresponding tropical complexes, and how their combinatorial data is encoded in both fine and coarse types and cotypes. We then discuss the connection to regular subdivisions of root polytopes and mixed subdivisions of generalized permutohedra.
For proofs and further reading, we refer to~\cite{JoswigBook22}.

\label{sec:background}

\subsection{Tropical Hyperplane Arrangements}
\phantom{}

\begin{definition}[Tropical Semiring]\label{def: tropicalSemiRing}
Let $\TT = (\RR\cup \{-\infty\},\oplus, \odot)$ be the \emph{(max)-tropical semiring} where
$$
x\oplus y \coloneqq \max(x,y) \quad \text{and} \quad x\odot y \coloneqq x+y.
$$
One can also define the \emph{(min)-tropical semiring} by replacing $-\infty$ with $\infty$, and the operation $\max$ with $\min$; the two semi-rings are isomorphic via
$$
-\max(x,y) = \min(-x,-y).
$$
\end{definition}
\textit{Unless stated otherwise, we shall work with the max-tropical semiring,} and note that all statements can be rephrased in terms of the min-tropical semiring.
If we need to work with both simultaneously, we shall differentiate between them by denoting them respectively by $\Tmin$ and $\Tmax$.

We can extend the operations on $\TT$ to addition and scalar multiplication on $\TT^d$, giving it the structure of a semimodule:
\begin{align*}
{\bf x} \oplus {\bf y} = (x_1 \oplus y_1, \dots, x_d \oplus y_d) \quad , \quad c \odot {\bf x} = (c \odot x_1, \dots, c \odot x_d) \quad , \quad \forall \; {\bf x}, {\bf y} \in \TT^d \, , \, c \in \TT \, .
\end{align*}
Restricting $\TT$ to $\RR$ gives a similar semimodule structure on $\RR^d$.
This will be convenient when considering tropical convexity in Section~\ref{sec:type+topology}.

Following \cite{fink2015stiefel}, we generalize the notion of a tropical hyperplane arrangement by allowing some of the hyperplanes to be ``taken to infinity'' in some directions. 

\begin{definition}[Tropical Hyperplane Arrangement]\label{def: tropicalHyperplane}
Let ${\bf a} \in (\RR\cup\{\infty\})^d$ with at least one finite entry.
The \emph{(max)-tropical hyperplane} $H$ with \emph{apex} ${\bf a}$ is the set
$$
H({\bf a}) \coloneqq \left\{{\bf p}\in\RR^d\mid (-a_1\odot p_1)\oplus (-a_2\odot p_2)\oplus \dots \oplus (-a_d\odot p_d) \text{ is attained at least twice} \right\} \, , 
$$
that is, it is the vanishing locus of a max-tropical homogeneous linear polynomial whose coefficients are $-{\bf a}$.

A $d \times n$ matrix $A = ({\bf a}_1,\dots, {\bf a}_n)$ with entries $ \in {\mathbb R} \cup \{\infty\}$ such that each ${\bf a}_i$ has at least one finite entry will be called a \emph{tropical matrix}.  Such a matrix gives rise to an (ordered) \textit{(max)-tropical hyperplane arrangement} $\cH(A) = (H_1,\dots, H_n)$ in $\RR^d$ whose $i$th hyperplane $H_i := H({\bf a}_i)$ is the tropical hyperplane with apex ${\bf a}_i$.
\end{definition}

An important observation is that if $H({\bf a})$ is a max-tropical hyperplane, its apex ${\bf a}$ has entries in $\Tmin$, and vice-versa.
Unless stated otherwise, we shall work with max-tropical hyperplanes, so that their apices will be elements of $\Tmin^d$.
Although we allow tropical hyperplanes with infinite apices, we will only consider their vanishing locus in $\RR^d$.

Observe that for any point ${\bf p} \in H({\bf a})$, the point ${\bf p} + \lambda\cdot\one$ is also contained in $H({\bf a})$, where $\one\in\RR^d$ is the vector of all ones.
For this reason, one typically views tropical objects in the \emph{tropical torus} $\troptorus{d}$, by definition the quotient of the Euclidean space $\RR^d$ by the linear subspace $\RR\one$.
By interpreting this quotient in the category of topological spaces, $\troptorus{d}$ inherits a natural topology which is homeomorphic to the usual topology on $\RR^{d-1}$. While $\troptorus{d}$ is a convenient space to view geometric objects, we must define said objects in $\RR^d$ first, as the addition operation of the tropical semiring is not well defined in $\troptorus{d}$.

\begin{definition}[Support of a Hyperplane]\label{def: support}  The \emph{support} of a tropical hyperplane $H({\bf a})$ is $\supp(H) = \{k\in [d] \mid a_k \text{ finite }\}$.
The support of a hyperplane arrangement $\cH(A) = (H_1,\dots H_n)$ is the set
$$
\supp(A) = \{(i,j)\in [d]\times [n]\mid i\text{ is in the support of } H_j\},
$$
so that $\supp(\cH(A))$ is equal to the support of the matrix $A$.
\end{definition}

The combinatorics of supports can be described by a bipartite graph in the following way.

\begin{definition}[Bipartite graphs from tropical matrices]\label{def: bipartiteGraph}
Suppose $A$ is a $d \times n$ tropical matrix and let $\cH(A)$ denote the corresponding arrangement of $n$ tropical hyperplanes in $\troptorus{d}$. Define $B_A \subseteq K_{d,n}$ to be the bipartite graph whose edges encode the support of $A$, i.e.,
\[
(i,j) \in E(B_A) \ \Longleftrightarrow \ (i,j) \in \supp(A) \, .
\]
\end{definition}

We refer to Figure \ref{fig: tropHAex} for an illustration of ${\mathcal H}(A)$ and $B_A$ described in Definitions \ref{def: tropicalHyperplane} and \ref{def: bipartiteGraph}.  One can encode information about a tropical hyperplane arrangment $\cH(A)$ via left and right degree vectors of the associated bipartite graph $B_A$.

\begin{definition}[Left/right degree vectors]\label{def: leftrightDV} For a bipartite graph $B \subset K_{d,n}$, the \textit{left degree vector} $\LD(B) = (l_1, \dots, l_d)$ is the tuple of node degrees of vertices $1,\dots, d$, and the \textit{right degree vector} $\RD(B) = (r_1, \dots, r_n)$ is the tuple of node degrees of vertices $1,\dots, n$.

\end{definition}

\begin{example}\label{ex: supportSets}
We continue with our running example 
\[A = \begin{bmatrix} 0 & 3 & \infty & 0 \\ 0 & 0 & 2 & 3 \\ 0 & \infty & 0 & \infty \end{bmatrix}.\]
The corresponding bipartite graph $B_A$ and tropical hyperplane arrangement ${\mathcal H}(A)$ is depicted in Figure \ref{fig: tropHAex}. The tropical hyperplane $H({\bf a}_1)$ has full support $\{1,2,3\}$, hyperplanes $H({\bf a}_2)$ and $H({\bf a}_4)$ have support set $\{1,2\}$, and hyperplane $H({\bf a}_3)$ has support set $\{2,3\}$. This support data can be combined to give the corresponding bipartite graph $B_A$ (see Figure \ref{fig: tropHAex}), with left degree vector $\LD(B_A) = (3,4,2)$ and right degree vector $\RD(B_A) = (3,2,2,2)$.

\end{example}

\subsection{Types and cotypes} 
%Tropical hyperplane arrangements have a very rich combinatorial structure.
By definition, a point ${\bf p }\in \troptorus{d}$ is contained in $H({\bf a})$ if and only if $\max_{j \in [d]}(p_j - a_j)$ is attained at least twice. The points that obtain this maximum value in one particular term form a polyhedral \emph{sector}, and a tropical hyperplane decomposes $\troptorus{d}$ into these sectors.

\begin{definition}[Sector]
The $i$-th (closed) sector of the (max)-tropical hyperplane $H({\bf a})$  is
\[
S_i({\bf a}) = \{{\bf p} \in \troptorus{d} \mid p_i - a_i \geq p_j - a_j \ \text{for all} \ j \in [d] \} \, .
\]
\end{definition}
The definition of a min-tropical sector is identical except with the inequality reversed, as the corresponding tropical linear functional is looking to minimize.

Given an arrangement ${\mathcal H}$ of tropical hyperplanes, to each point in $\troptorus{d}$ we can record which sector it sits in with respect to each hyperplane. This leads to the notion of \emph{types} and \emph{cotypes}.

\begin{definition}[Fine type/cotype]\label{def: fineType}
Let $A = ({\bf a}_1,\dots, {\bf a}_n)$ be a $d \times n$ tropical matrix, and let $\cH = \cH(A)$ denote the corresponding arrangement of $n$ tropical hyperplanes in $\troptorus{d}$. The \emph{fine type} of a point ${\bf p} \in \troptorus{d}$ with respect to $\cH$ is the table $T_\cH({\bf p})\in\{0,1\}^{d\times n}$ with
$$
T_\cH({\bf p})_{ik} = 1 \text{ if and only if } {\bf p}\in S_i({\bf a}_k)
$$
for $(i,k)\in \supp(\cH)$. The \emph{fine cotype} $\overline{T}_\cH({\bf p})\in\{0,1\}^{d\times n}$ is defined by $\overline{T}_\cH({\bf p})_{ik} = 1-T_\cH({\bf p})_{ik}$ for $(i,k)\in \supp(\cH)$, and zero for $(i,k)\notin \supp(\cH)$.
\end{definition}

\begin{definition}[Coarse type/cotype]\label{def: coarseType} Let $\cH = \cH(A)$ be an arrangement of $n$ max-tropical hyperplanes in $\troptorus{d}$. The \emph{coarse type} of a point ${\bf p}\in \troptorus{d}$ with respect to $\cH$ is given by $\bft_\cH({\bf p}) = (t_1, t_2,\dots, t_d)\in \NN^d$ with
$$
t_i = \sum_{k=1}^n T_\cH({\bf p})_{ik}
$$
for $i\in [d]$.
\end{definition}

\begin{notation}\label{not: typesAsGraphs}
We will often view fine (co)types as (complements of) bipartite subgraphs of the support graph $B_A$ described in Definition \ref{def: bipartiteGraph}. Explicitly, a fine type $T \in \{0,1\}^{d\times n}$ encodes a subgraph $B_T \subset B_A$ whose edges are $(i,k)$ whenever $T$ has a 1 in the $(i,k)$-th entry.
Similarly, a fine cotype $\overline{T}$ encodes a subgraph $B_{\overline{T}}\subset B_A$ given by the complement of $B_T$ inside $B_A$.
As a bipartite graph, the corresponding coarse (co)type is the left degree vector $\LD(B_T)$.
As they are equivalent, we will only explicitly use $B_T$ when graph properties are required.
\end{notation}

Observe that not all tables or bipartite graphs can arise as fine types, as every point in $\troptorus{d}$ is contained in at least one sector of a tropical hyperplane.
Each column of a viable table must have at least one non-zero entry, and viable bipartite graphs must have no disconnected nodes in $[n]$.
These tables and graphs are called \emph{type tables} and \emph{type graphs}.

\begin{remark}
Tropical types are also referred to as \textit{tropical covectors} in the literature (see e.g. \cite{fink2015stiefel}). As we will work with both types and cotypes, we choose to follow the terminology in \cite{dochtermann2012tropical} .
\end{remark}

As discussed in \cite{develin2004tropical} and \cite{JoswigLoho:2016}, 
the type data of points ${\bf p} \in \troptorus{d}$ gives rise to a polyhedral complex associated to a tropical hyperplane arrangement as follows.

\begin{definition}[Tropical complex, bounded subcomplex]\label{def: cells} Let $A$ be a $d \times n$ tropical matrix and let $\cH(A)$ denote the corresponding arrangement in $\troptorus{d}$. For a fixed fine type $T = T_\cH({\bf p})$, the set
$$
C^\circ_T \coloneqq \{{\bf q}\in\troptorus{d}\mid T({\bf q}) = T\}
$$
is a relatively open subset of $\troptorus{d} = \RR^{d-1}$ called the \emph{relatively open cell of type $T$}. The set of all relatively open cells $\cC^\circ = \cC^\circ(A)$ partitions the tropical torus $\troptorus{d}$.

Denote by $C_T$ the \emph{closed cell} of type $T$, the set of all points ${\bf q} \in \troptorus{d}$ such that $T({\bf q}) \supseteq T$.
The collection of all closed cells yields a polyhedral subdivision $\cC = \cC(A)$ of $\troptorus{d}$ called the \emph{tropical complex}.
The cells which are bounded form the \emph{bounded subcomplex} $\cB = \cB(A)$. 
% \ayah{I think there can be issues where the bounded complex is not actually bounded, unless we assume that the corresponding graph $B$ has no isolated vertices and is connected. To get around this, one option would be to define the bounded complex to consist of cells such that $t_i\neq 0$ for all $i$. This gives the ``correct'' complex even when $B$ is disconnected and the ``bounded'' complex is unbounded. I think it's reasonable to assume that $B$ has no isolated vertices, since it's no problem to restrict to smaller dimension or number of hyperplanes.}
% \anton{OK I say we just restrict to connected graphs $B$? This seems to the general strategy in the toric edge ring papers.}
\end{definition}
% \begin{remark}
% These closed cells are sometimes called \emph{polytropes} in the literature, because they are polyhedra that are also tropically convex.
% \end{remark}

\begin{example} \label{ex: tropicalcomplex}
Consider the running example discussed in Example~\ref{ex: supportSets}.
Figure~\ref{fig: tropicalcomplex} shows the tropical complex and bounded complex, along with certain distinguished cells whose types are given in Figure~\ref{fig: tropicaltypes}.
\begin{figure}
    \centering
    $\vcenter{\hbox{\includegraphics[scale=0.75]{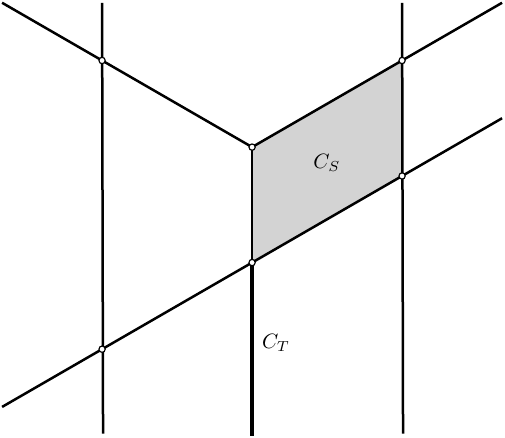}}} \qquad\qquad
    \vcenter{\hbox{\includegraphics[scale=0.75]{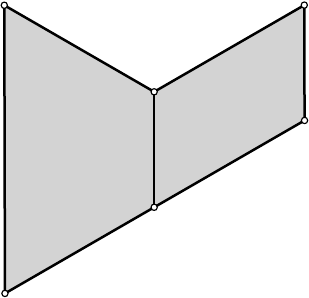}}}$
    \caption{The tropical complex and bounded complex corresponding to the tropical hyperplane arrangement in Figure~\ref{fig: tropHAex}.}
    \label{fig: tropicalcomplex}
\end{figure}
\begin{figure}
    \centering
    \includegraphics[]{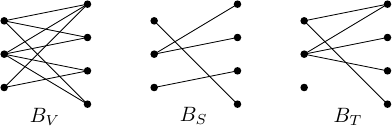}
    \caption{The tropical types corresponding to the tropical complex and distinguished cells in Figure~\ref{fig: tropicalcomplex}.}
    \label{fig: tropicaltypes}
\end{figure}
\end{example}

 The combinatorics of tropical types of an arrangement ${\mathcal H} = {\mathcal H}(A)$ reflect the geometry of the tropical complex $\cC(A)$. In the following we consider fine types as bipartite graphs, and use $c(G)$ to denote the number of connected components of a graph $G$.
 \begin{lemma}[\cite{develin2004tropical}, \cite{JoswigLoho:2016}] \label{lem:type+facts}
 Suppose $A$ is a $d \times n$ tropical matrix with corresponding arrangement ${\mathcal H}(A)$, and let $\cC(A)$ be the tropical complex. Let $\cT$ be the collection of fine types that appear. For all $S,T \in \cT$, we have:
 \begin{enumerate}
     \item $C_S \cap C_T = C_{S \cup T}$ (where $B_{S\cup T} = B_S \cup B_T$),
     \item $C_S \subseteq C_T \ \Longleftrightarrow \ B_S \supseteq B_T$,
     \item $\dim(C_T) = c(B_T) - 1$.
 \end{enumerate}
 \end{lemma}
 \begin{proof}
 See Remark 28 and Corollary 38 of~\cite{JoswigLoho:2016}.
 \end{proof}

Tropical hyperplanes arrangements enjoy a  natural notion of duality that we will need in our work. 
Explicitly, an arrangement $\cH(A)$ of $n$ hyperplanes in $\troptorus{d}$ is ``dual'' to $\cH(A^T)$, the arrangement of $d$ hyperplanes in $\troptorus{n}$ obtained by taking the transpose of the tropical matrix.
Moreover, despite being in different ambient spaces, these arrangements have isomorphic bounded complexes; this follows from \cite{fink2015stiefel}*{Proposition 6.2}.

\begin{prop}\label{prop: bounded+duality}
Let $A$ be a $d \times n$ tropical matrix.
There exists a piecewise linear bijection between the bounded complexes $\cB(A)$ and $\cB(A^T)$ that induces an isomorphism between their face posets.
\end{prop}
% \begin{proof}
% \cite[Proposition 6.2]{fink2015stiefel} establishes the result for the tropical cones generated by $A$ and $A^T$.
% As the bounded complex is a subcomplex of the tropical cone, this bijection holds for the bounded complex as well.
% \end{proof}
\begin{example}
Consider the tropical matrix
\[
A = \begin{bmatrix} 0 & 0 \\ 0 & 1 \\ 0 & 2 \end{bmatrix} \, .
\]
The tropical hyperplane arrangements $\cH(A) \subseteq \troptorus{3}$ and $\cH(A^T)\subseteq \troptorus{2}$ are given in Figure~\ref{fig: bounded+duality}.
In particular, both bounded complexes consist of two edges and three vertices.
\end{example}
\begin{figure}
    \centering
    \includegraphics[scale=1.2]{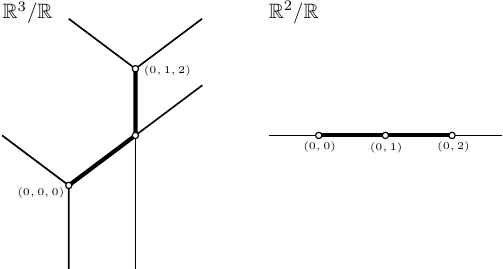}
    \caption{The tropical hyperplane arrangements $\cH(A)$ and $\cH(A^T)$. Their bounded complexes, shown in bold, are isomorphic.}
    \label{fig: bounded+duality}
\end{figure}

\subsection{Subdivisions of root polytopes and generalized permutohedra} \label{sec: permutohedra}
As discussed in \cite{fink2015stiefel}, an arrangement $\cH = \cH(A)$ of $n$ hyperplanes in $\troptorus{d}$ defined by a $d \times n$ tropical matrix $A$ gives rise to a regular subdivision of a root polytope associated to $A$.  An application of the Cayley trick then provides a regular mixed sudivision of a certain generalized permutoheron. For the case of hyperplanes with full support, this recovers the connection to subdivisions of products of simplices and mixed divisions of dilated simplices, as worked out in \cite{develin2004tropical}.  We first recall some definitions.

\begin{definition}[Root Polytopes]\label{def: rootPolytope} Let $B\subseteq K_{d,n}$ be a bipartite graph and denote by $E(B)$ its edge set. Define $Q_B$ to be the \emph{root polytope}
$$
Q_B = \conv\{{\bf e}_i-{\bf e}_{\bar{j}} \mid (i,j)\in E(B)\}\subset \RR^{d+n},
$$
where ${\bf e}_1,\dots, {\bf e}_d, {\bf e}_{\bar{1}},\dots, {\bf e}_{\bar{n}}$ are the coordinate vectors in $\RR^{d+n}$. 
\end{definition}

\begin{example}
A well-studied class of root polytopes arises when $B=K_{d,n}$ is a complete bipartite graph. In this case, the root polytope $Q_{K_{d,n}}$ is the product of simplices $\Delta_{d-1}\times \Delta_{n-1}$. For other bipartite graphs, $Q_B$ is the convex hull of some subset of vertices of $\Delta_{d-1}\times\Delta_{n-1}$.
\end{example}

The root polytope $Q_B$ of a bipartite graph $B \subset K_{d,n}$ is also closely related to another convex body obtained as a sum of (not necessarily full dimensional) simplices in ${\mathbb R}^d$.  Such polytopes are examples of the \emph{generalized permutohedra} studied by Postnikov in~\cite{postnikov2009permutohedra}.

\begin{definition}[Generalized permutohedra]\label{def: generalizedPermutohedron} Let $B\subseteq K_{d,n}$ be a bipartite graph and let $N_j(B)\subseteq [d]$ denote the set of neighbours of $j \in [n]$.
The \emph{generalized permutohedron} associated to $B$ is the polytope
$$
P_B = \sum_{j = 1}^n \Delta_{N_j(B)} \subset \RR^{d},
$$
where $\Delta_{N_j(B)}$ denotes the subsimplex of $\Delta_{d-1}$ whose vertices are labelled by $N_j(B)$.
\end{definition}

\begin{definition}[Mixed subdvision]
Suppose $P_B = \sum_{j = 1}^n \Delta_{N_j(B)} \subset \RR^{d}$ is a generalized permutohedron of dimension $k$.
A \emph{Minkowski cell} in this sum is a polytope of dimension $k$ of the form $\Delta_{I_1} + \cdots + \Delta_{I_n}$, where $I_j \subset N_j(B)$.
A \emph{mixed subdivision} $\cM$ of $P_B$ is a decomposition into a union of Minkowski cells, with the property that the intersection of any two cells is their common face.
\end{definition}

\begin{remark}
Our definition of a mixed subdivision of a sum of simplices is a special case of the corresponding concept for any Minkowski sum of polytopes. In the more general case, stricter conditions must be satisfied that are always satisfied for simplices; see~\cite{San02} for further details.
\end{remark}

The \emph{Cayley trick}~\cite{santos2005cayley} provides a geometric correspondence between the generalized permutohedron $P_B$ and the root polytope $Q_B$.  In particular we recover $P_B$ by intersecting $Q_B$ with the affine subspace $\RR^d \times \{-\sum_{j=1}^n {\bf e}_{\bar{j}}\}$.
Under this intersection, polyhedral subdivisions of $Q_B$ are in one-to-one correspondence with mixed subdivisions of $P_B$~\cite{huber2000cayley}.
In our context a $d \times n$ tropical matrix $A$ induces a \emph{regular subdivision} $\cS(A)$ of the root polytope $Q_{B_A}$ by lifting the vertex ${\bf e}_i-{\bf e}_{\bar{j}}$ to height $v_{ij}$ and projecting the lower faces of the resulting polytope back to $\RR^{d+n}$. Via the Cayley trick, this gives rise to a corresponding regular mixed subdivision $\cM(A)$ of the generalized permutohedron $P_{B_A}= \sum_{j=1}^n \Delta_{N_j(B_A)}\subseteq \RR^d$.

\begin{notation}
We say that a $d \times n$ tropical matrix $A$ is \emph{sufficiently generic} if the corresponding subdivision ${\mathcal S}(A)$ of the root polytope $Q_{B_A}$ is a triangulation.
\end{notation}

In \cite{develin2004tropical} it was shown that the tropical complex associated to an arrangement of tropical hyperplanes (with full support) is dual to the corresponding mixed subdivision of a dilated simplex. Furthermore, this duality is described by fine type data and descriptions as mixed cells. This was generalized to the setting of arbitrary tropical matrices by Fink and Rincon as follows.

\begin{prop}\label{prop: mixedSubdiv}\cite{fink2015stiefel}*{Proposition 4.1} Let $A$ be a $d \times n$ tropical matrix and let $\cH = \cH(A)$ denote the corresponding arrangement of tropical hyperplanes. Then the tropical complex $\cC(A)$ is dual to the mixed subdivision $\cM(A)$ of $P_{B_A}$. A face of $\cC(A)$ labelled by fine type $T \subseteq B_A$ is dual to the cell $\cM(A)$ obtained as the sum $\sum_{j=1}^n \Delta_{N_j(T)}$.
\end{prop}

%\begin{example}\label{ex: mixedSubdivEx}
%Let us reconsider our running example of four tropical hyperplanes in $\troptorus{3}$ from Example~\ref{ex: supportSets}.

%Applying the Cayley trick to the corresponding triangulation of the Cayley polytope gives the mixed subdivision of a trimmed permutohedron in Figure \ref{fig: mixedSubdiv}.

%\begin{figure}[h]
%\begin{center}
%  \includegraphics[width=\textwidth]{figures/Mixedsubres.pdf}
%\end{center}
%\caption{Mixed subdivision dual to the hyperplane arrangement in Figure %\ref{fig: tropHAex}.}\label{fig: mixedSubdiv}
%\end{figure}
%\end{example}

We refer to Figure \ref{fig:my_label} for an illustration of this duality in the context of our running example.

\begin{figure}[h]
    \centering
    \includegraphics[width=0.9\textwidth]{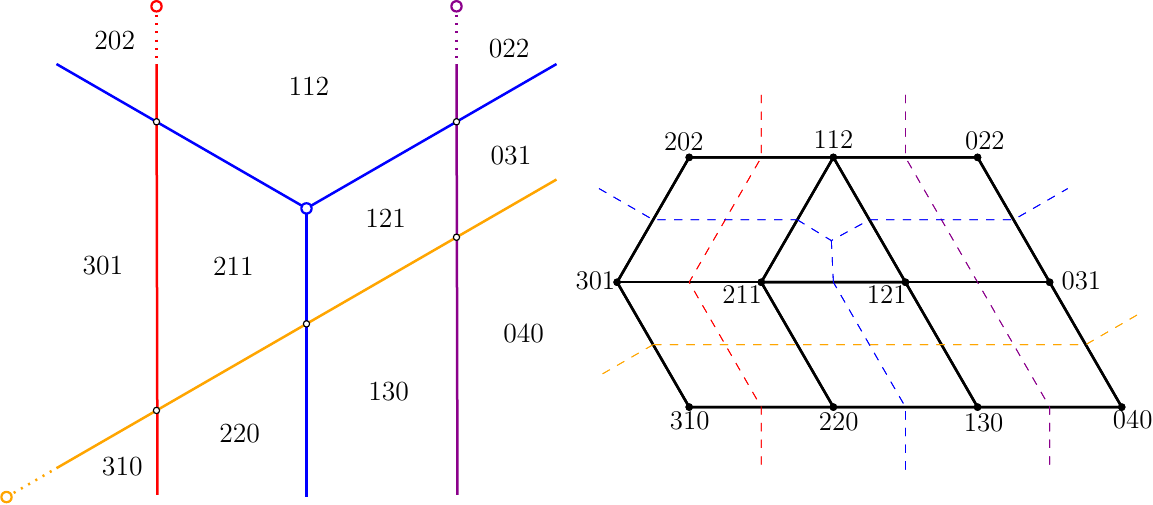}
    \caption{The tropical hyperplane arrangement ${\mathcal H}(A)$, where $A$ is from Example \ref{ex: running}, along with the corresponding mixed subdivision of a generalized permutohedron.}
    \label{fig:my_label}
\end{figure}

\section{The bounded complex and topology of types}\label{sec: boundedComplex}
In this section, we extend existing techniques in tropical convexity to derive new results regarding the bounded complex and the topology of type data associated to tropical hyperplane arrangements.
These results will be interpreted in the algebraic setting later in the paper.

\subsection{The recession cone and dimension of the bounded complex}

For the case of a tropical matrix $A$ with only finite entries, the bounded complex $\cB(A)$ can be recovered as the \emph{min}-tropical polytope $\tconv(A)$, i.e. the smallest tropically convex set containing the columns of $A$; see Definition~\ref{def: tropConvex}.
In this case there is a very simple combinatorial characterization for when a cell is contained in $\tconv(A)$: the cell $C_T$ is contained in $\cB(A)$ if and only if its type $T$ has no isolated left nodes~\cite{develin2004tropical}.

For a general tropical matrix $A$, this characterization of $\tconv(A)$ still holds (see~\cite{JoswigLoho:2016}), but if $A$ has infinite entries then the corresponding tropical polytope will have unbounded cells.
In particular, the bounded complex $\cB(A)$ will be a subcomplex of $\tconv(A)$, and so $T$ having no isolated left nodes becomes a necessary but not sufficient condition for $C_T$ being contained in $\cB(A)$.
In this section we give a precise combinatorial condition on $T$ that characterizes when $C_T$ is contained in $\cB(A)$ for a general tropical matrix $A$.

\begin{definition}\label{def: recession+graph}
Let $B \subset K_{d,n}$ be a finite simple bipartite graph and $S \subseteq B$ a subgraph.
We define the \emph{recession graph} $R = \bddgraph{S}{B}$ to be the directed bipartite graph with vertices $V(B) = [d] \sqcup [n]$ and with edges $e \in E(B)\setminus E(S)$ directed from $[d]$ to $[n]$, and $e' \in E(S)$ bidirected, i.e. directed from $[d]$ to $[n]$ and $[n]$ to $[d]$.

%\begin{align*}
%E(\bddgraph{B'}{B}) = \left\{e \in E(B)\setminus E(B') \text{ directed from $[d]$ to $[n]$}\right\} \cup \left\{e \in E(B') \text{ bidirected} \right\} .
%\end{align*}
If $B = B_A$ for a tropical matrix $A$ and $S = B_T$ for some fine type $T$, we will also write $\bddgraph{T}{A}$ for the recession graph.
\end{definition}
See Example~\ref{ex: recession_graph} and Figure~\ref{fig: recession_graph} for an example of recession graphs arising from a tropical hyperplane arrangement. 

\begin{figure}
    \centering
    \includegraphics{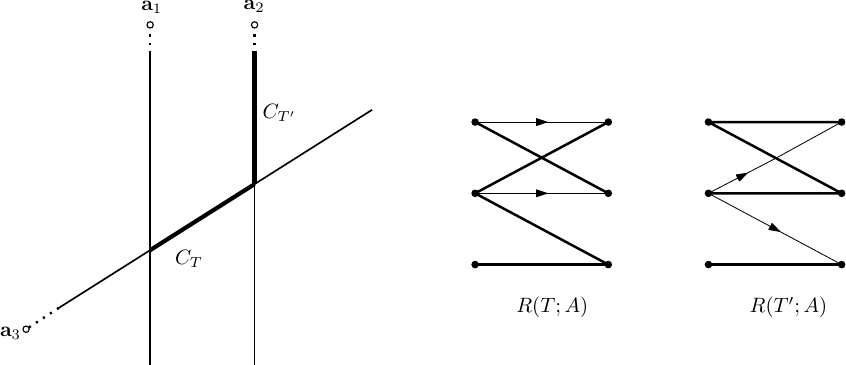}
    \caption{A tropical complex with two distinguished cells, $C_T, C_{T'}$, along with their corresponding recession graphs.
    Thin edges are directed only from $[d]$ to $[n]$, while thick edges are bidirected.}
    \label{fig: recession_graph}
\end{figure}

The recession graph is so named as it describes the \emph{recession cone} of a cell.
Recall that the recession cone of a polyhedron $P$ is the cone
\[
\rec(P) = \left\{y \in \RR^d \mid x + \lambda y \in P \; \forall x \in P, \lambda \geq 0\right\} \, ,
\]
which describes the directions in which $P$ is unbounded.
The following proposition shows the relation between the recession cone and the recession graph.
Recall that a digraph is \emph{strongly connected} if one can find a directed path from any vertex to any other vertex.

\begin{prop}\label{prop:bounded+graph} \cite{JoswigLoho:2016}
Let $A$ be a tropical matrix, and suppose $C_T$ is a cell in the tropical complex $\cC(A)$.
Then the dimension of the recession cone of $C_T$ in $\troptorus{d}$ is equal to the number of strongly connected components of $\bddgraph{T}{A}$, minus one.

In particular, the cell $C_T$ is contained in the bounded complex $\cB(A)$ if and only if the graph $\bddgraph{T}{A}$ is strongly connected.
\end{prop}
\begin{proof}
This can be deduced from a number of results from~\cite{JoswigLoho:2016} via weighted digraph polyhedra, all results cited within this proof are contained there.
By Lemma 7 and Proposition 30, the cell $C_T$ is affinely isomorphic to a weighted digraph polyhedron.
In particular, the recession cone of $C_T$ is the weighted digraph polyhedron $Q(\bddgraph{T}{A})$.
By Proposition 9, the dimension of $Q$ in $\troptorus{d+n}$ is equal to the number of strongly connected components of $\bddgraph{T}{A}$ minus one, which extends to $C_T$ in $\troptorus{d}$ by affine isomorphism.
Therefore $C_T$ is bounded if and only if $\bddgraph{T}{A}$ is strongly connected.
\end{proof}

\begin{example}\label{ex: recession_graph}
Consider the tropical complex depicted in Figure~\ref{fig: recession_graph}, along with the recession graphs for the cells $C_T$ and $C_{T'}$.
The cell $C_{T}$ is bounded, as reflected by the fact that the recession graph $\bddgraph{T}{A}$ is strongly connected.
The cell $C_{T'}$ is unbounded, since in the recession graph $\bddgraph{T'}{A}$, the edge $(2,3)$ can only be traversed from left to right. Hence there is no path from left node $3$ to any other left node that respects edge orientation.
\end{example}

We get the following as an immediate corollary.

\begin{cor}\label{prop: bdd_complex_nonempty}
The bounded complex $\cB(A)$ is non-empty if and only if $B_A$ is connected.
\end{cor}
\begin{proof}
The bounded complex is non-empty if and only if it contains a 0-cell, as any bounded $k$-cell must contain a 0-cell.
By Lemma~\ref{lem:type+facts}, a cell $C_T$ is zero dimensional if and only if $B_T\subseteq B_A$ is connected.
This immediately shows $\cB(A)$ non-empty implies $B_A$ is connected.

Conversely if $B_A$ is connected, the root polytope $Q_{B_A}$ is full-dimensional~\cite{postnikov2009permutohedra}*{Lemma 12.5} and the maximal cells of its regular subdivision $\cS(A)$ correspond to 0-cells of $\cC(A)$ by Proposition~\ref{prop: mixedSubdiv}.
As $\cC(A)$ contains a 0-cell, the bounded complex must be non-empty.
\end{proof}

Inspired by these observations we define the following invariant of a bipartite graph.  Recall that we use $c(B)$ to denote the number of connected components of a graph $B$.

\begin{definition}\label{def: recession}
For a bipartite graph $B$ define the \emph{recession connectivity} of $B$ to be the integer
\[
\lambda(B) = \max_{S \subseteq B}\left\{c(S) \mid \bddgraph{S}{B} \text{ is strongly connected } \right\} \, .
\]
\end{definition}

\begin{remark}\label{rem: lambdamatching}
Note that for any bipartite graph $B$ we have 
\[\lambda(B) \leq \mat(B),\]
\noindent
where $\mat(B)$ denotes the \emph{matching number} of $B$. Indeed, any $S \subset B$ realizing $\lambda(B)$ contains a matching of size $c(S)$.
\end{remark}

\begin{prop} \label{prop: bounded+complex+dimension}
Let $B$ be a finite simple connected bipartite graph.
For any tropical matrix $A$ satisfying $B = B_A$, we have $\dim(\cB(A)) \leq \lambda(B_A) - 1.$
Furthermore, there exists a generic tropical matrix $A$ with $\dim(\cB(A)) = \lambda(B_A) - 1$.
%If $A$ is sufficiently generic then this is bound is attained.
\end{prop}

% \ben{If we remove the unmixed cut stuff, Prop~\ref{prop: bdd_complex_nonempty} can be deduced from this.}

\begin{proof}
The first part of this statement is an immediate corollary of Proposition~\ref{prop:bounded+graph}, as the dimension of any bounded cell is bounded above by $\lambda(B_A) - 1$.
For the second part, let $S \subseteq B$ be a subgraph realizing $\lambda(B)$, i.e. $\bddgraph{S}{B}$ is strongly connected, and has $\lambda(B)$ connected components.
We can assume that $S$ is a forest, as removing an arbitrary edge from any cycle in $S$ does not strongly disconnect $\bddgraph{S}{B}$ nor does it affect the number of connected components of $S$.

It remains to show that we can find a tropical matrix $A$ and fine type $T$ such that $B = B_A$ and $S = B_T$.
We define the matrix $A$ by setting $a_{ij} = 0$ for all $(i,j) \in S$, sufficiently large generic values for $(i,j) \in B \setminus S$, and $\infty$ otherwise.
Then $B = B_A$, and the regular subdivision of $Q_B$ induced by $A$ is a triangulation that contains the cell $Q_{S}$.
%As $\overline{H}$ contains no unmixed cuts, $Q_H$ is not contained in a facet of $Q_G$.
After applying the Cayley trick, $Q_{S}$ corresponds to a mixed cell dual to the cell $C_T$ in the tropical complex.
By Proposition~\ref{prop:bounded+graph}, $C_T$ is contained in the bounded complex, and its dimension is $\lambda(B) - 1$.
Furthermore, as $\lambda(B) - 1$ is the maximum possible dimension of a bounded cell, this shows $\dim(\cB(A)) = \lambda(B) -1$.
\end{proof}

\begin{lemma}\label{lem: recbipartite}
For the complete bipartite graph $K_{d,n}$ we have $\lambda(K_{d,n}) = \min\{d,n\}$.
\end{lemma}

\begin{proof}
Without loss of generality assume that $d \leq n$ and denote the vertices of $B = K_{d,n}$ as $\{v_1, \dots, v_d\} \cup \{w_1, \dots, w_n\}$.  It is clear that $\lambda(K_{d,n}) \leq d$. For the reverse inequality, define $S$ to be the graph with edges $v_1w_1, \dots, v_dw_d, v_dw_{d+1}, \dots v_dw_n$. Then $S$ has $d$ connected components and $R(S;B)$ is strongly connected.  The result follows.
\end{proof}

The following lemma shows that $\lambda$ is monotone on subgraphs, and will be useful at numerous points.
\begin{lemma}\label{lem: lambda+monotone}
Let $B$ be a finite simple connected bipartite graph and $B' \subseteq B$ a connected subgraph.
Then $\lambda(B') \leq \lambda(B)$.
\end{lemma}
\begin{proof}
Let $S' \subseteq B'$ be a graph realizing $\lambda(B')$ i.e. $c(S') = \lambda(B')$ and $\bddgraph{S'}{B'}$ is strongly connected.
Note that we can assume $S'$ is a forest: removing an edge from any cycle in $S'$ does not disconnect $\bddgraph{S'}{B'}$ as we can traverse the cycle in the opposite direction.
Let $T_{B'}$ be a spanning tree of $B'$ that contains $S'$.
We then extend this to $T_B$, a spanning tree of $B$ that contains $T_{B'}$.
We define $S \subseteq B$ as the graph
\[
S = S' \cup \left(T_B \setminus T_{B'}\right) \, .
\]
We claim that $\bddgraph{S}{B}$ is strongly connected and that $c(S) \geq c(S')$.

We first show strong connectivity.
Consider $v, v' \in V(\bddgraph{S}{B})$.
If $v, v' \in V(B')$, the same path in $\bddgraph{S'}{B'}$ connects them in $\bddgraph{S}{B}$.
If $v, v' \in V(B \setminus B')$, the edges of $T_B \setminus T_{B'}$ give a path between them.
If $v \in V(B \setminus B'), v \in V(B')$, the edges of $T_B \setminus T_{B'}$ give a path into $B'$, where we can use the same path in $\bddgraph{S'}{B'}$.

Finally we show that $c(S) \geq c(S')$.
Let $S' = S_1 \sqcup \dots \sqcup S_k$ be the decomposition of $S^\prime$ into connected components.
Suppose that there exist some set of edges $E \subseteq T_B \setminus T_{B'}$ that connect two components $S_i$ and $S_j$.
As $T_{B'}$ is a spanning tree, there exist edges $E' \subseteq T_{B'}$ that connect $S_i, S_j$ that are disjoint from $E$.
This implies $S_i \cup S_j \cup E \cup E'$ contains a cycle, and is contained in $T_B$, a contradiction as $T_B$ a tree.
Therefore $S$ has at least as many components as $S'$.
\end{proof}

From the previous two lemmas we observe that if $B \subset K_{d,n}$ is any connected bipartite graph then $\lambda(B) \leq \min\{d,n\}$.

%Note that the proof ignores much of the structure in $B$, as it may be possible to replace $T_B \setminus T_{B'}$ with a graph with far more connected components depending on the structure of $B$.\As a result, the inequality in Lemma~\ref{lem: lambda+monotone} can be arbitrarily large. 
%\anton{What do you mean by `arbitrarily large'? Surely the inequality is bounded by eg the number of vertices of $B$? I guess you mean you constructed families of $B^\prime \subset B$ where the gap grows without bound, but is that surprising?} 

\subsection{Topology of types}\label{sec:type+topology}
In this section, we investigate topological properties of certain subsets of $\troptorus{d}$ induced by a tropical hyperplane arrangement. 
In particular, our goal is to show that these subsets are contractible, which will be central to applications to (co)cellular resolutions in Section \ref{sec: cellRes}. Note that none of our results in this section require $\cH$ to be generic nor have full support. The results here mimic those in \cite{dochtermann2012tropical}*{Section 2.3}, but require some additional considerations to account for not necessarily full support.

We begin by recalling the notion of tropical convexity. It was established in \cite{develin2004tropical}*{Theorem 2} that tropically convex sets are contractible.

\begin{definition}[Tropical convexity]\label{def: tropConvex} A subset $S$ of $\RR^d$ is \textit{tropically convex} if the set $S$ contains the point ${\bf p} = c\odot {\bf x} \oplus d\odot {\bf y}$ for all ${\bf x},{\bf y} \in S$ and $c,d\in\RR$.
\end{definition}

If the distinction is important, we shall specify that a set is \emph{$\max$-tropically convex} (resp. $\min$) if the tropical addition operation is $\max$ (resp. $\min$).
Note that all tropically convex sets contain $\RR\one$, therefore we often view them as subsets of $\troptorus{d}$.
When we refer to a tropically convex set $S$ in $\troptorus{d}$, we mean that $S + \RR\one$ is tropically convex; recall that $\oplus$ is not well defined on $\troptorus{d}$.

Given a tropical matrix $A = (\aa_1, \dots, \aa_n)$ with entries in $\Tmin$, we can define the \emph{$\min$-tropical polytope} $\tconv(A)$ to be the set of points of $\RR^d$ obtainable as a $\min$-tropically convex combination of columns of $A$, i.e.
\[
\tconv(A) = \left\{\pp \in \RR^d \mid \exists \lambda_i \in \RR \text{ s.t. } \pp = \bigoplus_{i=1}^n \lambda_i \odot \aa_i \right\} \, .
\]

\begin{notation}\label{not: typesPartialOrder} For two fine types $T, T' \in \{0,1\}^{n\times d}$, we write $T\leq T'$ if $T_{ij} \leq T'_{ij}$ for all $i,j$. This is equivalent to saying that the bipartite graph $B_T$ is a subgraph of $B_{T'}$. We let $\min(T,T')$ and $\max(T,T')$ denote the table with entries given by the component-wise minimum and maximum, respectively.
Equivalently, $B_{\min(T,T')}:= B_T \cap B_{T'}$ is the largest subgraph of $B_A$ contained in both $T$ and $T'$, and $B_{\max(T,T')}:= B_T \cup B_{T'}$ is the smallest subgraph of $B_A$ containing both $T$ and $T'$.
\end{notation}

The following theorem is the main theorem of this subsection. It says that down-sets of cells partially ordered by (either fine or coarse) type are tropically convex and hence contractible.

\begin{theorem}\label{thm: contractible} Let $A$ be a $d\times n$ tropical matrix and let $\cH = \cH(A)$ denote the corresponding arrangement of $n$ tropical hyperplanes in $\troptorus{d}$.
Let $U\in \{0,1\}^{n\times d}$ and $\uu\in \NN^d$.
With labels determined by fine (resp. coarse) type, the following subsets of $\troptorus{d}$ are max-tropically convex and hence contractible:
\begin{align*}
    &(\cC(A), T)_{\leq U} &\coloneqq \qquad&\{\pp\in \troptorus{d} \mid T_\cH(\pp)\leq U\} &= \qquad &\bigcup \{C\in \cC(A) \mid T_\cH(C) \leq U\} \\
    &(\cC(A), \bft)_{\leq \uu} &\coloneqq \qquad &\{\pp\in \troptorus{d} \mid \bft_{\cH}(\pp) \leq \uu\} &= \qquad &\bigcup \{C\in \cC(A) \mid \bft_\cH(C)\leq \uu\}.
\end{align*}
Analogously, the following subsets of $\troptorus{d}$ with labels determined by fine (respectively, coarse) cotype are min-tropically convex and hence:
\begin{align*}
    &(\cC(A), \overline{T})_{\leq U} &\coloneqq \qquad&\{\pp\in \troptorus{d} \mid \overline{T}_\cH(\pp)\leq U\} &= \qquad &\bigcup \; \{C\in \cC(A) \mid \overline{T}_\cH(C) \leq U\} \\
    &(\cC(A), \overline\bft)_{\leq \uu} &\coloneqq \qquad &\{\pp\in \troptorus{d} \mid \overline\bft_{\cH}(\pp) \leq \uu\} &= \qquad &\bigcup \; \{C\in \cC(A) \mid \overline\bft_\cH(C)\leq \uu\}.
\end{align*}
\end{theorem}

The following lemma gives the first step towards proving Theorem \ref{thm: contractible}.

\begin{lemma}\label{lem: maxTropConvex}  Let $A$ be a $d\times n$ tropical matrix and let $\cH = \cH(A)$ denote the corresponding arrangement of $n$ tropical hyperplanes in $\troptorus{d}$, and let $\pp,\qq\in\troptorus{d}$. Then
$$
\min(T_\cH(\pp),T_\cH(\qq)) \leq T_\cH({\mathbf r}) \leq \max(T_\cH(\pp), T_\cH(\qq))
$$
and
$$
\bft_\cH({\mathbf r})\leq \max(\bft_\cH(\pp),\bft_\cH(\qq))
$$
for $r\in\tconv^{\max}\{\pp,\qq\}$ on the max-tropical line segment between $\pp$ and $\qq$.
\end{lemma}

\begin{proof} 
To allow us to take tropical sums of points, we will prove the result in $\RR^d$.
Recall that hyperplane sectors and tropically convex sets are closed under scalar addition, and so the result will also hold in $\troptorus{d}$.

Let ${\mathbf r} = (\lambda \odot \pp) \oplus (\mu \odot \qq) = \max\{\lambda + \pp, \mu+\qq\}$ for $\lambda, \mu\in \RR$, and fix some $k\in[d]$. Without loss of generality, assume that $r_k = \lambda+p_k \geq \mu + q_k$.

First we show that $\min(T_\cH(\pp),T_\cH(\qq)) \leq T_\cH({\mathbf r})$. If at least one of the points $\pp$ and $\qq$ are not in the $k$th sector of some hyperplane $H({\mathbf v}_i)$, then $\min(T_\cH(\pp)_{ki},T_\cH(\qq)_{ki}) = 0$ and certainly $T_\cH({\mathbf r})_{ki}\geq 0$.
Now suppose that both $\pp$ and $\qq$ are in the $k$th sector for some hyperplane $H(\vv)$, so $p_k-p_i \geq v_k-v_i$ and $q_k - q_i \geq v_k - v_i$ for all $i\in\supp(H(\vv))$. Fix some $j\in [d]$. If $r_j = \lambda + p_j \geq \mu + q_j$, then $r_j - r_k = p_j - p_k$ and $r$ is in the $k$th sector of $H(\vv)$; similarly, if $r_j = \mu + q_j \geq \lambda + p_j$, then $r_j - r_k = q_j - q_k$ and $r$ again lies in the $k$th sector of $H(\vv)$.

Now we show that $T_\cH({\mathbf r}) \leq \max(T_\cH(\pp), T_\cH(\qq))$. We show that if $r$ is in the $k$th sector of some hyperplane $H(\vv)$, then at least one of $\pp$ and $\qq$ must be, too. Assume that $r_k - r_j \geq v_k - v_j$ for all $j\in [d]$, and also assume without loss of generality that $r_k = \lambda+p_k \geq \mu + q_k$. Then $\lambda + p_k \geq v_k - v_j + r_j$ for all $j\in [d]$. We also know that $r_j \geq \lambda + p_j$ by the definition of $r$, so $\lambda + p_k \geq v_k - v_j + \lambda + p_j$. This implies that $p_k - p_j \geq v_k - v_j$, so $\pp$ is in sector $k$ of the hyperplane $H(\vv)$.

The last claim follows from the definition of coarse types.
\end{proof}

With this lemma in hand, we may now prove the main theorem of this subsection.

\begin{proof}[Proof of Theorem \ref{thm: contractible}] The max-tropical convexity of the sets $(\cC(A),T)_{\leq U}$ and $(\cC(A),\bft)_{\leq \uu}$ follows from Lemma~\ref{lem: maxTropConvex}.
Reasoning similar to that of the proof of Lemma~\ref{lem: maxTropConvex} shows that if ${\mathbf r}$ is a point in the $\min$-tropical line segment between $\pp$ and $\qq$, then $T_\cH({\mathbf r}) \geq \min(T_\cH(\pp), T_\cH(\qq))$; passing to complements, we obtain that $\overline T_{\cH}({\mathbf r}) \leq \max(\overline{T}_\cH(\pp), \overline{T}_\cH(\qq))$ and $\overline \bft_{\cH}({\mathbf r}) \leq \max(\overline{\bft}_\cH(\pp), \overline{\bft}_\cH(\qq))$.
This gives the claim for cotypes.

\end{proof}

We will also need to extend Theorem~\ref{thm: contractible} to downsets of cotypes restricted to the bounded complex.

\begin{prop} \label{prop: bounded+contractible}
The following subsets of $\troptorus{d}$ with labels determined by fine (respectively, coarse) cotype are contractible: 
\begin{align*}
    (\cB(A), \overline{T})_{\leq U} \quad &\coloneqq \quad \bigcup \; \{C\in \cB(A) \mid \overline{T}_\cH(C) \leq U\} \\
    (\cB(A), \overline\bft)_{\leq \uu} \quad &\coloneqq \quad \bigcup \; \{C\in \cB(A) \mid \overline\bft_\cH(C)\leq \uu\}.
\end{align*}
\end{prop}

\begin{proof}
We prove the theorem for downsets of fine cotypes, the proof for coarse cotypes is identical.
If $(\cB(A), \overline{T})_{\leq U}$ is empty then we are done, so assume it is non-empty.

We first note that every cell $C \subseteq \cC(A)$ is a (min)-tropical polytope.
To see this, observe that we can write $C_T$ as the solution set of the linear inequalities:
\[
\{x_i - a_{ik} \geq x_j - a_{jk} \, \mid \, T_{ik} = 1 \, , \, a_{ik}, a_{jk} \neq \infty \} .
\]
Each of these inequalities can also be written as min-tropical inequalities, and so $C_T$ is the solution set of finitely many min-tropical inequalities.
By the Tropical Minkowski-Weyl Theorem~\cite{JoswigBook22}*{Theorem 7.11}, this implies that $C_T$ is a min-tropical polytope.

For each cell $C \in (\cC(A), \overline{T})_{\leq U}$, let $V_C$ denote a finite generating set such that $C = \tconv(V_C)$ and $V = \bigcup V_C$ the union of these finite generating sets.
For any $p \in (\cC(A), \overline{T})_{\leq U}$, we can write $p$ as a min-tropical convex combination of elements of $V$, and so $(\cC(A), \overline{T})_{\leq U} \subseteq \tconv(V)$.
However, the set $(\cC(A), \overline{T})_{\leq U}$ is min-tropically convex itself by Theorem~\ref{thm: contractible}, and so $(\cC(A), \overline{T})_{\leq U} = \tconv(V)$.

Letting $m = |V|$, one can view $V$ as a $d \times m$ tropical matrix.
By~\cite{JoswigLoho:2016}*{Theorem 33}, $\tconv(V)$ is obtainable as the upper hull of the \textit{envelope} of the tropical matrix $V$ (see also \cite{JoswigBook22}*{Section 6.1})
\[
    \mathcal{E}(V) = \{(y,z) \in \RR^d \times \RR^m \mid y_i - z_j \leq v_{ij}\}
\]
orthogonally projected onto $\RR^d$. 
Moreover, this projection is an affine isomorphism when restricted to the bounded cells of $\mathcal{E}(V)$ and $\tconv(V)$ respectively.
Note that the bounded cells of $\tconv(V)$ are precisely $(\cB(A), \overline{T})_{\leq U}$.

As $(\cB(A), \overline{T})_{\leq U}$ is non-empty, $\tconv(V)$ has at least one bounded cell and therefore so does $\mathcal{E}(V)$.
This makes $\mathcal{E}(V)$ a pointed polyhedron, a polyhedron with no lineality space.
By~\cite{JoswigBook22}*{Example 10.54}, $\mathcal{E}(V)$ is projectively equivalent to a polytope with a distinguished \emph{face at infinity}
and the set of faces of $\mathcal{E}(V)$ that are disjoint from the face at infinity are precisely the bounded cells of $\mathcal{E}(V)$.  Therefore by~\cite{MillerSturmfels:2005}*{Theorem 4.17} the bounded part of $\mathcal{E}(V)$ is contractible .
As this is affinely isomorphic to $(\cB(A), \overline{T})_{\leq U}$, this implies that $(\cB(A), \overline{T})_{\leq U}$ is also contractible.
\end{proof}

The proof of Proposition~\ref{prop: bounded+contractible} is much cleaner for arrangements with full support.
When all entries of $A$ are finite, the bounded complex $\cB(A)$ agrees with the tropical polytope $\tconv(A)$ and hence is $\min$-tropically convex.
This immediately implies $(\cB(A), \overline{T})_{\leq U}$ is $\min$-tropically convex as an intersection of $\min$-tropically convex sets, and therefore contractible.
When $A$ has infinite entries, $\cB(A)$ is only a subset of $\tconv(A)$ and no longer tropically convex in general.
Figure~\ref{fig: bounded+complex+nonconvex} shows an example of such a complex consisting of two maximal cells.
For any pair of points in the interior of these two cells, the tropical line segment between them must leave the bounded complex.

\begin{figure}
    \centering
    \includegraphics[scale=1.2]{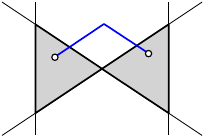}
    \caption{An arrangement of four degenerate tropical hyperplanes whose bounded complex is not tropically convex.
    The tropical line segment (blue) between any two points in the interior of different maximal cells must leave the bounded complex at some point.}
    \label{fig: bounded+complex+nonconvex}
\end{figure}

\section{Ideals and resolutions from tropical arrangements}\label{sec: cellRes}
In \cite{dochtermann2012tropical} the authors introduce monomial ideals arising from the type data associated to a tropical hyperplane arrangement, and prove that the facial structure of the various tropical complexes describes algebraic relations (syzygies) among the generators.
In this section we extend these constructions and results to the case of tropical hyperplanes with not necessarily full support. In particular, we show that the cellular complexes from Definition \ref{def: cells} support (co-)cellular resolutions of certain ideals arising naturally from the combinatorics of the arrangement.

\subsection{Ideals from the arrangement}
The ideals that we study will live in polynomial rings defined by the finite entries of a tropical matrix.

\begin{notation}\label{not: cellRings}
Suppose $A$ is a $d \times n$ tropical matrix.  We let $X_A$ denote the $d \times n$ generic matrix with entries 
\begin{align*}
    X_A (i,j) = \begin{cases}
    x_{i,j} &\text{ if } (i,j)\in\supp(A)\\
    0 &\text{ otherwise}
    \end{cases}
\end{align*}
and for a field $\kk$ define $\widetilde S = \kk[X_A]$ to be the polynomial ring in the indeterminates in $X_A$.  Note that if $B = B_A$ is the bipartite graph associated to $A$, then the number of indeterminates of $\widetilde S$ is given by $|E(B)|$, the number of edges of $B$. We let $S$ denote the polynomial ring $S = \kk[x_1,\dots, x_d]$.

Let the weight of the variable $x_{ij}$ in $\widetilde S$ be $v_{ij}$ and the weight of a monomial $\xx^\aa = \prod x_{ij}^{a_{ij}}\in \widetilde S$ be $\sum a_{ij} v_{ij}$.
The initial form $\inn_A(f)$ of a polynomial $f = \sum c_i \xx^{\aa_i}$ is defined to be the sum of terms $c_i \xx^\aa_i$ such that $\xx^\aa_i$ has maximal weight.
% Fix natural numbers $d$ and $n$.
% Let $S = k[x_1,\dots, x_d]$ be a polynomial ring over an arbitrary field $k$. 
% If $A$ is a matrix in $(\RR\cup \{\infty\})^{d\times n}$, 
% % set $\widetilde S_A = k[x_{i,j}\mid (i,j)\in\supp(A)]$.
% denote by $X_A$ the $d\times n$ ``sparse'' generic matrix with entries
% \begin{align*}
%     X_A (i,j) = \begin{cases}
%     x_{i,j} &\text{ if } (i,j)\in\supp(A)\\
%     0 &\text{ otherwise}
%     \end{cases}
% \end{align*}
% and define $\widetilde S_A = k[X_A]$ to be the polynomial ring in the indeterminates in $X_A$.
% Let the weight of the variable $x_{ij}$ in $\widetilde S_A$ be $v_{ij}$ and the weight of a monomial $\xx^\aa = \prod x_{ij}^{a_{ij}}\in \widetilde S_A$ be $\sum a_{ij} v_{ij}$.
% The initial form $\inn_A(f)$ of a polynomial $f = \sum c_i \xx^{\aa_i}$ is defined to be the sum of terms $c_i \xx^\aa_i$ such that $\xx^\aa_i$ has maximal weight.

% Let $B_A$ be the bipartite graph associated to the tropical hyperplane arrangement $\cH(A)$ as in Definition \ref{def: bipartiteGraph}.
% Let $Q_{B_A}$ be the root polytope associated to $B_A$, and denote by $\cA$ the integer matrix whose column vectors are the vertices of $Q_{B_A}$.
\end{notation}

The fine and coarse (co)-type data from Definitions \ref{def: fineType} and \ref{def: coarseType} naturally give rise to monomial ideals. These ideals were first introduced in the full-support case in \cite{dochtermann2012tropical} as a means of giving a tropical analogue of the \emph{oriented matroid ideals} studied by Novik, Postnikov, and Sturmfels in \cite{NPS}.

\begin{definition}[Type/cotype ideals]\label{def: typeIdeals} Let $A$ be a $d \times n$ tropical matrix and let $\cH = \cH(A)$ denote the corresponding arrangement of $n$ tropical hyperplanes in $\troptorus{d}$. The \emph{fine type ideal} and \emph{fine cotype ideal} associated to $\cH$ are the squarefree monomial ideals
\begin{align*}
I_{T(\cH)} &= \langle \xx^{T({\bf p})}\mid {\bf p}\in \troptorus{d}\rangle \subset \widetilde S\\
I_{\overline{T}(\cH)} &= \langle \xx^{\overline{T}({\bf p})}\mid {\bf p} \in \troptorus{d}\rangle \subset \widetilde S
\end{align*}
where $\xx^{T({\bf p})} = \displaystyle \prod \{x_{ij}\mid T({\bf p})_{ij} = 1\}$. Analogously, the \emph{coarse type ideal} and \emph{coarse cotype ideal} associated to $\cH$ are given by
\begin{align*}
    I_{\bft(\cH)} &= \langle \xx^{\bft({\bf p})}\mid {\bf p}\in \troptorus{d} \rangle \subset S\\
I_{\overline{\bft}(\cH)} &= \langle \xx^{\overline{\bft}({\bf p})}\mid {\bf p}\in \troptorus{d}\rangle \subset S
\end{align*}
where $\xx^{\bft({\bf p})} = x_1^{t_1}\dots x_n^{t_n}$ with $\bft({\bf p}) = (t_1,\dots, t_d)$.
\end{definition}

We will see in the following sections that an arrangement ${\mathcal H}$ gives rise to complexes that support minimal free resolutions of these ideals. This provides a means to understand the projective dimension and regularity of our type ideals.  Later in the paper we will use our techniques to obtain a formula for the Krull dimension as well, see Section \ref{sec: Krull}.

\subsection{Cellular resolutions of type and cotype ideals}
Next we show that the tropical complex and bounded complex of a tropical hyperplane arrangement support minimal (co)cellular resolutions of the various type and cotype ideals. This provides a tropical analogue of the work of \cite{NPS} and generalizes results from \cite{block2006tropical} and \cite{dochtermann2012tropical}. 
For a reference on cellular resolutions we refer the reader to \cite{MillerSturmfels:2005}*{Chapter 4}; for details on cocellular resolutions, see \cite{MillerSturmfels:2005}*{Section 5.3}. The main criteria we use in this section to check whether a resolution is cellular is \cite{MillerSturmfels:2005}*{Prop. 4.5}, and the analagous co-cellular version which follows from the same proof techniques.

\begin{prop}\label{prop: typesLabelCells} 
Let $A$ be a $d \times n$ tropical matrix and let $\cH = \cH(A)$ denote the corresponding arrangement of $n$ tropical hyperplanes in $\troptorus{d}$. For every cell $C\in \cC(A)$ of codimension at least $1$, 
\begin{align*}
    T_\cH(C) &= \max\{T_\cH(D)\mid C\subset D\}, \text{ and }\\
    \bft_\cH(C) &= \max\{\bft_\cH(D)\mid C\subset D\}.
\end{align*}
Therefore, both the fine type and the coarse type yield a monomial colabeling for the complex $\cC(A)$.

Similarly, for every cell $D\in \cC(A)$ of dimension $\geq 1$ we have
\begin{align*}
    \overline{T}_\cH(D) &= \max\{\overline{T}_\cH(C)\mid C\subset D\}, \text{ and }\\
    \overline{\bft}_\cH(D) &= \max\{\overline{\bft}_\cH(C)\mid C\subset D\}.
\end{align*}
Thus, both the fine and coarse cotypes yield monomial labelings for the complex $\cC(A)$.
\end{prop}

The proof is essentially identical to that of \cite{dochtermann2012tropical}*{Propositions 3.4 and 3.8}.
% , but we include it here for completeness.

% \begin{proof} We first prove the statement for fine types; the proof for coarse types is analogous. Let $C$ and $D$ be cells of $\cC(A)$. By Lemma \ref{lem:type+facts}, we know that $C\subseteq D$ implies that $T_\cH(D)\leq T_\cH(C)$. We will show that if $T_\cH(C)_{ik} = 1$ for some $(i,k)\in \supp A$, then there is a cell $D$ containing $C$ such that $T_\cH(D)_{ik} = 1$.
% Fix $i\in [d]$ such that $T_\cH(C)_{ik} = 1$, and consider the set $Q_i \coloneqq \{{\bf  x}\in \RR^d : x_i \leq x_j \text{ for all } j\neq i\}$.
% The set $Q_i$ is a polyhedron of dimension $d$, and one has that for any $p\in S^{\max}_i(v_k)$, the set $p + Q_i \subseteq S_i(v_k)$.
% Since $\codim(C)\geq 1$, the set $C+Q_i$ must meet the star of $C$, so there must be a cell $D$ containing $C$ that is also contained in $S_i(v_k)$, giving the result.

% A similar argument to that above yields the inequality $T_\cH(D) \geq \min\{T_\cH(C) : C\subset D\}$, proving the claims regarding fine and coarse co-types.
% \end{proof}

\begin{prop}\label{prop: cocellularRes} Let $A$ be a $d \times n$ tropical matrix and let $\cH = \cH(A)$ denote the corresponding tropical hyperplane arrangement in $\troptorus{d}$ with tropical complex $\cC(A)$. Then with labels given by fine (respectively coarse) type, the labeled complex $\cC(A)$ supports a minimal cocellular resolution of the fine (respectively coarse) type ideal $I_{T_\cH}$ (respectively $I_{\bft(\cH)})$.
\end{prop}

\begin{proof}
Proposition \ref{prop: typesLabelCells} implies that the polyhedral complex $\cC(A)$ is colabeled by both fine and coarse types. Combining Theorem \ref{thm: contractible} and a cocellular version of the criterion in \cite{MillerSturmfels:2005}*{Prop. 4.5} gives the result. Minimality follows from the fact that if $C$ is strictly contained in $D$, then $T_\cH(D) < T_\cH(C)$; see e.g. \cite{develin2004tropical}*{Corollary 13}.
\end{proof}

In a similar argument to that above, one can combine Proposition \ref{prop: typesLabelCells}, Proposition~\ref{prop: bounded+contractible}, and \cite{MillerSturmfels:2005}*{Prop. 4.5} to obtain the following result.

\begin{prop}\label{prop: cotypeCellRes} Suppose $A$ is a $d \times n$ tropical matrix, let $\cH = \cH(A)$ denote the corresponding arrangement of tropical hyperplanes, and let $\cB(A)$ be the subcomplex of bounded cells of $\cC(A)$ (if it exists). Then $\cB(A)$, with labels given by the fine (resp. coarse) cotype, supports a minimal cellular resolution of the fine (resp. coarse) cotype ideal.
\end{prop}

\subsection{Resolutions supported on mixed subdivisions of generalized permutohedra}

Recall from Definition \ref{def: generalizedPermutohedron} that the support of a tropical matrix $A$ also gives rise to a generalized permutohedron $P_{B_A}$, given by
\[P_{B_A}= \sum_{j=1}^n \Delta_{N_j(B_A)}\subseteq \RR^d.\]
Furthermore, the entries of the matrix $A$ gives rise to a regular mixed subdivision $\cM(A)$ of $P_{B_A}$.
\noindent
We define a `fine monomial labeling' on $\cM(A)$ (with monomials in $\widetilde{S}$) in the following way. Given a face $F = \sum_{j = 1}^n \Delta_{I_j}$, where $I_j \subseteq N_{B_A}(j)$,  we assign to $F$ the monomial 
\begin{equation}\label{eqn: fineTypePB}
     \prod_{j \in [n]} \; \prod_{i \in I_j} x_{ij}.
\end{equation}
\noindent
Similarly we obtain a `coarse monomial labeling' of $P_{B_A}$ (with monomials in $S$) by specializing the monomials above to the second index $x_{ij} \mapsto x_j$. 
Note that under this coarse labeling a $0$-cell $v$ of $P_{B_A}$ is labeled by a monomial whose exponent vector is precisely the coordinates of $v$. See Figure \ref{fig: mixedSubdiv} for an example. Using Proposition~\ref{prop: mixedSubdiv}, it turns out these monomial labelings are dual to the type labelings of the corresponding tropical complex.
From this we observe the following.

\begin{cor}\label{cor: cellResGenPerm}
For any $d \times n$ tropical matrix $A$, the monomial labeled mixed subdivision $\cM(A)$ of the generalized permutohedron $P_{B_A}$, with labels given by fine (respectively coarse) type, supports a minimal cellular resolution of the ideal $I_{T(\cH)}$ (respectively $I_{\bft(\cH)}$). 
\end{cor}

\begin{proof}
From Proposition \ref{prop: mixedSubdiv} we have that the tropical complex $\cC(A)$ is dual to the mixed subdivision ${\mathcal M}(A)$ of $P_{B_A}$, and in addition the fine (respectively coarse) monomial labeling of a cell in $P_{B_A}$ has the same fine (respectively coarse) type of the corresponding cell in $\cC(A)$.  The result then follows from Proposition \ref{prop: cocellularRes}.
\end{proof}

\begin{cor}\label{cor: coarsemonomials}
For a sufficiently generic $d \times n$ tropical matrix $A$, the coarse type ideal $I_{\bft(\cH)}$ is generated by monomials corresponding to the lattice points of $P_{B_A}$.
\end{cor}
\begin{proof}
Corollary~\ref{cor: cellResGenPerm} gives that the generators are the labels of the 0-dimensional cells of ${\mathcal M}(A)$.
It suffices to show this is all lattice points when $A$ is sufficiently generic, i.e. when $\mathcal{M}(A)$ is a fine mixed subdivision.

Every 0-dimensional cell is a sum $\sum_{j=1}^n \ee_{k_j}$ for some $k_j \in N_j(B_A)$, so is a lattice point.
Conversely, let $\pp$ be a lattice point and $C = \Delta_{I_1} + \cdots + \Delta_{I_n}$ the smallest fine mixed cell containing $\pp$.
By~\cite{postnikov2009permutohedra}*{Lemma 14.12}, we can write $\pp = \sum_{j=1}^n \ee_{k_j}$ where $k_j \in I_j$.
If $C \neq \{\pp\}$, then $\pp$ is contained in the relative interior of $C$, and so we can perturb $\pp$ in any direction in the affine span of $C$.
The mixed subdivision structure implies that there exists some $k, l \in [d]$ such that $\pp \pm (\ee_k - \ee_l) \in C$.
In particular, this means $k, l \in I_s \cap I_t$ for distinct $s, t \in [n]$.
This contradicts $C$ being a fine mixed cell, as $\dim(\Delta_{I_s} + \Delta_{I_t}) \neq \dim(\Delta_{I_s}) + \dim(\Delta_{I_t})$.
Therefore $C = \{\pp\}$, and so is a cell of $\mathcal{M}(A)$.
\end{proof}

%  Because $i_k \in N_{B_A}(k)$, notice that it is not possible for $i_k$ to appear as a subscript on a generator of $I_{T_\cH}$ more than $\deg_{B_A}(i_k)$ times. 
% In particular, after specializing to the first index via $x_{i_k k} \mapsto x_{i_k}$, the set of monomials in $I_{\bft_\cH}$ is exactly equal to the set of generators of $\m^n(\leq u)$ where $u$ is the left degree vector of $B_A$. 
% Applying Corollary \ref{cor: cellResGenPerm} and noting that the fine type ideal is a polarization of the coarse type ideal if and only if the Betti numbers are preserved under specialization (see e.g. \cite[Lemma 6.9]{NR09}) gives the result.

Recall that a minimal free resolution of any finitely generated $\widetilde{S}$-module is unique up to isomorphism. The fact that any regular fine mixed subdivision of $P_{B_A}$ supports a minimal resolution of $I_{T(\cH)}$ allows us to recover the strong equidecomposability of such complexes, generalizing \cite{dochtermann2012tropical}*{Corollary 5}.

\begin{cor}
Suppose $B \subset K_{d,n}$ is a bipartite graph and let $P_{B}$ denote the generalized permutohedron associated to $B$.  Then for any fine mixed subdivision of $P_{B}$, the collection of coarse types ${\bf t}(C)$ for cells $C$ with $\dim C = k$, counted with multiplicities, is the same.
%Then for any fine mixed subdivision of $P_{B}$ the collection of coarse types of cells, counted with multiplicity, is the same.  
In particular, any fine mixed subdivision of $P_{B}$ has the same $f$-vector.
\end{cor}

\subsection{Alexander duality and initial ideals of lattice ideals}
In the case of arrangements of tropical hyperplanes with full support, Block and Yu \cite{block2006tropical} observed that the fine cotype ideal is Alexander dual to a corresponding initial ideal of a certain toric ideal.  Here we establish a similar result for the case of arbitrary $A$, which will be important for our applications to toric edge ideals studied in Section \ref{sec: lattice}.  We first recall some relevant notions.

% One ideal that is central in the study of lattice polytopes is the toric \textit{lattice ideal}.

\begin{definition}[Lattice ideal]\label{def: latticeIdeal} Let $\cA$ be a set of vectors in $\ZZ^n$. Then the \textit{lattice ideal} associated to $\cP$ is the ideal $I_\cA = \ker \varphi_\cA$, where
\begin{align*}
\varphi_\cA: \; \kk[ x_\mathbf{p}\mid \pp\in \cA] & \; \longrightarrow  \;\kk[x_1,\dots, x_n]\\
x_{\pp} & \;\longmapsto \; x_1^{p_1}\cdots x_n^{p_n}.
\end{align*}
\end{definition}

In the case of a root polytope $Q_B$ associated to a bipartite graph $B$,  we have $\cA = \{e_i - e_{\overline j} \mid (i,\overline{j})\in E(B)\}$ and the lattice ideal of $Q_B$ will be the kernel of the map
\begin{align}\label{eq: latticeIdeal}
\begin{split}
\varphi_\cA: \kk[x_{ij} \mid  (i,\overline{j})\in E(B)] \; &\longrightarrow \; \kk[x_1,\dots, x_d, y_1^{-1},\dots, y_n^{-1}]\\
x_{ij}\; &\longmapsto \; x_i y_j^{-1}.
\end{split}
\end{align}

\begin{remark}
Since $Q_B$ is unimodular, it is known that monomial initial ideals of the lattice ideal of $Q_B$ are in bijection with regular triangulations of $Q_B$.  As we have seen, regular triangulations of $Q_B$ are encoded by sufficiently generic tropical hyperplane arrangements ${\mathcal H}(A)$ where $A$ satisfies $B = B_A$.
\end{remark}

\begin{example}\label{ex: latticeIdeal} Let $Q_B$ be the root polytope corresponding to the bipartite graph in our running example (see Figure \ref{fig: tropHAex}), where
\[A= \begin{bmatrix} 0 & 3 & \infty & 0 \\ 0 & 0 & 2 & 3 \\ 0 & \infty & 0 & \infty \end{bmatrix}.\]
\noindent
Then the lattice ideal $L = I_\cA$ of $Q_B$ is the ideal generated by
\begin{equation}\label{eq: latticeIdealLead}
\langle x_{11}x_{22} - \underline{x_{12} x_{21}}, \; \underline{x_{11} x_{24}} - x_{14} x_{21}, \; \underline{x_{12} x_{24}} - x_{14} x_{22}, \; x_{21} x_{33} - \underline{x_{23} x_{31}}\rangle.
\end{equation}
Recall from Notation \ref{not: cellRings} that a $d \times n$ tropical matrix $A$ gives rise to a monomial term order on $\widetilde{S}$ and hence an initial ideal $\inn_A(I)$ for any ideal $I \subset \widetilde{S}$.
The underlined terms correspond to the generators of the leading term ideal of $L$ with respect to the monomial term order on $\widetilde S$ induced by $A$.
\end{example}

\begin{definition}[Crosscut complex]\label{def: crosscut} Let $\cS(A)$ the subdivision of $Q_{B_A}$ induced by $A$. Define the \emph{crosscut complex} $\textrm{CrossCut}(\cS(A))$ to be the unique simplicial complex with the same vertices-in-facet incidences as the polyhedral complex $\cS(A)$.
\end{definition}

\begin{remark} If $\cS(A)$ is a triangulation of $Q_{B_A}$, then $\textrm{Crosscut}(\cS(A)) = \cS(A)$.
\end{remark}

\begin{prop}\label{prop: cotypeIdealAD}

Let $A$ be a $d \times n$ tropical matrix and let $\cH = \cH(A)$ denote the associated arrangement of tropical hyperplanes in $\troptorus{d}$. Let $L$ denote the lattice ideal of $Q_{B_A}$ (or, equivalently, the toric edge ideal of $B_A$). Then  the fine cotype ideal $I_{\overline{T}(\cH)}$ associated to $\cH$ is the Alexander dual of $M(\inn_A(L))$, where $M(\inn_A(L))$ is the largest monomial ideal contained in $\inn_A(L)$.
\end{prop}

\begin{proof}
Let $J$ be the Stanley--Reisner ideal of $\textrm{Crosscut}(\cS(A))$. Then $J$ is the Alexander dual of the fine cotype ideal associated to $\cH$.

It remains to check that $J = M(\inn_A(L))$. Because $\cA$ is unimodular, one has that $\inn_A(L)$ is squarefree and therefore $M(\inn_A(L))$ is radical. The proof of \cite{dochtermann2012tropical}*{Lemma 6.7} applies directly here, so one has
$$
M(\inn_A(L)) = \textrm{Rad}(M(\inn_A(L))) = J
$$
as desired.
\end{proof}

Combining Propositions \ref{prop: cotypeCellRes} and \ref{prop: cotypeIdealAD} gives the following result.  Here for a squarefree monomial ideal $I$ we use $I^\vee$ to denote its \emph{Alexander dual}.

\begin{cor}\label{cor: cellResDual} For any tropical matrix $A$ the bounded complex $\cB(A)$ of the corresponding tropical hyperplane arrangement supports a minimal cellular resolution of the ideal $M(\inn_A(L))^\vee$. In particular, if $\cH(A)$ is sufficiently generic, then $\inn_A(L)$ is a monomial ideal and $B(A)$ supports a minimal free resolution of $(\inn_A(L))^\vee$.
\end{cor}

Moreover, if $A$ is sufficiently generic, one can say more about the resolution of the (fine or coarse) cotype ideal.

\begin{cor}\label{cor: linearRes} If $A$ is sufficiently generic, then $(\inn_A(L))^\vee$ has a linear minimal free resolution.
\end{cor}

\begin{proof} 
If $A$ is sufficiently generic, then $\cS(A)$ is a triangulation of $Q_{B_A}$ corresponding to the Stanley--Reisner complex of the (squarefree) monomial ideal $\inn_A(L)$ \cite{sturmfels1996grobner}*{Theorem 8.3}.
In this case, $\cB(A)$ supports a minimal \emph{linear} free resolution of $(\inn_A(L))^\vee$. This follows from the following observations:
\begin{enumerate}[(i)]
    \item If $A$ is sufficiently generic, then the labels on the vertices of $\cB(A)$ are all of the same degree: each vertex is ``outside of'' the same number of regions as every other vertex.
    \item Let $D$ be a cell in $\cB(A)$. For any cell $C\subset D$ with $\dim C = \dim D - 1$, the cotype $T_{\cH}(C)$ differs from $T_{\cH}(D)$ in exactly one spot if $A$ is sufficiently generic. Therefore, the monomial labeling of the cell $D$ is exactly one degree higher than the monomial labeling of any cell $C\subset D$ such that $\dim C = \dim D - 1$.
\end{enumerate}
\end{proof}

\begin{example}\label{ex: cellRes} Consider the tropical hyperplane arrangement of Figure \ref{fig: tropHAex}. The bounded complex supports a minimal cellular resolution of the fine cotype ideal
$$
I_{\overline T(\cH)} = \langle x_{12}x_{23}x_{24}, \; x_{11}x_{12}x_{23}, \; x_{11} x_{31} x_{12}, \; x_{31} x_{12} x_{24}, \; x_{21} x_{31} x_{24}, \; x_{21} x_{23} x_{24} \rangle.
$$
The Alexander dual of this ideal in the polynomial ring $\kk[x_{ij} \mid (i,j)\in\supp(A)]$ is
$$\langle {x}_{21}{x}_{12}, \; {x}_{31}{x}_{23},\; {x}_{
     11}{x}_{24},\; {x}_{12}{x}_{24} \rangle$$
which is exactly the ideal $\inn_A(L)$ corresponding to the underlined leading terms of the ideal $L$ from (\ref{eq: latticeIdealLead}).
\end{example}

\begin{example}\label{ex: cellResNonGeneric} Now consider the tropical hyperplane arrangement associated to the tropical matrix
$$
A' = \begin{bmatrix}
0 & 0 & \infty & 0 \\
0 & 0 & 2 & 3 \\
0 & \infty & 0 & \infty
\end{bmatrix}.
$$
This arrangement corresponds to the same one as in Figure \ref{fig: tropHAex}, except that now the hyperplane corresponding to $v_2$ lies on top of the hyperplane corresponding to $v_1$, so it is no longer generic. However, the bipartite graph for $A$ and $A'$ coincide, and therefore so do their lattice ideals $L$. In this case, the fine cotype ideal is
$$
I_{\overline{T}(\cH(A'))} = \langle x_{23} x_{24}, \; x_{24} x_{31}, \; x_{11} x_{31} x_{12}, \; x_{11} x_{12} x_{23} \rangle 
$$
and the Alexander dual of the fine cotype ideal has generators
$$
\langle x_{11} x_{24}, \; x_{12} x_{24}, \; x_{23} x_{31} \rangle. 
$$
This coincides with the largest monomial ideal inside of the initial ideal of $L$ in (\ref{eq: latticeIdealLead}) with respect to the term order induced by $A'$, which is
$$
\inn_{A'}(L) = \langle x_{11} x_{22} - x_{12} x_{21}, \; x_{11} x_{24}, \; x_{12} x_{24}, \; x_{23} x_{31} \rangle.
$$
\end{example}

\subsection{Volumes and syzygies}\label{sec: dimension}

Recall that a bipartite graph $B \subset K_{d,n}$ encodes a generalized permutohedron $P_B$ obtained as a sum of simplices, as described in Definition \ref{def: generalizedPermutohedron}. For any tropical matrix $A$ satisfying $B = B_A$ we have from Corollary \ref{cor: cellResGenPerm} that the corresponding mixed subdivision of $P_B$ supports a minimal cellular resolution of the resulting fine type ideal $I_{T(\cH)}$ as well as the coarse type ideal $I_{\bft(\cH)}$.  In particular, the Betti numbers of these ideals can be read off from the $f$-vectors of the underlying subdivision.  In many cases the volume of these cells can also be understood in terms of combinatorial data, and in this section we explore this connection.
\subsubsection{Graphic tropical hyperplane arrangements}\label{subsec: graphic}

The most straightforward application arises in the case of tropical arrangements where the apex of each hyperplane has exactly two finite coordinates.  This inspires the following definition.

\begin{definition}\label{def: graphic}
An arrangement ${\mathcal H}(A)$ of $n$ tropical hyperplanes in $\troptorus{d}$ is \emph{graphic} if each column ${\bf a}_i$ in the underlying $d \times n$ matrix $A = ({\bf a}_1, \dots, {\bf a}_n)$ has exactly two finite coordinates.
\end{definition}

We explain the choice of terminology. First note that a graphic tropical hyperplane arrangement ${\mathcal H}(A)$ in $\troptorus{d}$ consists of Euclidean subspaces that are each dual to a line segment, so that ${\mathcal H}(A)$ is in fact a collection of (ordinary) affine hyperplanes in $(d-1)$-dimensional Euclidean space. In fact if a column ${\bf a}_k \in A$ has finite entries $a_{ik}$ and $a_{jk}$ then the corresponding hyperplane $H(\aa_k)$  consists of the set 
\[\{{\bf x} \in \troptorus{d}\mid x_i-x_j = a_{jk} - a_{ik}\}.\]
\noindent
We can now think of the rows of $A$ as vertices of a graph ${\cG}(A)$, where each column of $A$ describes an edge.  One can then see that the tropical hyperplane arrangement ${\mathcal H}(A)$ corresponds to a perturbation of the classical (central) \emph{graphic hyperplane arrangement} associated to the graph $\cG(A)$.

For a graphic tropical hyperplane arrangement ${\mathcal H}(A)$, the corresponding generalized permutohedron $P_{B_A}$ is the \emph{(graphic) zonotope} ${\mathcal Z}_{\cG(A)}$ given by the sum of line segments of the form $[\ee_i,\ee_j]$, where $i$ and $j$ are the finite entries in some column of $A$ (or in other words some edge in the graph $\cG(A)$).  A generic perturbation of the central graphic hyperplane arrangement  corresponds to a fine mixed subdivision (in this case a \emph{zonotopal tiling}) of ${\mathcal Z}_{\cG(A)}$. In our case the subdivision we obtain is dual to the subdivision of tropical space $\troptorus{d}$ induced by the sufficiently generic tropical hyperplane arrangement ${\mathcal H}(A)$ discussed above.

%\begin{remark}
%We remark that the graph $G(A)$ is an object distinct from the (bipartite) graph $B = B_A \subset K_{d,n}$ that encodes the finite entries of the tropical hyperplane arrangement ${\mathcal H}(A)$.  
%\end{remark}

\begin{example}\label{ex: zonotope}
As an example, we can take $A$ to be the tropical matrix
\[A = \begin{bmatrix} 0 & \infty & \infty & 0 \\ 0 & 1 & \infty & \infty \\ \infty & 2 & 4 & \infty \\ \infty & \infty & 6 & 9 \end{bmatrix}\]
One can see that $\cG(A)$ is a 4-cycle, and the corresponding zonotope ${\mathcal Z}_{\cG(A)}$ is depicted in Figure \ref{fig: zonotope} as a Minkowski sum of edges of the simplex.
That fine mixed subdivision (zonotopal tiling) of ${\mathcal Z}_{\cG(A)}$ induced by $A$ is also depicted, computed via \texttt{Polymake}~\cite{polymake:2000}.
\end{example}

% \begin{figure}[h]
% \begin{center}
% \includegraphics[scale = 0.35]{figures/4cycleZonotope.pdf} 
% \end{center}
% \end{figure}

\begin{figure}[h]
\centering
\includegraphics[width=0.85\textwidth]{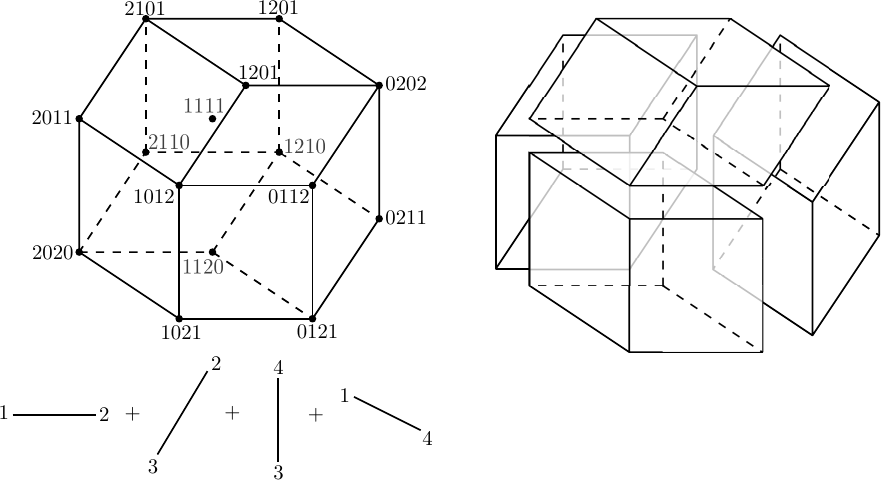}
\caption{The zonotope ${\mathcal Z}_{\cG(A)}$ and its fine mixed subdivision from Example~\ref{ex: zonotope}.}
\label{fig: zonotope}
\end{figure}

Now suppose ${\mathcal H} = {\mathcal H}(A)$ is a sufficiently generic graphic tropical hyperplane arrangement with underlying graph $\cG = \cG(A)$. Combining our results with well-known properties of zonotopes provides enumerative interpretations of certain Betti numbers of our type ideals.

%Corollaries \ref{cor: cellResGenPerm} and \ref{cor: coarsemonomials} imply that any fine zonotopal tiling of ${\mathcal Z}_{G(A)}$ supports a minimal resolution of the type ideals, whose generators correspond to the lattice points of ${\mathcal Z}_{\cG(A)}$.  Combining our results with well-known properties of zonotopes provides enumerative interpretations of other Betti numbers. 

\begin{cor}\label{cor: zonotope}
Suppose ${\mathcal H} = {\mathcal H}(A)$ is a graphic tropical hyperplane arrangement with underlying connected graph $\cG = \cG(A)$. Then the number of generators of the fine type ideal $I_{T({\mathcal H})}$, as well as the coarse type ideal $I_{\bft(\cH)}$, is given by the number of forests in $\cG$.  The rank of the highest syzygy module of both $I_{T({\mathcal H})}$ and $I_{\bft(\cH)}$ is given by the number of spanning trees of $\cG$.
\end{cor}

\begin{proof}

In \cite{postnikov2009permutohedra}*{Proposition 2.4} it is shown that for a connected graph $\cG$ the volume of the graphic zonotope ${\mathcal Z}(\cG)$ equals the number of spanning trees of $\cG$, and also that the number of lattice points of ${\mathcal Z}(\cG)$ equals the number of forests in $\cG$.  Our first claim then follows from the fact that ${\mathcal Z}(\cG)$ supports a minimal resolution of both $I_{T({\mathcal H})}$ and $I_{\bft(\cH)}$, whose generators in each case are in bijection with the lattice points of ${\mathcal Z}(\cG)$ (see Corollaries \ref{cor: cellResGenPerm} and \ref{cor: coarsemonomials}). The second claim follows from the fact that any maximal cell in a zonotopal tiling (in this case a  parallelepiped corresponding to a spanning tree of ${\mathcal G}$) has volume 1.
\end{proof}

\begin{remark}
We observe that the type ideals coming from graphic tropical hyperplane arrangements can be realized as \emph{oriented matroid ideals} studied by Novik, Postnikov, and Sturmfels in \cite{NPS}.
As we have seen, a graphic tropical hyperplane arrangement can be realized as a perturbation of the corresponding central graphic hyperplane arrangement.  The oriented matroid ideal of the resulting affine arrangement can be seen to coincide with the underlying fine type ideal. 

On the other hand, the authors of \cite{NPS} study a different oriented matroid ideal arising from a graph $\cG$.  Namely, they use the affine oriented matroid ${\mathcal A}$ obtained by intersecting the central graphic hyperplane arrangement of $\cG$ by a generic affine hyperplane.  It is proved that the resulting oriented matroid ideal $O_{\mathcal A}$ is  Cohen-Macaulay, and in \cite{NPS} the authors provide formulas and interpretations for the M\"obius invariants and coinvariants of both the complete graph $K_n$ and complete bipartite graph $K_{m,n}$.  

In particular, our methods provide a different way to obtain an oriented matroid ideal from a graph $\cG$.  It is not clear to us how these two ideals relate (note they live in different polynomial rings).
\end{remark}

Graphic tropical hyperplane arrangements also have connections to \emph{$G$-parking function ideals} and other concepts coming from the theory of \emph{chip-firing} on a graph $G$.  Such ideals $M_G$ were first introduced by Postnikov and Shapiro in \cite{postnikovshapiro}, where connections to power ideals were studied. In \cite{DocSan} Dochtermann and Sanyal proved that a minimal cellular resolution of $M_G$ is supported on a certain affine slice of the graphic hyperplane arrangement associated to $G$.  One can recover this ideal as a certain coarse type ideal associated to the underlying graphic tropical hyperplane arrangement, leading to a number of applications and generalizations.  See Section \ref{sec: GP} for more discussion.

%We see that cotype ideals associated to graphical tropical hyperplane arrangements also have interesting algebraic properties.

%\begin{lemma}
%Suppose ${\mathcal H} = {\mathcal H}(A)$ is a graphical tropical hyperplane arrangement with underlying graph ${\cG} = {\cG}(A)$.  Then the corresponding cotype ideal $I_{\overline{T}(\cH)}$ is Cohen-Macaulay and its Cohen-Macaulay type is given by the volume of the zonotope ${\mathcal Z}_{\cG}$.
%\end{lemma}

%\begin{proof}
%Its \emph{Cohen-Macaulay type} (by definition its highest Betti number) is given by the volume of the zonotope, and by results of Novik, Postnikov, and Sturmfels, this value is given by the M\"obius invariant of the underlying graphic matroid.
%\end{proof}

\subsubsection{Computing volumes via Alexander duality}

% As we have seen, geometric properties of graphical zonotopes can be understood in terms of the underyling graph. The volume and number of lattice points of other generalized permutohedra also have rich combinatorial properties.  For instance,
For $P_B$ a generalized permutohedron obtained as a sum of simplices Postnikov \cite{postnikov2009permutohedra}*{Theorem 9.3} provides a formula for the volume of $P_B$ in terms of what he calls \emph{$G$-draconian sequences} (where $G$ = $B$ in our context). Given a sufficiently generic $d\times n$ tropical matrix $A$ with $B_A = B$, one may also obtain the volume of $P_B$ by studying the types associated to the vertices of the tropical complex.
% For a sufficiently generic tropical hyperplane arrangement ${\mathcal H} = {\mathcal H}(A)$ satisfying $B=B_A$, the fact that a regular fine mixed subdivision of $P_B$ supports a minimal resolution of the type ideals also allows us to compute the volume of $P_B$ using tools from computational commutative algebra.

% \begin{definition}[$B$-draconian sequence]
% Let $B \subseteq K_{d,n}$ be a bipartite graph. A sequence of nonnegative integers $(a_1,\dots, a_d)$ is a \emph{$B$-draconian sequence} if $\sum a_i = n-1$ and, for any subset $\{i_1<\dots i_k\}\subseteq [m]$, we have $\abs{I_{i_1}\cup \dots \cup I_{i_k}}\geq a_{i_1}+\dots+ a_{i_k} +1$. 
% \end{definition}

% \begin{theorem}[\cite[Cor 9.4]{postnikov2009permutohedra}] Let $P_B$ be the generalized permutohedron associated to a bipartite graph $B\subset K_{d,n}$. Then the volume of $P_B$ is given by
% $$
% \sum_{(a_1,\dots, a_d)} \frac{y_1^{a_1}}{a_1 !}\dots \frac{y_d^{a_d}}{a_d !}
% $$
% where the sum is over all $B$-draconian sequences $(a_1,\dots, a_d)$.
% \end{theorem}

%(which depend on the underlying bipartite graph $G$, denoted $B$ in our work).
In what follows, recall that a cell in the bounded complex $\cB(A)$ can be described by its type $T$.  Also recall from Definition \ref{def: leftrightDV} that for any fine type $T$ we denote by $\RD(T) = (a_1,\dots, a_n)$ the right degree vector of $T$.

\begin{prop}\label{prop: volumeGenPerm}
Let $A$ be a sufficiently generic $d\times n$ tropical matrix with corresponding tropical hyperplane arrangement $\cH(A)$, and set $B = B_A\subseteq K_{d,n}$. Let $\cS$ be the set of fine types associated to the vertices of the bounded complex $\cB(A)$. 
Then the volume of the associated generalized permutohedron is
$$
\vol P_B = \sum_{T\in \cS} \frac{1}{(\RD(T)_1 - 1)!} \dots \frac{1}{(\RD(T)_n - 1)!}.
$$
\end{prop}

\begin{proof}
Note that every element $T\in \cS$ is in bijection with a full-dimensional fine mixed cell $M_T$ in ${\mathcal M}(A)$. We can read off the volume of $M_T$ in the following way. Let $B_T$ denote its fine type graph and $\RD(B_T) = (a_1, a_2, \dots, a_n)$ its right degree sequence.  Then one can see that 
\[\vol(M_T)= \prod_{i=1}^n \frac{1}{(a_i - 1)!},\]
\noindent
as the volume of $M_T$ is the product of the volumes of the simplices in the Minkowski sum.
Summing over the elements of $\cS$ provides the desired volume of $P_B$.
\end{proof}

Comparing this with Postnikov's formula for volumes of generalized permutohedra (see \cite{postnikov2009permutohedra}*{Theorem 9.3}) yields the following.

\begin{cor}
Let $A$ be a sufficiently generic $d\times n$ tropical matrix with corresponding hyperplane arrangement $\cH(A)$, and set $B = B_A\subseteq K_{d,n}$. Let $\cS$ be the set of fine types associated to the vertices of the bounded complex associated to $\cH$, and let $\one = (1,\dots, 1)\in \ZZ^n$. Then the $B$-draconian sequences correspond to the set
$$
\{\RD(T)-\one \mid T\in \cS \}.
$$
\end{cor}

Given a bipartite graph $B\subseteq K_{d,n}$, this gives a novel method to compute and prove statements about volumes of generalized permutohedra $P_B$:
\begin{enumerate}
    \item Compute the lattice ideal $L$ of the associated root polytope $Q_B$.
    \item Find a sufficiently generic tropical matrix $A$ such that $B_A = B$ and compute $\inn_A L$.
    \item Compute the Alexander dual $(\inn_A L)^\vee = I_{\overline{T}(\cH(A))}$ (see Proposition~\ref{prop: cotypeIdealAD}).
    \item Take the complements of the monomial generators of $I_{\overline{T}(\cH(A))}$ to obtain the set $\cS$ described above.
    \item Compute $P_B$ using the formula in Proposition \ref{prop: volumeGenPerm}.
\end{enumerate}

\begin{example}
Consider the generalized permutohedron defined by $B = K_{3,2}$, so that $P_B$ is the sum of two triangles.  To compute its volume we choose the sufficiently generic tropical matrix $A = \begin{bmatrix} 0 & 0 \\ 0 & 2 \\ 0 & 1 \end{bmatrix}$.  This defines an initial ideal
\[\langle \underline{x_{11}x_{22}} - x_{12}x_{21}, \; \underline{x_{11}x_{32}} - x_{12}x_{31}, \; x_{21}x_{32} - \underline{x_{31}x_{22}} \rangle,\]
\noindent
generated by the underlined monomials. The Alexander dual of this initial ideal is given by the fine cotype ideal
\[I_{\overline{T}(\cH(A))} = \langle x_{11}x_{22}, \; x_{11}x_{31}, \; x_{22}x_{32} \rangle.\]

Each of these generators defines a fine cotype, and the right degree sequences of the corresponding types is given by $(2,2)$, $(1,3)$, $(3,1)$.  This provides cells of volume $1$, $1/2$, $1/2$, implying that $P_B$ has volume $2$.  Note also that the $G$-draconian sequences for the case $G=K_{3,2}$ are given by $(1,1), (0,2)$, and $(2,0)$.
\end{example}

% \section{Edge rings and edge polytopes}\label{sec:edge rings}
\section{Lattice Ideals of Root Polytopes as Toric Edge Ideals}\label{sec: lattice}

In this section we describe a connection between cotype ideals associated to tropical hyperplane arrangements and the class of \emph{toric edge ideals}.  The key observation here is that the lattice ideal of a root polytope $Q_B$ for a \emph{bipartite} graph $B$ is exactly the toric edge ideal of the corresponding graph $B$, see Lemma \ref{lem: bipartite}. With this we are able to describe the generators of $Q_B$ using a well-known result of Villarreal, and describe certain classes of lattice ideals which are well-studied by commutative algebraists. We then use results from Section \ref{sec: cellRes} to show that the regularity of a toric edge ideal of a bipartite graph $B$ coincides with the dimension of a corresponding bounded tropical complex.  This leads to new proofs of existing bounds on homological invariants as well as new results. We begin by establishing notation and recalling some basic results regarding toric edge ideals.
\begin{notation}
Suppose $G = (V,E)$ is a finite simple graph with vertex set $V = \{v_1, \dots, v_n\}$ and edge set $E = \{e_1, \dots, e_k\}$. If $\kk$ is any field, let $\kk[E]$ and $\kk[V]$ denote the polynomial rings in these indeterminates. 
\end{notation}

\begin{definition}\label{def: toric+edge+ideal}
Suppose $G = (V,E)$ is a finite simple graph. The \emph{edge ring} $\kk[G]$ is the $\kk$-algebra generated by the monomials $x_ix_j$ where $\{i,j\} \in E(G)$.

We define a monomial map
\begin{align*}
    \pi:\kk[E] &\rightarrow \kk[V] \\
    e_j &\mapsto v_{j_1}v_{j_2}
\end{align*}
where $e_j = \{v_{j_1}, v_{j_2}\}$.  The \emph{toric edge ideal} of $G$ is defined to be the ideal $J_G = \ker \pi$, so that 
\[\kk[G] \cong \kk[E]/J_G.\]
\end{definition}

An effective way to study the edge ring of a graph $G$ is via a certain polytope associated to $G$ called the \textit{edge polytope}. 

\begin{definition}\label{def: edgePolytope}
Let $G$ be a finite simple graph with vertex set $V = [n]$. For an edge $e =\{i,j\}$ in $G$ define $\rho(e) \in {\mathbb R}^n$ by $\rho(e) = {\bf e}_i + {\bf e}_j$, where ${\bf e}_i$ is the $i$th unit coordinate vector in ${\mathbb R}^n$. The \emph{edge polytope} ${\mathcal P}_G$ is the convex hull of all such vectors
\[{\mathcal P}_G = \conv\{\rho(e) \mid e \in E(G)\} \subset {\mathbb R}^n.\]
\end{definition}

One can see that for any graph $G$, the toric edge ideal $J_G$ is recovered as the toric ideal of the edge polytope ${\mathcal P}_G$. For this note that ${\mathcal P}_G \cap {\mathbb Z}^n = \{\rho(e): e \in E(G)\}$ and also that the vertices of ${\mathcal P}_G$ coincide with ${\mathcal P}_G \cap {\mathbb Z}^n$ (see \cite{HHOBinomial}*{Lemma 5.2}).  For the case that $G = B \subset K_{d,n}$ is \emph{bipartite} the ideal $J_G$ can be recovered in another way. The proof of the following lemma is straightforward.

\begin{lemma}\label{lem: bipartite}
If $B \subset K_{d,n}$ is a bipartite graph, then the toric edge ideal $J_B$ is equal to the lattice ideal of the root polytope $Q_B$.  
\end{lemma}

% \begin{proof}
% In the case when $B$ is a bipartite graph, an exponent of $-1$ appears on every $y_i$ variable in the image of $\varphi_\cA$ in the map (\ref{eq: latticeIdeal}) and not on any $x_i$ variable. Therefore, the kernel of $\varphi$ is unchanged whether the vertices in $\cA$ are of the form ${\bf e}_i - {\bf e}_{\overline{j}}$ or ${\bf e}_i + {\bf e}_{\overline{j}}$. In particular, the lattice ideal of the root polytope $Q_B$ and of the edge polytope $\cP_B$ coincide.
% \end{proof}

% Recall that the \emph{Krull dimension} of a ring $R$ is the length of the longest chain of prime ideals in $R$. 
Recall that the Krull dimension of a toric ring associated to a matrix $M$ is given by $\rank(M)$ 
\cite{HHOBinomial}*{Proposition 3.1}.  The toric edge ring $\kk[B]$ of a bipartite graph $B \subset K_{d,n}$ is the toric edge ring of its incidence matrix, which is known to have rank $d + n -1$.  Hence we get the following.

\begin{cor}\label{cor: krulldim}
For a bipartite graph $B \subset K_{d,n}$ the Krull dimension of $\kk[B]$ is given by
\[\dim(\kk[B]) = d+n-1.\]
\end{cor}

Toric edge rings are a particular example of a construction called \textit{special fiber rings}, which are blow-up algebras that give insight into the \textit{Rees algebra} of an ideal but which are often more tractable to compute.
% The Rees algebra has been the focus of many commutative algebra projects since the late 1950’s; see for instance, \cite{Vas} for an overview of applications of the Rees algebra.
In general, there is interest in bounding the regularity of the Rees algebra of an ideal $I$ because this provides information about the regularity of powers of $I$.
Finding the defining ideal of the special fiber ring is a challenging problem and is still open for many classes of ideals. 
However, a famous result of Villarreal \cite{villarreal1995rees} tells us that we can write down a minimal generating set for the special fiber of the edge ideal of a bipartite graph by studying its minimal cycles.

The following proposition gives an explicit generating set for the toric edge ideal of a graph.

\begin{prop}\label{prop: villarreal}\cite{villarreal1995rees}*{Proposition 3.1} Let $G$ be a finite simple graph.
For any even closed walk $\ww = (i_1,j_1,\dots, i_k,j_k)$
in $G$, define a binomial
$$
f_\ww \coloneqq
% x_{i_1 j_1} x_{i_3 j_3} \dots x_{i_{k-1} j_{k-1}} - x_{i_2 j_2} x_{i_3 j_3} \dots x_{i_k j_k}.
x_{i_1 j_1} x_{i_2 j_2} \dots x_{i_{k} j_{k}} - x_{j_1 i_2} x_{j_2 i_3} \dots x_{j_{k-1} i_k} x_{j_k i_1}.
$$
Then the toric edge ideal $J_G$ is generated by all $f_\ww$ where $\ww$ is an even closed walk in $G$. In particular, if $G$ is bipartite then $J_G$ is generated by all $f_\ww$ where $\ww$ is a minimal even cycle in $G$.
\end{prop}

% \begin{notation}\label{not: rings} For some fixed integers $n$ and $d$, let $S = k[x_1,\dots, x_d]$ be a polynomial over a field $k$. If $A$ is a matrix in $\RR^{d\times n}_\infty$, set $\tilde S_A = k[x_{i,j}\mid (i,j)\in\Supp(A)\}$.
% Denote by $X_A$ the $d\times n$ matrix with entries
% \begin{align*}
%     X_A (i,j) = \begin{cases}
%     x_{i,j} &\text{ if } (i,j)\in\Supp(A)\\
%     0 &\text{ otherwise.}
%     \end{cases}
% \end{align*}
% Let the weight of the variable $x_{ij}$ in $\tilde S_A$ be $v_{ij}$ and the weight of a monomial $\xx^\aa = \prod x_{ij}^{a_{ij}}\in \tilde S_A$ be $\sum a_{ij} v_{ij}$.
% The initial form $\inn_A(f)$ of a polynomial $f = \sum c_i \xx^{\aa_i}$ is defined to be the sum of terms $c_i \xx^\aa_i$ such that $\xx^\aa_i$ has maximal weight.

% Let $A\in \RR^{d\times n}_\infty$ be a tropical matrix in which no column has all entries equal to $\infty$. Set $G_A$ to be the bipartite graph on with edge set
% $$
% E(G_A) = \{(i,\bar{j})\in [d]\times [n] \mid (i,j)\in\Supp(A)\}.
% $$
% Let $Q_{G_A}$ be the root polytope associated to $G_A$, and denote by $\cA$ the integer matrix whose column vectors are the vertices of $Q_{G_A}$. When it will not cause confusion, we will often write $Q_A$ instead of $Q_{G_A}$.
% \end{notation}

% \begin{definition}
% The toric lattice ideal $L$ of $Q_{G_A}$ is defined by the kernel of the map
% \begin{align*}
% \ZZ^{dn} &\ra \ZZ^{d+n}\\
% x_{ij} &\mapsto (i,\overline j).
% \end{align*}
% \end{definition}

A bipartite graph $B$ is \emph{chordal bipartite} if it contains no induced cycles of length longer than four. In \cite{ohsugihibikoszul1999}, the following algebraic characterization of the toric edge rings of such graphs is given.

\begin{theorem}\cite{ohsugihibikoszul1999} \label{thm: chordalbip}
Suppose $B$ is a finite connected bipartite graph. Then the following are equivalent.

\begin{enumerate}
    
    \item
    $B$ is chordal bipartite;
    \item
    The toric edge ideal $J_B$ admits a quadratic binomial  Gr\"obner basis;
    \item
    The edge ring $\kk[B]$ is Koszul;
    \item
    The toric edge ideal $J_B$ is generated by quadratic binomials.
    \end{enumerate}

\end{theorem}    

Toric edge ideals of chordal bipartite graphs also have a connection to certain \emph{determinantal} ideals.  For this recall that for a $d \times n$ tropical matrix $A$ we use $X_A$ to denote the matrix associated to the finite entries $A$, see Notation \ref{not: cellRings}.  We will let $L_A$ denote the lattice ideal associated to the root polytope $Q_{B_A}$.
Using the language established in Laura Ballard's thesis \cite{ballard2021properties}, a $2$-minor of $X_A$ is called a \emph{distinguished minor} if it is a $2$-minor involving only nonzero entries of $X_A$. 
The following result also appeared in \cite{ballard2021properties}.

\begin{prop}\label{prop: ladder-like} Suppose $A$ is a tropical matrix with bipartite graph $B_A$. If $B_A$ is chordal bipartite, then the lattice ideal associated to the root polytope $Q_{B_A}$ is generated by all distinguished minors of $X_A$.
\end{prop}

% \begin{proof} Let $L_A$ denote the lattice ideal associated to the root polytope $Q_{B_A}$. Applying Proposition~\ref{prop: villarreal}, if $B_A$ is chordal bipartite, then $L_A = J_{B_A}$ is generated by all $4$-cycles in $B_A$. The distinguished minors of $B_A$ exactly correspond to these $4$-cycles.
% \end{proof}

%\begin{obs} 
Observe that in the case that the tropical matrix $A$ has full support, $L_A$ is generated by all $2$-minors of a generic $d\times n$ matrix as in \cite{block2006tropical}.
%\end{obs}

\begin{example}[Ladder determinantal ideals, Ferrers graphs]\label{ex: ladder} Suppose the tropical matrix $A$ satisfies the following property:
$$
\text{ if $(i,j),(h,k)\in \Supp(A)$ and $i\leq h, j\leq k$, then $(i,k),(h,j)\in\Supp(A)$}.
$$
Then $B_A$ is chordal bipartite and $L_A$ is called a \emph{ladder determinantal ideal}.
Ladder determinantal ideals have been well-studied in commutative algebra and algebraic geometry; see work of Conca \cites{conca1995, conca1996gorenstein}, Gorla \cite{gorla2007mixed}, and more recently Rajchgot, Robichaux, and Weigandt \cite{rajchgot2022castelnuovo}.
One particular case of this example is the case of toric ideals associated to \textit{Ferrers graphs}. A Ferrers graph $F$ is a bipartite graph such that if $(i,j)$ is an edge of $F$, then so is $(p,q)$ for $1\leq p\leq i$ and $1\leq q \leq j$. In addition, $(1, n)$ and $(d, 1)$ are required to be edges of $F$. The monomial edge ideals and toric edge ideals of these graphs were studied by Corso and Nagel in \cites{corso2008specializations, corso2009}.
\end{example}

\begin{example}\label{ex: latticeIdealLadder}
We return to our running Example \ref{ex: running}, whose corresponding tropical hyperplane arrangement is depicted in Figure \ref{fig: tropHAex}. In this case the matrix $X_A$ is
\[X_A = \begin{bmatrix} x_{11} & x_{12} & 0 & x_{14}\\ x_{21} & x_{22} & x_{23} & x_{24} \\ x_{31} & 0 & x_{33} & 0 \end{bmatrix}.\]
The lattice ideal $L_A$ is generated by all the distinguished $2$-minors of this matrix:
$$
\ker\psi = \langle x_{11} x_{22} - x_{12} x_{21}, x_{21} x_{33} - x_{23} x_{31}, x_{11} x_{24} - x_{14} x_{21}, x_{12}x_{24} - x_{14}x_{22} \rangle.$$
Note that one can realize $L_A$ as a determinantal ideal by rearranging the columns of $A$.
\end{example}

\begin{example}\label{ex: latticeIdealNonLadder}
For an example of a lattice ideal that is not a ladder determinantal ideal, consider the tropical hyperplane arrangement corresponding to the tropical matrix
\[A = \begin{bmatrix} 0 & 1 & \infty \\ 0 & \infty & 0 \\ \infty & 0 & 1 \end{bmatrix}.\]
Then the corresponding lattice ideal is
$
\langle x_{11} x_{23} x_{32} - x_{21} x_{33} x_{12}\rangle .
$
\end{example}

In the case that $B$ has no cycles, we easily recover a result of Postnikov \cite{postnikov2009permutohedra}*{Lemma 12.5}.

\begin{cor} Suppose that $B = B_A$ is a forest. Then $Q_B$ is a simplex.
\end{cor}

\begin{proof}
If $B$ is a forest, then it contains no cycles and the generating set of $L_A$ is empty. Therefore, $Q_B$ has no nontrivial triangulations and must itself be a simplex.
\end{proof}

\subsection{Krull dimension of type ideals} \label{sec: Krull}

Most of our results involve applying methods from tropical geometry to say something about toric edge rings, this is the focus of Section~\ref{sec: edge+ideal+results}. We can also apply results in the other direction to understand algebraic properties of type ideals.

\begin{prop}
Suppose $d, n \geq 2$. Let $A$ be a sufficiently generic $d \times n$ tropical matrix with corresponding bipartite graph $B = B_A$ and tropical hyperplane arrangement ${\mathcal H}(A)$. Then the Krull dimension of the fine cotype ideal satisfies
\
\[\dim(I_{\overline{T}(\cH)}) \leq |E(B)| - 2,\]
with equality if and only if $B$ is bipartite chordal.
\end{prop}

\begin{proof}
First recall that $I_{\overline{T}(\cH)}$ is a squarefree monomial ideal in the polynomial ring $\widetilde{S}$ with $|E(B)|$ indeterminates (see Notation \ref{not: cellRings} and Definition \ref{def: typeIdeals}).  Let $\Delta_{\mathcal H}$ denote the associated Stanley-Reisner simplicial complex, so that $\dim(I_{\overline{T}(\cH)}) = \dim (\Delta_{\mathcal H})$. 
To see that $\dim(I_{\overline{T}(\cH)}) \leq |E(B)| - 2$, note that since $d, n \geq 2$ we have that a generic element in $\troptorus{d}$ must sit in some region of 2 distinct tropical hyperplanes.  Hence $\Delta_{\mathcal H}$ has a missing edge, from which it follows that $\dim(\Delta_{\mathcal H}) \leq 2$.  This proves the first claim.

For the case of equality, recall from Theorem \ref{thm: chordalbip} that $B$ is bipartite chordal if and only if $J_B$ admits a Gr\"obner basis consisting of quadratic binomials. The term order underlying such a basis corresponds (via its weight vector) to a regular triangulation of the root polytope $Q_B$. Let $A$ be the corresponding $d \times n$ tropical matrix (with finite entries corresponding to the height of the lifted vertices of $Q_B$, and with infinite coordinates otherwise).  For this term order $\preceq_A$ we have that the initial ideal $\inn_{\preceq_A}(J_B)$ is a quadratic monomial ideal.  If we let $\Gamma_A$ denote the Stanley-Reisner complex of this ideal, it then follows that the minimal nonfaces of $\Gamma_A$ have cardinality 2 (so that $\Gamma_A$ is a so-called \emph{flag} complex). 

This implies that the facets of the Alexander dual complex $\Gamma_A^\vee$ has facets of cardinality $|E(B)| - 2$ (recall that the vertex set of $\Delta_{\mathcal H}$ is given by $E(B)$).  But from Proposition \ref{prop: cotypeIdealAD} we gave that $\Delta_{\mathcal H} = \Gamma_A^\vee$, and the claim follows.
\end{proof}

% \subsection{Examples and applications}

% K\'alm\'an and Postnikov  prove that the interior polynomial of $H$ is equivalent to the Ehrhart polynomial of $Q_H$, which in turn is equivalent to the $h$-vector of any triangulation of $Q_H$. It
% follows that the interior polynomials of H and its transpose $H = (E, V )$ agree.

\section{Homological properties of bipartite toric edge ideals} \label{sec: edge+ideal+results}

In this section we use results from above to investigate homological invariants of toric edge ideals of bipartite graphs.  We take advantage of the fact that for a bipartite graph $B$, the lattice ideal of the root polytope $Q_B$ is exactly the toric edge ideal of $B$. We can then  utilize results of Section \ref{sec: cellRes}. We first define the invariants that we will be interested in.

\begin{definition}\label{def: regularity} Suppose $R = \kk[x_1, \dots, x_m]$ and $M$ is a graded finitely generated $R$-module. The \textit{(Castelnuovo-Mumford) regularity} and \textit{projective dimension} of $M$ are given by the values
\begin{align*}
\reg_R(M) &= \max\{j \mid \beta_{i,i+j}(M) \neq 0 \text{ for some } i\}, \\
\pdim_R(M) &= \max\{i \mid \beta_{i,i+j}(M)\neq 0 \text{ for some } j\},
\end{align*}
where $\beta_{i,k}(M)$ are the Betti numbers of $M$.
%corresponds to the number of summands in a minimal free resolution of $M$ in homological degree $i$ of the form $R(-k)$.
\end{definition}

% Regularity is a generally highly nontrivial invariant that characterizes cohomological vanishing of the associated coherent sheaf. 
Regularity corresponds to the highest degree that appears in any of the matrices in a minimal free resolution, and projective dimension measures the number of steps needed in constructing the resolution. Hence one may think of these invariants as measuring the ``complexity'' of the module.

\begin{definition}\label{def: linearRes} A finitely generated graded $R$-module $M$ admits a \textit{$q$-linear free resolution} if its Betti numbers satisfy $\beta_{i,j}(M) = 0$ for each $1 \leq i \leq \pdim(M)$ and for each $j \neq q + i -1$.
%it is generated by forms of degree $q$ and all the elements of the matrices of the maps in the minimal free resolution of $M$ are linear forms.
\end{definition}

For the case of $M = R/I$ this means that $I$ is generated by homogeneous polynomials of degree $q$ and $\reg(R/I) = q-1$.  
It turns out that regularity and projective dimension of $R/I$ for a squarefree monomial ideal $I$ are related via Alexander duality, and a key tool for us will be the following result of Terai.

\begin{theorem}\cite{terai1997generalization}*{Theorem 0.2}\label{thm: projreg}
For $I \subset R$ a squarefree monomial ideal, the projective dimension and regularity satisfy
\[\pdim(R/I) = \reg(I^\vee).\]
\end{theorem}

In particular, this implies that $R/I$ has a linear resolution if and only if $R/I^\vee$ is Cohen-Macaulay, a fact first observed by Eagon and Reiner in \cite{eagon1998resolutions}. We will also need the following recent result of Conca and Varbaro.

\begin{theorem}[\cite{conca2020square}]\label{thm: concaVarbaro} Let $I$ be a homogeneous ideal of a polynomial ring $R$ such that $\inn(I)$ is a square-free monomial ideal for some term order. Denote by $h^{ij}(M)$ the dimension of the degree $j$ component of its $i$th local cohomology module $H^i_\m(M)$ supported on the maximal ideal $\m$. 
Then
$$
h^{ij}(R/I) = h^{ij}(R/\inn(I)).
$$
In particular, the depth, regularity, and projective dimension of $S/I$ and $S/\inn(I)$ coincide.
\end{theorem}

Recall that an $R$-module $M$ is \emph{Cohen-Macaulay} if $\depth(M) = \dim(M)$, where $\dim(M)$ is the Krull dimension of $M$. If $B$ is a bipartite graph, it is known that the toric edge ring $\kk[B]$ is Cohen-Macaulay. This follows by combining the fact that toric rings of bipartite graphs are \emph{normal} (by a result of Villarreal \cite{villarreal1995rees}), along with a celebrated result of Hochster \cite{hochster1972} that says all normal semigroup rings are Cohen-Macaulay.   
%In fact in \cite{ohsugihibinormal} Ohsugi and Hibi prove that the edge polytope of a graph $G$ is normal if and only if $G$ satisfies the `odd cycle condition', that is if for any two minimal odd cycles $C$ and $C^\prime$ in $G$, either $C$ and $C^\prime$ have a common vertex or there exists and edge joining a vertex of $C$ with a vertex of $C^\prime$. In particular $\kk[G]$ is Cohen-Macaulay whenever $G$ is bipartite.
We can easily recover this property by applying our results from previous sections.

\begin{prop}\label{prop: CM} For any connected bipartite graph $B$ the toric edge ideal $J_B$ is Cohen-Macaulay.
\end{prop}

\begin{proof}
The root polytope $Q_B$ corresponds to a (sufficiently generic) hyperplane arrangement $\cH(A)$ where $A$ is a tropical matrix satisfying $B = B_A$.  By Corollary \ref{cor: linearRes} the bounded complex $\cB(A)$ supports a linear minimal free resolution of the Alexander dual of $\inn_A(J_B)$. It follows from Theorem \ref{thm: projreg} that $\tilde{S}/\inn_A(J_B)$ is Cohen-Macaulay. Applying Theorem \ref{thm: concaVarbaro}, one has $\depth(\tilde{S}/J_B) = \depth(\tilde{S}/\inn_A(J_B))$, and the result follows.
\end{proof}

% We will now use similar ideas to bound the \textit{regularity} of the toric edge ideal $J_G$ and the quotient ring $k[E]/J_G$. We begin by recalling the definition of regularity.

% According to \cite{herzog2020matching}, there are two standard methods to compute the regularity of a toric ring $A$. The first is to compute the multigraded Betti numbers of $A$ by employing the \emph{squarefree divisor complex}. In some cases, this provides lower bounds for regularity but is difficult to use in general. 
% The second method, which applies if $A$ is Cohen-Macaulay, is to compute the $a$-invariant of $A$.  If $A$ is also normal then one can obtain the $a$-invariant of $A$ via a computation of its canonical module.

We will see that our knowledge of a minimal resolution of the fine type ideal allows for a novel way to compute regularity of toric edge rings of bipartite graphs. The main tool in our study is the following result, which provides a tropical geometric interpretation for the regularity of a toric edge ring.

%which tells us that for a bipartite graph $B$ we can precisely determine the regularity of the toric edge ring $K[B]$ via tropical methods.

\begin{theorem}\label{thm: tropicalCplxRegularity}
Suppose $A$ is a sufficiently generic $d \times n$ tropical matrix.  Let $\cH = \cH(A)$ denote the corresponding arrangement of $n$ tropical hyperplanes in $\RR^d/\RR\one$ with bounded complex $\cB(A)$. If $B = B_A$ is the bipartite graph corresponding to finite entries of $A$ we have
\[\reg(\kk[B]) = \dim(\cB(A)).\]
\end{theorem}

\begin{proof}
By Lemma \ref{lem: bipartite}, the toric edge ideal $J_B$ is exactly the lattice ideal of the root polytope $Q_B$. 
By Corollary \ref{cor: cellResDual}, the bounded complex $\cB(A)$ supports a minimal free resolution of $\tilde{S}/I$, where $I$ is the Alexander dual of the ideal $\inn_A(J_B)$. Hence we have that $\pdim(\tilde{S}/I) = \dim(\cB(A))+1$. 
From Theorem \ref{thm: projreg} we also have $\reg(\inn_A(J_B)) = \pdim(\tilde{S}/I)$. 
Moreover, $\inn_A(J_B)$ is a \textit{squarefree} initial ideal of $J_B$, so the result follows from Theorem \ref{thm: concaVarbaro} above and the fact that $\reg(\kk[B]) = \reg(J_B)-1$.
\end{proof}

While Theorem~\ref{thm: tropicalCplxRegularity} gives a precise formula for the regularity of $J_B$, a priori it depends on the matrix $A$.
Proposition~\ref{prop: bounded+complex+dimension} gives a sufficiently generic $A$ such that the regularity can be characterized in purely in terms of graph properties of $B$.  In what follows recall that $\lambda(B)$ denotes $\emph{recession connectivity}$ of a bipartite graph $B$, see Definition \ref{def: recession}.

\begin{cor}\label{cor: graph+regularity}
Let $B$ be a finite simple connected bipartite graph, then $\reg(\kk[B]) = \lambda(B)-1$.
\end{cor}

As the regularity of $\kk[B]$ is invariant under any choice of sufficiently generic $A$ satisfying $B = B_A$, we get that the dimension of the bounded complex is also invariant of $B$.

\begin{cor}\label{cor: bounded+complex+equality}
Let $A$ be a sufficiently generic tropical matrix with corresponding tropical hyperplane arrangement ${\mathcal H}(A)$ and bipartite graph $B_A$.
Then $\dim(\cB(A)) = \lambda(B_A) - 1$.
\end{cor}

\subsection{Warm-up: Complete bipartite graphs}

We first study the case of toric edge ideals $J_B$ in the case that $B = K_{d,n}$ is a complete bipartite graph.  These ideals are very well understood but we will see that the tropical perspective lends further insight into their homological properties.  We first note that for $B =K_{d,n}$ the corresponding arrangement consists of $n$ tropical hyperplanes in $\troptorus{d}$, each of which has full support.  The algebraic properties of the corresponding type and cotype ideals were studied in \cite{dochtermann2012tropical}.  

\begin{prop}\label{prop: completebipartite}
Suppose $A$ is a sufficiently generic tropical matrix with only finite entries, and let ${\mathcal H} = {\mathcal H}(A)$ denote the corresponding arrangement of $n$ tropical hyperplanes in $\troptorus{d}$. Then the fine cotype ideal $I_{\overline{T}(\cH)}$ satisfies
\begin{equation}
\begin{split}
\reg (I_{\overline{T}(\cH)})  & = (d-1)(n-1), \\
\pdim (I_{\overline{T}(\cH)}) & = \min\{d,n\} - 1.
\end{split}
\end{equation}
\end{prop}

\begin{proof}
Recall from Proposition \ref{prop: cotypeCellRes} that a minimal resolution of the fine type cotype ideal $I_{\overline{T}(\cH)}$ is supported on the bounded tropical complex $\cB(A)$.  From Proposition \ref{prop: bounded+complex+dimension} and Lemma \ref{lem: recbipartite} we see that $\dim(\cB_A) =min\{d,n\} -1$ and hence the formula for projective dimension follows.

For the regularity claim, note that since $A$ is sufficiently generic we have from Corollary \ref{cor: linearRes} that the fine cotype ideal $I_{\overline{T}(\cH)}$ has a linear resolution. 
Hence the regularity of $I_{\overline{T}(\cH)}$ is equal to the (common) degree of its generators.  
Note that if $v$ is the apex of any hyperplane of ${\mathcal H}(A)$ then it is not contained in $d-1$ regions of the remaining $n-1$ tropical hyperplanes.  But $v$ corresponds to a $0$-cell of the bounded complex, and hence a generator of $I_{\overline{T}(\cH)}$ of degree $(n-1)(d-1)$, as desired.
\end{proof}

Applying Proposition \ref{prop: completebipartite} and Theorem \ref{thm: tropicalCplxRegularity}, we recover a result from the literature regarding the toric edge rings of complete bipartite graphs (see for instance \cite{ha2019}*{Lemma 3.10}).
\begin{cor}
For $B = K_{d,n}$ a complete bipartite graph, the toric edge ring $\kk[B]$ satisfies
\begin{equation*}
\begin{split}
\reg (\kk[B]) = \min\{d,n\} - 1, \\
\pdim (\kk[B]) = (d-1)(n-1).
\end{split}
\end{equation*}
\end{cor}

Note that in this case we have $\reg(\kk[B]) = \mat(B) - 1$, where $B$ is the \emph{matching number} of $B$. 
We discuss this connection in futher detail below.

\subsection{Regularity bounds for general bipartite graphs}

We next consider toric edge ideals of arbitrary connected bipartite graphs.  Most of our results will be applications of Proposition \ref{thm: tropicalCplxRegularity}, where the regularity of a toric edge ring is described in terms of the dimension of a certain tropical complex. Our first observation is subgraph monotonicity in the regularity of toric edge rings of bipartite graphs.

\begin{theorem}\label{thm: subgraph}
Suppose $B \subseteq K_{d,n}$ is a connected bipartite graph and let $B^\prime \subseteq B$ be a connected subgraph.  Then we have
\[\reg(\kk[B^\prime]) \leq \reg(\kk[B]).\]
\end{theorem}

\begin{proof}
This follows directly from Lemma \ref{lem: lambda+monotone} and Corollary \ref{cor: graph+regularity}.
\end{proof}

In \cite{ha2019} the authors consider a similar question in the context of arbitrary (not necessarily bipartite) graphs.  They prove in \cite{ha2019}*{Theorem 3.7} that if $G$ is any graph with \emph{induced} subgraph $G^\prime$ then we have $\reg(\kk[G^\prime]) \leq \reg(\kk[G])$. We note that the induced property is necessary here: for instance if $G^\prime$ is the graph on $7$ vertices consisting of a pair of triangles connected by a path of length $2$ (for a total of 8 edges) then we have $\reg(\kk[G^\prime]) = 3$, whereas $\reg(\kk[G]) = 2$ for any graph on $7$ vertices with 19 edges.  We thank Tina O'Keefe for providing us with this example. In this example we see that $\kk[G^\prime]$ is not normal. Recall that the toric edge ring of any bipartite graph is normal, and hence a natural question that arises in the following.

\begin{question}\label{question: regularity-subgraphs}
Suppose $G$ is a connected graph with connected subgraph $G^\prime$ such that both $\kk[G]$ and $\kk[G^\prime]$ are normal.  Is it true that $\reg(\kk[G^\prime]) \leq \reg(\kk[G])$?
\end{question}

After posting a version of this paper to the arXiv, Akiyoshi Tsuchiya provided a positive answer to our question.  Here we include a proof based on his ideas. 
%\begin{itemize}
%    \item First note that the edge polytope $\cP_G$ of a finite connected graph $G$ is `spanning' (see the proof \cite[Corollary 3.4]{hibi3linear2019}).

%\item
%It is known that a spanning polytope is normal if and only if the polytope has the integer decomposition property (IDP).

%\item
%A polytope has IDP if and only if the toric ring $\kk[G]$ coincides with the Ehrhart ring.

%\item
%If a polytope has IDP, then $\reg(\kk[P]) = \deg(h^*(P))$, where $h^*(P)$ is the $h^*$-polynomial of P.

%\item
%From Stanley's Monotonicity Theorem, if $P^\prime$ is a subpolytope of $P$, then $\deg(h^*(P^\prime)) \leq deg(h^*(P))$, which implies the desired inequality.
%\end{itemize}

\begin{proof}
    As $\kk[G]$ is normal, it follows that the Hilbert function of $\kk[G]$ is equal to the Ehrhart function of the edge polytope $\cP_G$, see~\cite{HHOBinomial}*{Lemma 4.22(b)}.
    Explicitly, the Hilbert series of $\kk[G]$ is given by
    \[
    H(\kk[G], t) = \frac{h_0^* + h_1^*t + \dots + h_d^*t^d}{(1-t)^d}
    \]
    where $\dim(\cP_G) = d-1$ and $(h_0^*, h_1^*, \dots, h_d^*) \in \ZZ_{\geq 0}^{d+1}$ is the $h^*$-vector of $\cP_G$.
    Let $(h_0^{*'}, h_1^{*'}, \dots, h_d^{*'})$ be the $h^{*}$-vector of $\cP_{G'}$, note that $\cP_{G'}$ is a subpolytope of $\cP_{G}$.
    Therefore by Stanley's monotonicity theorem~\cite{stanley1993}, we have $h_i^{*'} \leq h_i^*$ for all $0\leq i \leq d$.
    Since $\kk[G]$ is normal, we have from \cite{BenVar}*{Lemma 2.5} that $\reg(\kk[G]) = \max\{s \mid h_s^* \neq 0\}$ . The result follows.
\end{proof}

We next use tropical methods to establish new regularity bounds on toric edge rings.

\begin{theorem}\label{thm: regbound}
Suppose $B$ is a bipartite graph with bipartition $V=\{v_1, \dots, v_n\} \cup \{w_1, \dots, w_d\}$ for $n, d \geq 2$. Let $r= |\{v_i:\deg(v_i) = 1\}|$ and $s = |\{w_j: \deg(w_j) = 1\}|$.  Then
\[ \reg(\kk[B]) \leq \min\{n-r, d-s\}-1.\]
\end{theorem}

\begin{proof}
Let $A$ be a sufficiently generic tropical matrix satisfying $B = B_A$.  Let ${\mathcal H}(A)$ denote the corresponding arrangement of tropical hyperplanes with bounded complex $\cB(A)$. By Theorem \ref{thm: tropicalCplxRegularity}, it suffices to show that $\dim(\cB(A)) \leq \min\{n-r,d-s\} -1$.
% Suppose $G$ is a bipartite graph giving rise to a $d\times n$ matrix $X$ as in Notation \ref{not: bipartiteGraphMons}.
% Choose a sufficiently generic tropical matrix $A = (v_1, \dots, v_d)$ such that $X = X_A$, and let ${\mathcal H}(A)$ denote the tropical hyperplane arrangement that it defines.  
% The choice of $A$ defines a term order that gives rise to a square-free initial ideal $\inn_A(J_G)$. 
By Proposition~\ref{prop: bounded+duality}, the bounded complex is a subcomplex of both $\cH(A)$ and $\cH(A^T)$, where $B^T := B_{A^T}$ is just $B$ where the vertices have swapped sides.  
The degree 1 vertices of $B$ and $B^T$ correspond to hyperplanes comprised of a single unbounded sector and hence cannot contribute to the dimension of the bounded complex.
Hence the dimension of the bounded complex is at most $\min\{n-r, d-s\} -1$.
\end{proof}

We note that Theorem \ref{thm: regbound} was established for the case of \emph{chordal} bipartite graphs by Biermann, O'Keefe, and Van Tuyl in \cite{BieOkeVan17}.
An alternative proof can be given via recession graphs: we observe that if $\bddgraph{S}{B}$ is strongly connected then the subgraph $S \subseteq B$ must contain all leaves of $B$.
Moreover, these leaves cannot form their own connected component, and hence removing them does not decrease recession connectivity.
Note that we can repeatedly apply this result to remove certain trees from the graph without decreasing the recession connectivity.

Next we use our tropical tools to recover a regularity result of Herzog and Hibi \cite{herzog2020regularity}.  Recall that $\mat(B)$ denotes the \emph{matching number} of a graph $B$.

\begin{cor}\cite{herzog2020regularity}*{Theorem 1} \label{cor: matching}
 If $B$ is a connected bipartite graph then
\[\reg(\kk[B]) \leq \mat(B) - 1.\]
\end{cor}

\begin{proof}
This follows from Proposition \ref{prop: bounded+complex+dimension}, Theorem \ref{thm: tropicalCplxRegularity}, and Remark \ref{rem: lambdamatching}
%Recall that a cell $C_T$ is contained in $\tcone(A)$ if and only if $G_T$ has no disconnected nodes.
%Furthermore, the dimension of $C_T$ is the number of connected components of $G_T$ minus one.
%The maximum number of connected components $C_T$ may have without disconnecting any nodes is precisely the matching number of $G_A$, as we pick a maximal matching and then add a single edge to any disconnected nodes.
%This implies that $\dim(\tcone(A)) = \mat(G_A) - 1$.
%As the bounded complex is a subcomplex of $\tcone(A)$, this recovers the statement. 
\end{proof}

We next observe that our bound $\kk[B] = \lambda(B) - 1$ can be arbitrarily stronger than the matching number bound.  For this consider the following families of examples.

\begin{example} \label{ex: path+regularity}
Let $B$ be the path on $2n$ vertices.
As $B$ contains no cycles, the edge ring $\kk[B]$ is just a polynomial ring and has regularity zero.
This can be seen from the recession connectivity, as $\bddgraph{S}{B}$ is strongly connected if and only if $S = B$, and so trivially we have $\lambda(B)=1$.
However, $B$ has a maximal matching of size $n$, and hence the bound in Corollary~\ref{cor: graph+regularity} is as bad as possible.
\end{example}

\begin{example} \label{ex: cycle+chain+regularity}
Let $B$ be the 4-cycles $C_1, \dots, C_k$ with an edge between $C_i$ and $C_{i+1}$ for each $1\leq i \leq k$, as shown in Figure~\ref{fig: cycle+chain+regularity}.
Taking two disjoint edges in each $C_i$ gives a maximal matching of size $2k$, however the recession connectivity of $B$ is $k+1$.
This is as any subgraph $S$ such that $\bddgraph{S}{B}$ is strongly connected must contain all edges between cycles; Figure~\ref{fig: cycle+chain+regularity} gives an example of such a subgraph.

We note that adding a single edge to a graph can have a huge impact on the regularity of the edge ring.
Consider the bipartite graph $B'$ obtained by adding a single edge to $B$ between the cycles $C_1$ and $C_k$, as depicted in Figure~\ref{fig: cycle+chain+regularity}.
$B'$ has the same matching number as $B$, but setting $S$ to be a maximal matching now ensures $\bddgraph{S}{B'}$ is strongly connected, and the bound in Corollary~\ref{cor: graph+regularity} is an equality in this example.
To see this note that $S$ no longer needs to contain the edges between cycles, as one can now walk directly from $C_k$ to $C_1$ without having to backtrack through earlier cycles.
\end{example}

\begin{figure}
    \centering
    \includegraphics{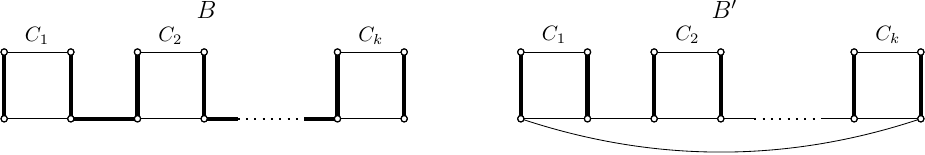}
    \caption{The bipartite graphs $B$ and $B'$ discussed in Example~\ref{ex: cycle+chain+regularity}. Both graphs have matching number $2k$, but very different recession connectivity: $\lambda(B) = k+1$ while $\lambda(B') = 2k$. Examples of subgraphs that realize the recession connectivity are highlighted in bold.}
    \label{fig: cycle+chain+regularity}
\end{figure}

% \begin{example}
% Consider the tropical matrix
% \[
% A = \begin{bmatrix}
% \end{bmatrix}
% \]
% \end{example}

\begin{comment}

\anton{When do we have equality?}
I guess this will be the case if $B$ is unmixed? A graph $G$ is K\"onig if its matching number coincides with its vertex cover number.  K\"onig's theorem say that all bipartite graphs are K\"onig. A graph $G$ is well-covered if all maximal independent sets of $G$ have the same cardinality, and $G$ is very well-covered if it is well-covered and $|V(G)| = 2\tau(G)$, where $\tau(G)$ is the size of a minimum vertex cover of $G$. A graph $G$ is unmixed if all minimum vertex covers have the same size.  A graph $G$ is very well-covered if and only if its unmixed K\"onig.  Hence a bipartite graph is very-well covered if and only if it is unmixed. Villareal \anton{reference} provides a nice characterization of unmixed bipartite graph.

\begin{theorem}
Let $B$ be a bipartite graph without isolated vertices. Then $B$ is unmixed if and only if there is a bipartition $V_1 = \{x_1,\dots,x_g\}$, $V_2 = \{y_1,\dots,y_g\}$ of $B$ such that
\begin{itemize}
    \item  $\{x_i,y_i\} \in E(B)$ for all $i$;
    \item  if $\{x_i,y_j\}$ and $\{x_j,y_k\}$ are in $E(B)$ and $i,j,k$ are distinct, then $\{x_i,y_k\} \in E(B)$.
    \end{itemize}

\end{theorem}

\end{comment}

\subsection{Linear resolutions}
We next use our tools to investigate toric edge rings of bipartite graphs that have linear resolutions (see Definition \ref{def: linearRes}).  In \cite{ohsugihibi1999}*{Theorem 4.6} Ohsugi and Hibi prove that a toric edge ring $\kk[G]$ has a 2-linear resolution if and only if $G = K_{2,n}$, in which case $\kk[G]$ is isomorphic to the polynomial ring over a Segre product of $\kk[x_1,x_2]$ and $\kk[x_1, \dots, x_n]$.  In \cite{hibi3linear2019} Hibi, Matsuda, and Tsuchiya prove that if $\kk[G]$ is the edge ring of a finite connected simple graph with a $3$-linear resolution then $\kk[G]$ is a hypersurface. Moreover they conjecture that the same is true for any $\kk[G]$ with a $q$-linear resolution.  In \cite{tsuchiya2021} Tsuchiya establishes this conjecture for the case of finite connected bipartite graphs, using combinatorial and algebraic methods. Very recently Mori, Ohsugi, and Tsuchiya \cite{MorOhsTsu} have established the conjecture in full generality.

\begin{example}
Examples of toric edge rings with $q$-linear resolutions are given by the case of $B = C_{2q}$, cycles of length $2q$.  In this case the toric edge ideal has a single generator (of degree $q$) given by the binomial corresponding to the cycle itself.  The corresponding tropical hyperplane arrangement consists of $q$ tropical hyperplanes, each of which has only two finite coordinates (an example of a \emph{graphical tropical hyperlane arrangement} from Section \ref{sec: dimension}). In this case the bounded complex $\cB(A)$ is a simplex of dimension $q-1$, and hence by Theorem \ref{thm: tropicalCplxRegularity} we have $\reg(R/J_{C_{2q}}) = q-1$.
\end{example}

We are able to use our machinery to provide an alternative proof of these results using the geometry of tropical complexes.  Our first observation relates the $q$-linear resolution property of toric edge rings to the geometry of tropical complexes.

\begin{lemma}\label{lem: twocycles}
Let $B$ be a connected finite simple bipartite graph with no cycles of length $<2q$, and at least two cycles of length $2q$ where $q \geq 3$.
Then $\reg(\kk[B]) \geq q$.
\end{lemma}
\begin{proof}
Let $C, C'$ be the two $2q$-cycles, the proof will be split into cases depending on the structure of $C,C'$.
In each case, we shall consider a minimal connected subgraph $B' \subseteq B$ such that $C \cup C' \subseteq B$.
We then build a subgraph $A \subseteq B'$ with at least $q+1$ connected components such that $\bddgraph{A}{B'}$ is strongly connected, showing that $\lambda (B') \geq q+1$ for $q \geq 3$.
As $\lambda(B) \geq \lambda(B')$ by Lemma~\ref{lem: lambda+monotone}, combining with Corollary~\ref{cor: graph+regularity} gives the desired result that $\reg(\kk[B]) \geq q$.

We first distinguish the cases of interest.
Let $\widehat{V} = V(C) \cap V(C')$ be the set of shared vertices and $\widehat{E} = E(C) \cap E(C')$ the set of shared edges.
We claim we have one of the following three cases:
\begin{enumerate}
    \item $|\widehat{V}| \leq 1\, , \,  \widehat{E} = \emptyset$,
    \item $\widehat{V} = \{v, v'\} \, , \, \widehat{E} = \emptyset$ where $v, v'$ are distance $q$ apart,
    \item $\widehat{E}$ is a path of length $|\widehat{V}| - 1 \leq q$.
\end{enumerate}

The proof of this claim comes from the following case analysis.
If $|\widehat{V}| \leq 1$, clearly there are no shared edges.

Suppose that $|\widehat{V}| > 1$, and that there exist $v, v' \in \widehat{V}$ that are not connected by edges in $\widehat{E}$.
Let $P$ and $P'$ be the shortest paths between $v$ and $v'$ in $C$ and $C'$ respectively; the union $P \cup P'$ contains at most $2q$ edges and must contain a cycle as $P \neq P'$.
As $B$ contains no cycles less than length $2q$, this can only happen if $P, P'$ are disjoint and exactly length $q$.
Furthermore, $\widehat{V} = \{v, v'\}$ as all other vertices of $C \cup C'$ are less than distance $q$ from either $v$ or $v'$.

Suppose all vertices $\widehat{V}$ are connected by edges $\widehat{E}$.
As $\widehat{E}$ is given by the intersection of cycles, it must be a path.
Furthermore, if it is of length $> q$, the edges $C \cup C' \setminus \widehat{E}$ form a cycle of length $\leq 2q$, a contradiction.

We now complete the proof by constructing $A \subseteq B' \subseteq B$ such that $\bddgraph{A}{B'}$ is strongly connected and $c(A) \geq q+1$ for each of the cases.
An example for the construction in each case is given in Figure~\ref{fig: cycle_cases}.

\paragraph{\bf Case 1:}
Let $B' \subseteq B$ be the union of $C, C'$ along with a minimal path between them - such a path exists as $B$ is connected.
We let $A$ be the subgraph of $B'$ whose edges form a matching in each cycle, along with all edges in the minimal path between the cycles.
The recession graph $\bddgraph{A}{B'}$ is strongly connected: adding a matching from a cycle to $A$ allows one to move between any vertices in the cycle, and adding the minimal path allows one to move between cycles.
Furthermore, $A$ is a forest with precisely $2q-1$ connected components: $q$ disconnected edges in each cycle, minus one as the path between cycles connects an edge from each cycle.
As a result, $\lambda(B') \geq 2q -1 \geq q+1$ when $q \geq 2$. 

\paragraph{\bf Case 2:}
Let $B' \subseteq B$ be the union $C \cup C'$ and let $A$ consist of a matching from each cycle.
As each cycle is covered by a maximal matching, the recession graph $\bddgraph{A}{B'}$ is strongly connected.
Furthermore, $A$ consists of $2q-2$ connected components, coming from the $2q$ edges that are disconnected everywhere, except at $v,v'$ where the two matchings share a vertex.
As a result, $\lambda(B') \geq 2q -2 \geq q+1$ when $q \geq 3$. 

\paragraph{\bf Case 3:}
Let $B' \subseteq B$ be the union $C \cup C'$.
Let $A$ be the union of maximal matchings in $C,C'$ that agree on $\widehat{E}$, and the matching on $\widehat{E}$ is also maximal.
As each cycle is covered by a maximal matching, the recession graph $\bddgraph{A}{B'}$ is strongly connected.
The number of edges of $A$ is
\[
2q - \left\lceil\frac{|\widehat{E}|}{2}\right\rceil \, .
\]
The only vertices where multiple edges may meet are the endpoints of $\widehat{E}$.
If $|\widehat{E}|$ is odd, the maximal matching on $\widehat{E}$ covers its endpoints, and so all edges of $A$ are disconnected.
If $|\widehat{E}|$ is even, one endpoint of $\widehat{E}$ is not covered by the maximal matching. Hence there are two edges meeting this endpoint, giving one connected component with two edges.
This means the number of connected components of $A$ is
\[
c(A) = \begin{cases} 2q - \left\lceil\frac{|\widehat{E}|}{2}\right\rceil & \widehat{E} \text{ odd} \\ 2q - \left\lceil\frac{|\widehat{E}|}{2}\right\rceil - 1 & \widehat{E} \text{ even} \end{cases}
\]
For $|\widehat{E}| \leq q$ and $q \geq 3$, this is greater than or equal to $q+1$.
This can be checked for small values of $q$, after which the $2q$ term becomes dominant.
\end{proof}
\begin{figure}
    \centering
    \includegraphics[width=\textwidth]{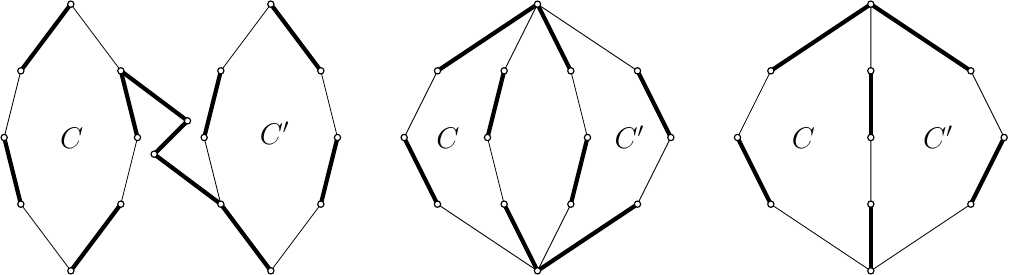}
    \caption{The graph $B'$ depicted each of the three cases.
    The collection of bold edges give the bidirected subgraph $A$.}
    \label{fig: cycle_cases}
\end{figure}

\begin{theorem} \label{thm: linear}
Suppose $B \subset K_{d,n}$ is a finite connected bipartite graph. Then we have
\begin{itemize}
    \item $\kk[B]$ has a $2$-linear resolution if and only if $B$ is obtained by appending (possibly no) trees to the vertices of $K_{2,n}$.  
    
    \item If $q \geq 3$ and  $\kk[B]$ has a $q$-linear resolution then $\kk[B]$ is a hypersurface.
    \end{itemize}
\end{theorem}

\begin{proof}
For the first part assume that $B$ is a bipartite graph as described. Note that from Proposition \ref{prop: villarreal} we have $\kk[B] \cong \kk[K_{2,n}]$, so we can in fact assume that $B = K_{2,n}$. Now, to obtain a tropical hyperplane arrangement ${\mathcal H}(A)$ satisfying $B = B_A$ we can take a collection of $d$ $0$-dimensional hyperplanes (points) on the real line $\troptorus{2} \cong {\mathbb R}$.  Then the bounded complex $\cB(A)$ has dimension 1, which implies $\kk[B]$ has a $2$-linear resolution.

Next suppose that $B \subset K_{d,n}$ is a connected bipartite graph such that $\kk[B]$ has a $2$-linear resolution.  
Let $B^\prime$ denote the graph obtained from $B$ by removing all cut edges (or, equivalently, all edges not contained in any cycle). Because $\kk[B] = \kk[B']$, it suffices to show that $B^\prime = K_{2,n}$. 
% cut edges (so that in particular $B^\prime$ has no vertices of degree 1).  We will show that $B^\prime = K_{2,n}$, which suffices to prove the claim.
%Further suppose that  ${\mathcal H}(A)$ is a sufficiently generic tropical hyperplane arrangement corresponding to a regular triangulation of $Q_B$.  Then from the above discussion we have that $B$ is bipartite chordal (has no induced cycles of length longer than $4$) and that the bounded subcomplex $\cB(A)$ satisfies $\dim(\cB(A)) = 1$.
Since $J_B$ must be generated by quadratics, we know that $B^\prime$ is chordal bipartite with a nonempty edge set. Notice that $B^\prime$ must also be connected, since otherwise the resolution of $k[B']$ would be the tensor product of the resolutions of $k[B_1]$ and $k[B_2]$ and therefore could not be linear.
% Note that $B^\prime$ cannot be obtained from $K_{2,n}$ by removing edges since this would create cut edges. 

Assume, seeking contradiction, that $B^\prime$ has at least 3 vertices in each vertex class. Since $B'$ is a nonempty chordal bipartite graph with no cut edges, it must contain a $4$-cycle $C$. Since $B$ is connected, let $e$ be an edge incident to $C$.  Then $e$ itself is contained in another $4$-cycle $C^\prime$ (since again $B$ is chordal bipartite and $B^\prime$ has no cut edges).  This leaves us with three cases, since $C$ and $C^\prime$ can intersect in no edges, a single edge, or two edges. If $C$ and $C^\prime$ intersect in a single edge, then there is a $6$ cycle on $6$ vertices contained in $B^\prime$, so $\lambda(B')$ is at least $3$ (by monotonicity of regularity, Theorem \ref{thm: subgraph}) and therefore $\reg(\kk[B^\prime])$ is at least $2$ by Corollary \ref{cor: graph+regularity}, giving a contradiction. Similarly, if $C$ and $C'$ share no edges, then there is an $8$ cycle in $B'$ on $7$ vertices and a similar argument shows that $\lambda(B')$ is at least $3$ and the regularity is too high. 
The only case remaining is if for any edge $e$ incident to $C$, the only cycles $C'$ containing $e$ share two edges with $C$. In this case $B'$ is exactly $K_{2,n}$, proving the claim.

% Then $C$ and $C^\prime$ can intersect in either 2 edges, 1 edge, or no edges.  In each case this implies we have a $6$-cycle as a subgraph of $B$.  Now the subgraph monotonicity property to conclude that regularity is too big.

For the second part we assume $q \geq 3$ and that $B \subset K_{d,n}$ is a connected bipartite graph such $\kk[B]$ has a $q$-linear resolution.  To prove the claim, it suffices to show that if $B$ has no even cycles of length $< 2q$ and has at least two even cycles of length $2q$ then $\reg(\kk[B]) \geq q$.  This follows from Lemma \ref{lem: twocycles}.
\end{proof}

Recall that $J_G$ has a $q$-linear resolution if and only if a square-free initial ideal $\inn(J_G)$ has this property, which is the case exactly when the $(\inn(J_G))^\vee$ (in this case corresponding to the underlying fine cotype ideal) is Cohen-Macaulay.  Hence we have a characterization of which tropical hyperplane arrangements give rise to Cohen-Macaulay cotype ideals.  Note that since type ideals have linear resolutions (see Corollary \ref{cor: linearRes}) this implies that such ideals are in fact \emph{bi-Cohen-Macaulay}.

\section{Further comments and questions}\label{sec: GP}

In this section we address some open questions and directions for future research.  As we have seen, the connection between subdivisions of root polytopes (and the associated theory of tropical hyperplane arrangements and generalized permutohedra) and the study of toric edge ideals have lead to new insights in both contexts.  
We expect that further applications of these ideas can be pursued, and in this section we collect some potential directions.
% Some of these ideas were discussed in previous sections but we collect them here for convenience.

% \subsection{Other graphs}
% In this paper we utilized methods from tropical convexity to study homological aspects of toric edge rings of bipartite graphs. These ideas relied on the fact that the root polytope and edge polytope coincide in this case.  A natural question is whether these techniques can be applied in a more general setting.  
% In \cite{postnikov2009permutohedra} Postnikov extends the theory of root polytopes to the case of transitive graphs. 
% It is not clear how these root polytopes relate to the corresponding edge polytope.

\subsection{Generalized permutohedra}

Recall from Definition \ref{def: generalizedPermutohedron} that a bipartite graph $B$ gives rise to a generalized permutohedron $P_B$ obtained as a sum of simplices. 
In Corollary \ref{cor: cellResGenPerm} we proved that any regular fine mixed subdivision of $P_B$ supports a minimal cellular resolution of the monomial ideal generated by the lattice points of $P_B$.  
% To accomplish this, for every lattice point ${\bf p} = (p_1, \dots, p_n) \in {\mathbb R}^n$ we construct a monomial $m^{\bf p} = x_1^{p_1} \cdots x_n^{p_n}$ whose exponent vector is given by ${\bf p}$.  
% In this context the ideal generated by lattice points is recovered as the coarse type ideal $I_{\bft(\cH)}$ of the underlying tropical hyperplane arrangement (which does \textit{not} depend on the choice of generic arrangement).
% As discussed above, the polytopes $P_B$ are examples of a more general construction due to Postnikov which we describe here.  Recall that the \emph{permutohedron} ${\mathcal P}_n$ is the $(n-1)$-dimensional polytope obtained as the convex hull of the $S_n$-orbit of the vector $(1,2,\dots,n)$ in ${\mathbb R}^n$. A \emph{generalized permutohedra} \cite{postnikov2009permutohedra} is a polytope obtained from ${\mathcal P}_n$ by translating facets in parallel directions without creating new edge directions. Equivalently they are characterized by 
% \[\left\{ {\bf x} \in {\mathbb R}^n: \sum_{i \in I} x_i \leq f(I), \sum_{i \in [n]} x_i = f([n]) \right\},\]
% \noindent
% where $f$ is a \emph{submodular function}.
% Examples of generalized permutohedra include the sums of simplices $P_B = \sum_k \Delta_{I_k}$ associated to a bipartite graph $B$.  
However, not all generalized permutohedra arise as sums of simplices; instead, it is known that any generalized permutohedron can be realized as a Minkowski sum and \emph{difference} of simplices (see for instance \cite{ArdBenDok}).  
It is natural to wonder whether one can extend the algebraic theory we developed here and describe minimal resolutions of the monomial ideal generated by lattice points of generalized permutohedra. A first step in this direction would be to understand our constructions for the case of Minkowski differences. This involves developing a good notion of mixed subdivision in this setting, which is not something we have seen in the literature.

One particularly well-studied class of generalized permutohedra is the class of \emph{matroid} or \textit{matroid basis} polytopes, which are generalized permutohedra whose vertices are $\{0,1\}$-vectors. 
% The vertices (and hence lattice points) of a matroid basis polytope $M$ are generators of a monomial ideal $I_M$ known as a \emph{matroidal ideal}. 
% By assigning the faces of $M$ monomial labels given by LCMs of the monomial labels of its vertices, it is not hard to see that 
The complex $M$ supports a (non-minimal) cellular resolution of a monomial ideal $I_M$ known as a \emph{matroidal ideal}.
Furthermore, one can check that the Betti numbers of the ideal $I_M$ are M\"obius invariants of certain deletions of $M$ [Alex Fink, private communication]. A natural question that arises is the following.

%For a polytope ${\mathcal P}$ in ${\mathbb R}^n$ we let $I_{\mathcal P} \subset \kk[x_1, \dots, x_n]$ denote the monomial ideal generated by the lattice points ${\mathcal P} \cap {\mathbb Z}^n$, and $I_{A({\mathcal P})}$ denote the ideal generated by its vertices.  Here a lattice point is taken as the exponent vector of the associated monomial.  Note that our coarse type ideals are examples of the former construction for the case of a sum of simplices. There is a general criteria on ${\mathcal P}$ that guarantees that ${\mathcal P}$ will support a resolution of $I_{A({\mathcal P})}$, and one can check that this criteria is satisfied for matroid polytopes.  Note that in this case all lattice points of $M$ are vertices.

\begin{question}
For a matroid $M$ can one describe a minimal cellular resolution of the matroidal ideal $I_M$? What about for certain classes of matroids?
\end{question}

%I think Alex had a way to do this for series-parallel matroids.

\noindent One of the simplest examples of a Minkowski difference is the case of a \emph{trimmed permutohedron},
\[P_B^- =P_B - \Delta_{[d]},\]
\noindent
where $P_B= \sum_k \Delta_{I_k}$ is the generalized permutohedron associated to some bipartite graph $B \subset K_{d,n}$ and where for all $k$ we have $\Delta_{I_k} \subset {\mathbb R}^n$. One can check \cite{postnikov2009permutohedra}*{Theorem 12.9} that for any triangulation $\{\Delta_{T_1}, \dots, \Delta_{T_s}\}$ of the root polytope $Q_B$, the set of shifted left degree vectors $\{\LD(T_1)-1, \dots, \LD(T_s)-1\}$ is exactly the set of  lattice points of $P_B^-$.

\begin{question}
Does a (subdivision of a) trimmed permutohedron $P_B^-$ support a resolution of any naturally occurring ideal?
\end{question}

%Note that if ${\mathcal M} = {\mathcal M}(G)$ is a graphic matroid then it can be decribed as 
%\[{\mathcal M} = {\mathcal Z}_G^* - \Delta_E.\]
%Here ${\mathcal Z}_G^*$ is the `dual generalized permutohedron' associated to the graph $G$,
%\[{\mathcal Z}_G^* = \sum_{v \in V(G)} \Delta_{\sigma_v},\]
%\noindent
%where $\Delta_\sigma = \sum_{i \in \sigma} e_i$ denotes the simplex on vertex set $\sigma$, and $\sigma_v$ denotes the sets of edges incident to $v$.  This is the generalized permutohedra dual to the graphical zonotope in the Postnikov-Kalaman sense.

%In what follows we say that a generalized permutohedron is \emph{lattice} if all of its lattice points are vertices.   

%\begin{theorem}
%Suppose ${\mathcal P}$ is a lattice generalized permutohedron. Then with the monomial labeling determined by the lattice points, ${\mathcal P}$ supports a cellular resolution of the ideal $I_{\mathcal P}$.
%\end{theorem}

\subsection{\texorpdfstring{$h$}{h}-vectors}
Stanley famously conjectured that the $h$-vector of a graded Cohen-Macaulay integral domain is log-concave in \cite{stanley1989log}. It is known that Stanley's conjecture is not true in general, but it is still open for special cases. For instance, Conca and Herzog conjectured that the $h$-vector of a ladder determinantal ring is always log concave and Rubey proved this in the special case of $2 \times 2$ ladders \cite{rubey2005h}.  As far as the authors are aware, the question is still open for the case of larger minors and for toric edge ideals of bipartite graphs.

In our situation, we know that the following polynomials coincide:
\begin{itemize}
    \item the $h$-polynomial of the toric edge ideal $J_B$ of a bipartite graph $B$.
    \item the $h$-polynomial of any \textit{initial} ideal $\inn_A J_B$ with respect to any weight matrix $A$ (such as a tropical matrix). 
    % The fact that the Hilbert series of an ideal and any of its initial ideals coincide is an old theorem due to Macaulay, see \cite[Theorem 15.3]{eisenbud2013commutative}.
    \item the $h$-polynomial of the Stanley-Reisner complex of any $\inn_A J_B$, which is a shellable unimodular triangulation of the root polytope $Q_B$.
    \item the Ehrhart polynoimal of the root polytope $Q_B$.
    \item the interior polynomial of either hypergraph of $Q_B$ (see \cite{KalPos}).    
\end{itemize}

Our hope is that our methods can provide some headway in proving Stanley's conjecture in the case of ladder determinantal ideals, or perhaps even for toric edge ideals of all bipartite graphs.

\subsection{Polarizations}
An arrangement ${\mathcal H} = {\mathcal H}(A)$ of $n$ tropical hyperplanes in $\troptorus{d}$ gives rise to both a fine type ideal $I_{T(\cH)} \subset \widetilde S = \kk[x_{11}, \dots, x_{dn}]$ as well as a coarse type ideal $I_{\bft(\cH)} \subset S = \kk[x_1, \dots, x_d]$.  There is a natural map $I_{T(\cH)} \rightarrow I_{\bft(\cH)}$ given by `forgetting the label' on each hyperplane. This map is a flat deformation and is an example of \emph{polarization} of monomial ideals. For example, in the case of hyperplanes with full support, we obtain a family of polarizations of the ideal $\langle x_1, \dots, x_n \rangle ^d$, studied for instance in \cite{almousa2019polarizations}.

% A natural question is whether all polarizations of $\langle x_1, \dots, x_n \rangle^d$ are obtained in this way, but it is not hard to see that even for the case of $\langle x_1, x_2, x_3\rangle^3$ there exist polarizations that do not arise from tropical hyperplane arrangements.  
On the other hand, 
an axiomatic approach to triangulations of products of simplices (and more generally of root polytopes) was initiated by Galashin, Nenashev, and Postnikov in \cite{galashin2018trianguloids}. Here the notion of a \emph{trianguloid} is defined as an edge colored graph satisfying simple local axioms, and in \cite{galashin2018trianguloids} it is shown that trianguloids are in bijection with triangulations of root polytopes. One can check that requiring a subset of these axioms gives rise to a polarization of the underyling ideal. The connection between trianguloids and polarizations first appeared in the first author's thesis \cite{almousa2021combinatorial} and is further explored by the authors in a follow up paper \cite{ADSpolarizations}.

\subsection{Graphic tropical hyperplane arrangements and \texorpdfstring{$G$}{G}-parking function ideals}

% For the case of \emph{graphic} tropical hyperplane arrangements (see Section \ref{def: graphic}) we have seen a connection to (usual) hyperplane arrangements.  
In Subsection \ref{subsec: graphic}, we saw a connection between graphic tropical hyperplane arrangements and classical hyperplane arrangements.
In particular, the tropical complex associated to such an arrangement can be recovered as the subdivision of ${\mathbb R}^{d-1}$ obtained from a perturbation of the underlying graphic arrangement.  As discussed in Section \ref{sec: dimension}, the fine type ideals of these arrangements can be recovered as the oriented matroid ideals of \cite{NPS}.

In \cite{DocSan} it is shown that an affine slice of a graphic hyperplane arrangement associated to $G$ defines a cellular resolution of the underlying \emph{$G$-parking function ideal} $M_G$.  Here $M_G$ is a monomial ideal depending on a graph $G$ and a choice of root vertex $v \in G$, and whose standard monomial ideals corresponding to $G$-parking functions. These ideals were introduced in \cite{postnikovshapiro} and have connections to power ideals and chip-firing on graphs.

One can see that $M_G$ can be obtained as a certain coarse type ideal associated to a graphic tropical hyperplane.  An application of tropical convexity recovers the main result of \cite{DocSan}, and also places the theory in a broader context where, for instance, a choice of root vertex is not necessary.  In particular, the tropical theory allows one to define parking functions for arbitrary root polytopes.  We explore these ideas in \cite{ADSparking}.

\section*{Acknowledgments}
We would like to thank Laura Ballard, Michael Joswig, Suho Oh, Tina O'Keefe, Vic Reiner, Raman Sanyal, and Adam Van Tuyl for helpful conversations. We also thank Akiyoshi Tsuchiya for providing us a positive answer to Question \ref{question: regularity-subgraphs} and Kieran Bhaskara for pointing out an error in earlier version of the statement of Theorem~\ref{thm: regbound}. We are grateful to an anonymous referee for helpful comments. The first author was partially supported by the NSF GRFP under Grant No. DGE-1650441.
The second author was partially supported by Simons Foundation Grant $\#964659$.
The third author was supported by the Heilbronn Institute for Mathematical Research.

\bibliographystyle{amsplain}
\bibliography{biblio}
\addcontentsline{toc}{section}{Bibliography}

\end{document}

%% file: macro.tex
\newtheorem{lemma}{Lemma}[section]
\newtheorem{theorem}[lemma]{Theorem}

\newtheorem{prop}[lemma]{Proposition}
\newtheorem{cor}[lemma]{Corollary}

\newtheorem{question}[lemma]{Question}

\newtheorem{thm}{Theorem}
\renewcommand*{\thethm}{\Alph{thm}}

\theoremstyle{remark}
\newtheorem{remark}[lemma]{Remark}

\theoremstyle{definition}
\newtheorem{definition}[lemma]{Definition}
\newtheorem{example}[lemma]{Example}

\newtheorem{notation}[lemma]{Notation}

\newcommand{\one}{\mathbb{1}}

\newcommand{\depth}{\operatorname{depth}}
\newcommand{\reg}{\operatorname{reg}}

\newcommand{\inn}{\operatorname{in}}

\newcommand{\pdim}{\operatorname{pdim}}

\newcommand{\Supp}{\operatorname{Supp}}

\newcommand{\conv}{\operatorname{conv}}

\newcommand{\rank}{\operatorname{rank}}

\newcommand{\supp}{\operatorname{supp}}

\newcommand{\tconv}{\operatorname{tconv}}
\newcommand{\mat}{\operatorname{mat}}

\newcommand{\kk}{\mathbb k}  % NOTE: Changed from bbm to bb as there were two different notations being used
\renewcommand{\aa}{\mathbf a}

\newcommand{\ee}{\mathbf e}
\newcommand{\vv}{\mathbf v}

\newcommand{\pp}{\mathbf p}
\newcommand{\qq}{\mathbf q}

\newcommand{\uu}{\mathbf u}
\renewcommand{\vv}{\mathbf v}
\newcommand{\ww}{\mathbf w}
\newcommand{\xx}{\mathbf x}

\newcommand{\bft}{\mathbf t}

\newcommand{\cA}{\mathcal{A}}
\newcommand{\cB}{\mathcal{B}}
\newcommand{\cC}{\mathcal{C}}

\newcommand{\cG}{\mathcal{G}}
\newcommand{\cH}{\mathcal{H}}

\newcommand{\cM}{\mathcal{M}}

\newcommand{\cP}{\mathcal{P}}

\newcommand{\cS}{\mathcal{S}}
\newcommand{\cT}{\mathcal{T}}
\newcommand{\NN}{\mathbb{N}}
\newcommand{\PP}{\mathbb{P}}

\newcommand{\RR}{\mathbb{R}}

\newcommand{\TT}{\mathbb{T}}
\newcommand{\ZZ}{\mathbb{Z}}

\renewcommand{\P}{\mathcal{P}}
\renewcommand{\b}{\bf{b}}
\renewcommand{\a}{\bf{a}}
\newcommand{\res}{\mathcal{F}}

\newcommand{\m}{\mathfrak{m}}

\newcommand{\Tmin}{\TT_{\min}} % min-tropical semiring
\newcommand{\Tmax}{\TT_{\max}} % max-tropical semiring
\newcommand{\troptorus}[1]{\RR^{#1}/\RR\one} % tropical torus
\newcommand{\bddgraph}[2]{R(#1;#2)}

\DeclareMathOperator{\RD}{RD}
\DeclareMathOperator{\LD}{LD}
\DeclareMathOperator{\vol}{Vol}
\DeclareMathOperator{\rec}{rec}

\newcommand{\anton}[1]{{\color{blue} \sf ANTON: [#1]}}
\newcommand{\ben}[1]{{\color{red} \sf BEN: [#1]}}

%% file: main.bbl
\begin{bibdiv}
\begin{biblist}

\bib{ArdBenDok}{article}{
      author={Ardila, Federico},
      author={Benedetti, Carolina},
      author={Doker, Jeffrey},
       title={Matroid polytopes and their volumes},
        date={2010},
        ISSN={0179-5376},
     journal={Discrete Comput. Geom.},
      volume={43},
      number={4},
       pages={841\ndash 854},
         url={https://doi-org.libproxy.txstate.edu/10.1007/s00454-009-9232-9},
      review={\MR{2610473}},
}

\bib{ADSparking}{article}{
      author={Almousa, Ayah},
      author={Dochtermann, Anton},
      author={Smith, Ben},
       title={Tropical types, parking function ideals, and root polytopes},
        date={2024},
        note={work in progress},
}

\bib{ADSpolarizations}{article}{
      author={Almousa, Ayah},
      author={Dochtermann, Anton},
      author={Smith, Ben},
       title={Tropical types, trianguloids, and polarizations},
        date={2024},
        note={work in progress},
}

\bib{almousa2019polarizations}{article}{
      author={Almousa, Ayah},
      author={Fl{\o}ystad, Gunnar},
      author={Lohne, Henning},
       title={Polarizations of powers of graded maximal ideals},
        date={2022},
     journal={Journal of Pure and Applied Algebra},
      volume={226},
      number={5},
       pages={106924},
}

\bib{almousa2021combinatorial}{thesis}{
      author={Almousa, Ayah},
       title={Combinatorial characterizations of polarizations of powers of the
  graded maximal ideal},
        type={Ph.D. Thesis},
        date={2021},
}

\bib{ballard2021properties}{thesis}{
      author={Ballard, Laura~E},
       title={Properties of the toric rings of a chordal bipartite family of
  graphs},
        type={Ph.D. Thesis},
        date={2020},
}

\bib{BieOkeVan17}{article}{
      author={Biermann, Jennifer},
      author={O'Keefe, Augustine},
      author={Van~Tuyl, Adam},
       title={Bounds on the regularity of toric ideals of graphs},
        date={2017},
     journal={Adv. in Appl. Math.},
      volume={85},
       pages={84\ndash 102},
}

\bib{butkovic}{book}{
      author={Butkovi\v{c}, Peter},
       title={Max-linear systems: theory and algorithms},
      series={Springer Monographs in Mathematics},
   publisher={Springer-Verlag London, Ltd., London},
        date={2010},
        ISBN={978-1-84996-298-8},
         url={https://doi.org/10.1007/978-1-84996-299-5},
      review={\MR{2681232}},
}

\bib{BenVar}{article}{
      author={Benedetti, Bruno},
      author={Varbaro, Matteo},
       title={On the dual graphs of {C}ohen-{M}acaulay algebras},
        date={2015},
        ISSN={1073-7928},
     journal={Int. Math. Res. Not. IMRN},
      number={17},
       pages={8085\ndash 8115},
         url={https://doi-org.libproxy.txstate.edu/10.1093/imrn/rnu191},
      review={\MR{3404010}},
}

\bib{block2006tropical}{article}{
      author={Block, Florian},
      author={Yu, Josephine},
       title={Tropical convexity via cellular resolutions},
        date={2006},
     journal={Journal of Algebraic Combinatorics},
      volume={24},
      number={1},
       pages={103\ndash 114},
}

\bib{cohen+gaubert+quadrat}{incollection}{
      author={Cohen, Guy},
      author={Gaubert, St\'{e}phane},
      author={Quadrat, Jean-Pierre},
       title={Duality and separation theorems in idempotent semimodules},
        date={2004},
      volume={379},
       pages={395\ndash 422},
         url={https://doi.org/10.1016/j.laa.2003.08.010},
        note={Tenth Conference of the International Linear Algebra Society},
      review={\MR{2039751}},
}

\bib{corso2008specializations}{article}{
      author={Corso, Alberto},
      author={Nagel, Uwe},
       title={Specializations of {F}errers ideals},
        date={2008},
     journal={Journal of Algebraic Combinatorics},
      volume={28},
      number={3},
       pages={425\ndash 437},
}

\bib{corso2009}{article}{
      author={Corso, Alberto},
      author={Nagel, Uwe},
       title={Monomial and toric ideals associated to {F}errers graphs},
        date={2009},
        ISSN={0002-9947},
     journal={Trans. Amer. Math. Soc.},
      volume={361},
      number={3},
       pages={1371\ndash 1395},
  url={https://doi-org.libproxy.txstate.edu/10.1090/S0002-9947-08-04636-9},
      review={\MR{2457403}},
}

\bib{conca1995}{article}{
      author={Conca, Aldo},
       title={Ladder determinantal rings},
        date={1995},
        ISSN={0022-4049},
     journal={J. Pure Appl. Algebra},
      volume={98},
      number={2},
       pages={119\ndash 134},
  url={https://doi-org.libproxy.txstate.edu/10.1016/0022-4049(94)00039-L},
      review={\MR{1319965}},
}

\bib{conca1996gorenstein}{article}{
      author={Conca, Aldo},
       title={Gorenstein ladder determinantal rings},
        date={1996},
     journal={Journal of the London Mathematical Society},
      volume={54},
      number={3},
       pages={453\ndash 474},
}

\bib{conca2020square}{article}{
      author={Conca, Aldo},
      author={Varbaro, Matteo},
       title={Square-free {G}r{\"o}bner degenerations.},
        date={2020},
     journal={Inventiones Mathematicae},
      volume={221},
      number={3},
}

\bib{dochtermann2012tropical}{article}{
      author={Dochtermann, Anton},
      author={Joswig, Michael},
      author={Sanyal, Raman},
       title={Tropical types and associated cellular resolutions},
        date={2012},
     journal={Journal of Algebra},
      volume={356},
      number={1},
       pages={304\ndash 324},
}

\bib{develin2004tropical}{article}{
      author={Develin, Mike},
      author={Sturmfels, Bernd},
       title={Tropical convexity},
        date={2004},
     journal={Documenta Mathematica},
      volume={9},
       pages={1\ndash 27},
}

\bib{DocSan}{article}{
      author={Dochtermann, Anton},
      author={Sanyal, Raman},
       title={Laplacian ideals, arrangements, and resolutions},
        date={2014},
        ISSN={0925-9899},
     journal={J. Algebraic Combin.},
      volume={40},
      number={3},
       pages={805\ndash 822},
         url={https://doi-org.libproxy.txstate.edu/10.1007/s10801-014-0508-7},
      review={\MR{3265234}},
}

\bib{eagon1998resolutions}{article}{
      author={Eagon, John~A},
      author={Reiner, Victor},
       title={Resolutions of {S}tanley-{R}eisner rings and {A}lexander
  duality},
        date={1998},
     journal={Journal of Pure and Applied Algebra},
      volume={130},
      number={3},
       pages={265\ndash 275},
}

\bib{fink2015stiefel}{article}{
      author={Fink, Alex},
      author={Rinc{\'o}n, Felipe},
       title={Stiefel tropical linear spaces},
        date={2015},
     journal={Journal of Combinatorial Theory, Series A},
      volume={135},
       pages={291\ndash 331},
}

\bib{polymake:2000}{incollection}{
      author={Gawrilow, Ewgenij},
      author={Joswig, Michael},
       title={{\texttt polymake}: a framework for analyzing convex polytopes},
        date={2000},
   booktitle={Polytopes---combinatorics and computation ({O}berwolfach, 1997)},
      series={DMV Sem.},
      volume={29},
   publisher={Birkh\"auser, Basel},
       pages={43\ndash 73},
      review={\MR{1785292}},
}

\bib{galashin2018trianguloids}{article}{
      author={Galashin, Pavel},
      author={Nenashev, Gleb},
      author={Postnikov, Alexander},
       title={Trianguloids and triangulations of root polytopes},
        date={2024},
        ISSN={0097-3165},
     journal={Journal of Combinatorial Theory, Series A},
      volume={201},
       pages={105802},
  url={https://www.sciencedirect.com/science/article/pii/S0097316523000705},
}

\bib{gorla2007mixed}{article}{
      author={Gorla, Elisa},
       title={Mixed ladder determinantal varieties from two-sided ladders},
        date={2007},
     journal={Journal of Pure and Applied Algebra},
      volume={211},
      number={2},
       pages={433\ndash 444},
}

\bib{ha2019}{article}{
      author={H\`a, Huy~T\`ai},
      author={Beyarslan, Selvi~Kara},
      author={O'Keefe, Augustine},
       title={Algebraic properties of toric rings of graphs},
        date={2019},
        ISSN={0092-7872},
     journal={Comm. Algebra},
      volume={47},
      number={1},
       pages={1\ndash 16},
  url={https://doi-org.libproxy.txstate.edu/10.1080/00927872.2018.1439047},
      review={\MR{3924764}},
}

\bib{herzog2020regularity}{article}{
      author={Herzog, J\"{u}rgen},
      author={Hibi, Takayuki},
       title={The regularity of edge rings and matching numbers},
        date={2020},
     journal={Mathematics},
      volume={8},
      number={1},
       pages={103\ndash 114},
}

\bib{HHOBinomial}{book}{
      author={Herzog, J\"{u}rgen},
      author={Hibi, Takayuki},
      author={Ohsugi, Hidefumi},
       title={Binomial ideals},
      series={Graduate Texts in Mathematics},
   publisher={Springer, Cham},
        date={2018},
      volume={279},
        ISBN={978-3-319-95347-2; 978-3-319-95349-6},
         url={https://doi-org.libproxy.txstate.edu/10.1007/978-3-319-95349-6},
      review={\MR{3838370}},
}

\bib{hibi3linear2019}{article}{
      author={Hibi, Takayuki},
      author={Matsuda, Kazunori},
      author={Tsuchiya, Akiyoshi},
       title={Edge rings with 3-linear resolutions},
        date={2019},
        ISSN={0002-9939},
     journal={Proc. Amer. Math. Soc.},
      volume={147},
      number={8},
       pages={3225\ndash 3232},
         url={https://doi-org.libproxy.txstate.edu/10.1090/proc/14382},
      review={\MR{3981103}},
}

\bib{hochster1972}{article}{
      author={Hochster, M.},
       title={Rings of invariants of tori, {C}ohen-{M}acaulay rings generated
  by monomials, and polytopes},
        date={1972},
        ISSN={0003-486X},
     journal={Ann. of Math. (2)},
      volume={96},
       pages={318\ndash 337},
         url={https://doi-org.libproxy.txstate.edu/10.2307/1970791},
      review={\MR{304376}},
}

\bib{huber2000cayley}{article}{
      author={Huber, Birkett},
      author={Rambau, J{\"o}rg},
      author={Santos, Francisco},
       title={The {C}ayley trick, lifting subdivisions and the {B}ohne-{D}ress
  theorem on zonotopal tilings},
        date={2000},
     journal={Journal of the European Mathematical Society},
      volume={2},
      number={2},
       pages={179\ndash 198},
}

\bib{JoswigLoho:2016}{article}{
      author={Joswig, Michael},
      author={Loho, Georg},
       title={Weighted digraphs and tropical cones},
        date={2016},
        ISSN={0024-3795},
     journal={Linear Algebra and its Applications},
      volume={501},
       pages={304\ndash 343},
}

\bib{JoswigBook22}{book}{
      author={Joswig, Michael},
       title={Essentials of tropical combinatorics},
      series={Graduate Studies in Mathematics},
   publisher={American Mathematical Society},
     address={Providence, RI},
        date={2022},
}

\bib{KalPos}{article}{
      author={K\'{a}lm\'{a}n, Tam\'{a}s},
      author={Postnikov, Alexander},
       title={Root polytopes, {T}utte polynomials, and a duality theorem for
  bipartite graphs},
        date={2017},
        ISSN={0024-6115},
     journal={Proc. Lond. Math. Soc. (3)},
      volume={114},
      number={3},
       pages={561\ndash 588},
         url={https://doi-org.libproxy.txstate.edu/10.1112/plms.12015},
      review={\MR{3653240}},
}

\bib{litvinov+maslov+shpiz}{article}{
      author={Litvinov, G.~L.},
      author={Maslov, V.~P.},
      author={Shpiz, G.~B.},
       title={Idempotent functional analysis. {A}n algebraic approach},
        date={2001},
        ISSN={0025-567X},
     journal={Mat. Zametki},
      volume={69},
      number={5},
       pages={758\ndash 797},
         url={https://doi.org/10.1023/A:1010266012029},
      review={\MR{1846814}},
}

\bib{MorOhsTsu}{article}{
      author={Mori, Kenta},
      author={Ohsugi, Hidefumi},
      author={Tsuchiya, Akiyoshi},
       title={Edge rings with $q$-linear resolutions},
        date={2022},
     journal={Journal of Algebra},
      volume={593},
       pages={550\ndash 567},
}

\bib{MillerSturmfels:2005}{book}{
      author={Miller, Ezra},
      author={Sturmfels, Bernd},
       title={Combinatoiral commutative algebra},
      series={Graduate Texts in Mathematics},
   publisher={Springer},
        date={2005},
      number={227},
}

\bib{NPS}{article}{
      author={Novik, Isabella},
      author={Postnikov, Alexander},
      author={Sturmfels, Bernd},
       title={Syzygies of oriented matroids},
        date={2002},
     journal={Duke Mathematical Journal},
      volume={111},
      number={2},
       pages={287\ndash 317},
}

\bib{ohsugihibikoszul1999}{article}{
      author={Ohsugi, Hidefumi},
      author={Hibi, Takayuki},
       title={Koszul bipartite graphs},
        date={1999},
        ISSN={0196-8858},
     journal={Adv. in Appl. Math.},
      volume={22},
      number={1},
       pages={25\ndash 28},
         url={https://doi-org.libproxy.txstate.edu/10.1006/aama.1998.0615},
      review={\MR{1657721}},
}

\bib{ohsugihibi1999}{article}{
      author={Ohsugi, Hidefumi},
      author={Hibi, Takayuki},
       title={Toric ideals generated by quadratic binomials},
        date={1999},
        ISSN={0021-8693},
     journal={J. Algebra},
      volume={218},
      number={2},
       pages={509\ndash 527},
         url={https://doi-org.libproxy.txstate.edu/10.1006/jabr.1999.7918},
      review={\MR{1705794}},
}

\bib{postnikov2009permutohedra}{article}{
      author={Postnikov, Alexander},
       title={Permutohedra, associahedra, and beyond},
        date={2009},
     journal={International Mathematics Research Notices},
      volume={2009},
      number={6},
       pages={1026\ndash 1106},
}

\bib{postnikovshapiro}{article}{
      author={Postnikov, Alexander},
      author={Shapiro, Boris},
       title={Trees, parking functions, syzygies, and deformations of monomial
  ideals},
        date={2004},
        ISSN={0002-9947},
     journal={Trans. Amer. Math. Soc.},
      volume={356},
      number={8},
       pages={3109\ndash 3142},
  url={https://doi-org.libproxy.txstate.edu/10.1090/S0002-9947-04-03547-0},
      review={\MR{2052943}},
}

\bib{rajchgot2022castelnuovo}{article}{
      author={Rajchgot, Jenna},
      author={Robichaux, Colleen},
      author={Weigandt, Anna},
       title={Castelnuovo-mumford regularity of ladder determinantal varieties
  and patches of grassmannian schubert varieties},
        date={2023},
        ISSN={0021-8693},
     journal={Journal of Algebra},
      volume={617},
       pages={160\ndash 191},
  url={https://www.sciencedirect.com/science/article/pii/S0021869322005063},
}

\bib{rubey2005h}{article}{
      author={Rubey, Martin},
       title={The h-vector of a ladder determinantal ring cogenerated by
  2$\times$2 minors is log-concave},
        date={2005},
     journal={Journal of Algebra},
      volume={292},
      number={2},
       pages={303\ndash 323},
}

\bib{San02}{book}{
      author={Santos, Francisco},
       title={Triangulations of oriented matroids},
   publisher={American Mathematical Soc.},
        date={2002},
}

\bib{santos2005cayley}{article}{
      author={Santos, Francisco},
       title={The {C}ayley trick and triangulations of products of simplices},
        date={2005},
     journal={Contemporary Mathematics},
      volume={374},
       pages={151\ndash 178},
}

\bib{stanley1989log}{article}{
      author={Stanley, Richard~P},
       title={Log-concave and unimodal sequences in algebra, combinatorics, and
  geometry},
        date={1989},
     journal={Ann. New York Acad. Sci},
      volume={576},
      number={1},
       pages={500\ndash 535},
}

\bib{stanley1993}{article}{
      author={Stanley, Richard~P.},
       title={A monotonicity property of $h$-vectors and $h^*$-vectors},
        date={1993},
     journal={European Journal of Combinatorics},
      volume={14},
       pages={251\ndash 258},
}

\bib{sturmfels1996grobner}{book}{
      author={Sturmfels, Bernd},
       title={Grobner bases and convex polytopes},
   publisher={American Mathematical Soc.},
        date={1996},
      volume={8},
}

\bib{terai1997generalization}{article}{
      author={Terai, Naoki},
       title={Generalization of {E}agon-{R}einer theorem and h-vectors of
  graded rings},
        date={1997},
        note={preprint},
}

\bib{tsuchiya2021}{article}{
      author={Tsuchiya, Akiyoshi},
       title={Edge rings of bipartite graphs with linear resolutions},
        date={2021},
        ISSN={0219-4988},
     journal={J. Algebra Appl.},
      volume={20},
      number={9},
       pages={Paper No. 2150163, 8},
         url={https://doi-org.libproxy.txstate.edu/10.1142/S0219498821501632},
      review={\MR{4301168}},
}

\bib{villarreal1995rees}{article}{
      author={Villarreal, Rafael~H},
       title={Rees algebras of edge ideals},
        date={1995},
     journal={Communications in Algebra},
      volume={23},
      number={9},
       pages={3513\ndash 3524},
}

\end{biblist}
\end{bibdiv}
